\documentclass[10pt,reqno]{amsart}
\usepackage[english]{babel}
\usepackage[a4paper,twoside]{geometry}
\usepackage{microtype}
\usepackage{csquotes}
\usepackage{bbm}
\usepackage{dsfont}									
\usepackage[inline,shortlabels]{enumitem}
\usepackage{capt-of}

\usepackage{scrextend}						

\usepackage{verbatim}
\usepackage{todonotes}

\usepackage{xspace}
\patchcmd{\subsection}{-.5em}{.5em}{}{}
\usepackage[hyperfootnotes=false]{hyperref} 
\usepackage{color}
\hypersetup{
	colorlinks = true,
	linkcolor = {blue},
	urlcolor = {red},
	citecolor = {blue}
}

\makeatletter
\renewcommand{\tocsection}[3]{
  \indentlabel{\@ifnotempty{#2}{\ignorespaces#1 #2\quad}}\bfseries#3}
\renewcommand{\tocsubsection}[3]{
  \indentlabel{\@ifnotempty{#2}{\ignorespaces#1 #2\quad}}#3}

\newcommand\@dotsep{4.5}
\def\@tocline#1#2#3#4#5#6#7{\relax
  \ifnum #1>\c@tocdepth
  \else
    \par \addpenalty\@secpenalty\addvspace{#2}
    \begingroup \hyphenpenalty\@M
    \@ifempty{#4}{
      \@tempdima\csname r@tocindent\number#1\endcsname\relax
    }{
      \@tempdima#4\relax
    }
    \parindent\z@ \leftskip#3\relax \advance\leftskip\@tempdima\relax
    \rightskip\@pnumwidth plus1em \parfillskip-\@pnumwidth
    #5\leavevmode\hskip-\@tempdima{#6}\nobreak
    \leaders\hbox{$\m@th\mkern \@dotsep mu\hbox{.}\mkern \@dotsep mu$}\hfill
    \nobreak
    \hbox to\@pnumwidth{\@tocpagenum{\ifnum#1=1\bfseries\fi#7}}\par
    \nobreak
    \endgroup
  \fi}
\AtBeginDocument{
\expandafter\renewcommand\csname r@tocindent0\endcsname{0pt}
}
\def\l@subsection{\@tocline{2}{0pt}{2.5pc}{5pc}{}}
\makeatother

\usepackage{amsmath}
\usepackage{amsthm}
\usepackage{amssymb}
\usepackage{mathrsfs} 
\usepackage{mathtools}
\usepackage{bm}

\usepackage{graphicx}
\usepackage{tikz}
\usetikzlibrary{cd}
\usetikzlibrary{arrows.meta}

\newcounter{results}[section]

\theoremstyle{plain}
\newtheorem{theorem}[results]{Theorem}

\newtheorem{lemma}[results]{Lemma}
\newtheorem{proposition}[results]{Proposition}
\newtheorem{corollary}[results]{Corollary}

\theoremstyle{remark}
\newtheorem{remark}[results]{Remark}
\newtheorem{example}[results]{Example}

\theoremstyle{definition}
\newtheorem{definition}[results]{Definition}

\numberwithin{equation}{section}

\allowdisplaybreaks

\usepackage{xparse}

\ExplSyntaxOn
\NewDocumentCommand{\makeabbrev}{mmm}
 {
  \yoruk_makeabbrev:nnn { #1 } { #2 } { #3 }
 }

\cs_new_protected:Npn \yoruk_makeabbrev:nnn #1 #2 #3
 {
  \clist_map_inline:nn { #3 }
   {
    \cs_new_protected:cpn { #2 } { #1 { ##1 } }
   }
 }
 \ExplSyntaxOff

\makeabbrev{\textbf}{tbf#1}{a,b,c,d,e,f,g,h,i,j,k,l,m,n,o,p,q,r,s,t,u,v,w,x,y,z,A,B,C,D,E,F,G,H,I,J,K,L,M,N,O,P,Q,R,S,T,U,V,W,X,Y,Z}

\makeabbrev{\textbf}{bf#1}{a,b,c,d,e,f,g,h,i,j,k,l,m,n,o,p,q,r,s,t,u,v,w,x,y,z,A,B,C,D,E,F,G,H,I,J,K,L,M,N,O,P,Q,R,S,T,U,V,W,X,Y,Z}

\makeabbrev{\textsf}{tsf#1}{a,b,c,d,e,f,g,h,i,j,k,l,m,n,o,p,q,r,s,t,u,v,w,x,y,z,A,B,C,D,E,F,G,H,I,J,K,L,M,N,O,P,Q,R,S,T,U,V,W,X,Y,Z}

\makeabbrev{\mathsf}{mss#1}{a,b,c,d,e,f,g,h,i,j,k,l,m,n,o,p,q,r,s,t,u,v,w,x,y,z,A,B,C,D,E,F,G,H,I,J,K,L,M,N,O,P,Q,R,S,T,U,V,W,X,Y,Z}

\makeabbrev{\mathfrak}{mf#1}{a,b,c,d,e,f,g,h,i,j,k,l,m,n,o,p,q,r,s,t,u,v,w,x,y,z,A,B,C,D,E,F,G,H,I,J,K,L,M,N,O,P,Q,R,S,T,U,V,W,X,Y,Z}

\makeabbrev{\mathrm}{mrm#1}{a,b,c,d,e,f,g,h,i,j,k,l,m,n,o,p,q,r,s,t,u,v,w,x,y,z,A,B,C,D,E,F,G,H,I,J,K,L,M,N,O,P,Q,R,S,T,U,V,W,X,Y,Z}

\makeabbrev{\mathbf}{mbf#1}{a,b,c,d,e,f,g,h,i,j,k,l,m,n,o,p,q,r,s,t,u,v,w,x,y,z,A,B,C,D,E,F,G,H,I,J,K,L,M,N,O,P,Q,R,S,T,U,V,W,X,Y,Z}

\makeabbrev{\mathcal}{mc#1}{A,B,C,D,E,F,G,H,I,J,K,L,M,N,O,P,Q,R,S,T,U,V,W,X,Y,Z}

\makeabbrev{\mathbb}{mbb#1}{A,B,C,D,E,F,G,H,I,J,K,L,M,N,O,P,Q,R,S,T,U,V,W,X,Y,Z}

\makeabbrev{\mathscr}{ms#1}{A,B,C,D,E,F,G,H,I,J,K,L,M,N,O,P,Q,R,S,T,U,V,W,X,Y,Z}

\makeabbrev{\mathrm}{#1}{
id,ran,rk,diag,stab,ann,pr,ev,End,Hom,sgn,im,op,can,fin,ext,red,tot,Aut,Ad,ad,hor,ver,
rot,usc,lsc,LocLip,bSymLip,osc,loc,uloc,spec,coz,z,ul,
Opt,Adm,Cpl,Geo,GeoOpt,GeoAdm,GeoCpl,reg,aut,
bd,co,Ric,Exp,dExp,seg,Seg,cut,fcut,Cut,SDiff,Iso,Isom,cl,Homeo,Diff,Der,vol,dvol,inj,relint,Graph,sub,eucl,
var,law,Gam,pa,so,iso,fs,inv,pqi,mix,inn,
TestF,
}

\makeabbrev{\mathsf}{#1}{CD,BE,MCP,Ent,wMTW,MTW,RCD,QCD,EVI,Irr,SC,wFe,VA,UP,Curv,Alex,CAT}

\makeabbrev{\textsc}{#1}{df,fv,tv}

\renewcommand{\paragraph}[1]{\medskip\emph{#1}.\quad}

\newcommand{\eqdef}{\coloneqq }

\newcommand{\set}[1]{\left\{#1\right\}}							
\newcommand{\ceiling}[1]{\left\lceil#1\right\rceil}					
\newcommand{\floor}[1]{\left\lfloor#1\right\rfloor}					
\newcommand{\tparen}[1]{\big({#1}\big)}

\newcommand{\comma}{\,\,\mathrm{,}\;\,}

\newcommand{\fstop}{\,\,\mathrm{.}}

\renewcommand{\div}{\mathrm{div}}

\newcommand{\sfH}{\mathsf H} 

\newcommand{\N}{\mathbb{N}}

\newcommand{\R}{\mathbb{R}}

\newcommand{\vv}{{\mbox{\boldmath$v$}}}

\newcommand{\aalpha}{{\mbox{\boldmath$\alpha$}}}

\newcommand{\ggamma}{{\mbox{\boldmath$\gamma$}}}

\newcommand{\eeta}{{\mbox{\boldmath$\eta$}}}

\newcommand{\ppi}{{\mbox{\boldmath$\pi$}}}

\newcommand{\ssigma}{{\mbox{\boldmath$\sigma$}}}

\newcommand{\sfc}{{\sf c}}
\newcommand{\sfd}{{\sf d}}
\newcommand{\sfe}{{\sf e}}

\newcommand{\sfr}{{\sf r}}

\newcommand{\sfx}{{\sf x}}

\newcommand{\rmC}{{\mathrm C}}

\newcommand{\rmB}{{\mathrm B}}

\newcommand{\Kliminf}{K\kern-3pt-\kern-2pt\mathop{\rm lim\,inf}\limits}  
\newcommand{\supp}{\mathop{\rm supp}\nolimits}   
\newcommand{\argmin}{\mathop{\rm argmin}\limits}   
\newcommand{\LSC}{\operatorname{LSC}}          

\renewcommand{\d}{{ \mathrm d}}

\newcommand{\eps}{\varepsilon}  
\newcommand{\nchi}{{\raise.3ex\hbox{$\chi$}}}
\newcommand{\weakto}{\rightharpoonup}
\newcommand{\prob}{\mathcal P}

\newcommand{\de}{{\,\rm d}}

\newcommand{\geo}{\rm Geo}

\newcommand{\nc}{\normalcolor}

\newcommand{\AC}{\mathrm{AC}}

\newcommand{\J}I
\newcommand{\CE}{\mathsf{C\kern-1pt E}}
\newcommand{\NE}{\mathsf{N\kern-2.5pt E}}
\newcommand{\wCE}{\mathsf{wC\kern-1pt E}}

\newcommand{\pCE}{\mathsf{pC\kern-1pt E}}

\newcommand{\mres}{\mathbin{\vrule height 1.6ex depth 0pt width
0.13ex\vrule height 0.13ex depth 0pt width 1.3ex}}

\newcommand{\cce}[1]{\ensuremath{\operatorname{\overline{\mathrm{co}}}_2\left(#1\right)}}

\newcommand{\W}{\mathbb W}

\newcommand{\meas}{\ensuremath{\mathcal{M}}} 
\newcommand{\pc}{\f{C}[X, X]} 
\newcommand{\HK}{{\mathsf H\!\!\mathsf K}}

\newcommand{\f}[1]{\mathfrak{#1}} 

\newcommand{\restricts}[2] {
	#1 
	\raisebox{-.3ex}{$|$}_{#2}	
}

\title[The inf-convolution structure of the HK distance]{The infimal convolution structure of the \\ Hellinger--Kantorovich distance}

\author{Nicolò De Ponti}
\address[Nicolò De Ponti]{Politecnico di Milano, Dipartimento di Matematica, Piazza Leonardo Da Vinci 32, I-20133 Milano, Italy}
\email{nicolo.deponti@polimi.it}

\author{Giacomo Enrico Sodini}
\address[Giacomo Enrico Sodini]{Faculty of Mathematics, University of Vienna, Oskar-Morgenstern-Platz 1, A-1090 Vienna, Austria}
\email{giacomo.sodini@univie.ac.at}

\author{Luca Tamanini}
\address[Luca Tamanini]{Dipartimento di Matematica e Fisica ``Niccol\`o Tartaglia'', Università Cattolica del Sacro Cuore, I-25133 Brescia, Italy}
\email{luca.tamanini@unicatt.it}


\begin{document}

\begin{abstract} We show that the Hellinger--Kantorovich distance can be expressed as the metric infimal convolution of the Hellinger and the Wasserstein distances, as conjectured by Liero, Mielke, and Savaré. To prove it, we study with the tools of Unbalanced Optimal Transport the so called Marginal Entropy-Transport problem that arises as a single minimization step in the definition of infimal convolution. Careful estimates and results when the number of minimization steps diverges are also provided, both in the specific case of the Hellinger--Kantorovich setting and in the general one of abstract distances.
\end{abstract}

\maketitle
\tableofcontents
\thispagestyle{empty}

\newcommand{\He}{\mathsf{He}}
\renewcommand{\W}{\mathsf{W}}

\section{Introduction}
Given a complete and separable metric space $(X, \sfd)$, the \emph{$L^2$-Kantorovich--Rubinstein} (also: \emph{Wasserstein}) \emph{distance}~$\W$ between two probability measures $\mu_0, \mu_1 \in \mcP(X)$ is given by the following minimization problem
\begin{align}\label{eq:WassersteinIntro}
\W_{2}^2(\mu_0, \mu_1) \eqdef  \inf_{\ggamma \in \Gamma(\mu_0, \mu_1)} \int_{X^2} \sfd^2 \de \ggamma \comma
\end{align}
where $\Gamma(\mu_0, \mu_1) \subset \mcP(X^2)$ is the set of \emph{couplings} between $\mu_0$ and $\mu_1$, i.e.~probabilities on the product space with marginals $\mu_0$ and $\mu_1$. Problem \eqref{eq:WassersteinIntro} belongs to the class of optimal transport problems, with cost function given by $\sfd^2$: the clever choice of the cost, connected to the underlying distance on the space $X$, ensures that, when restricted to the set of probabilities with finite second moment $\mcP_2(X)$, the Wasserstein distance $\W_{2}$ is complete and separable, length (resp.~geodesic) if $\sfd$ is a length (resp.~geodesic) distance. Moreover, the map $x \mapsto \delta_x$ becomes an isometric embedding of $(X,\sfd)$ into $(\mcP_2(X),\W_{2})$. All these remarkable properties make $\W_{2}$ the ideal candidate for metrizing the set $\mcP_2(X)$ and studying its analytical and geometrical properties. 
We refer to the classic monographs \cite{AGS08, Villani03, Villani09, santambrogio, Rachev-Ruschendorf98I, Rachev-Ruschendorf98II} and the more recent \cite{ABS21, fg21} for a thorough picture of optimal transport and the Wasserstein distance.

The extension of the optimal transport problem to the \emph{unbalanced setting} (i.e.,~the scenario where $\mu_0$ and $\mu_1$ may have different total non-negative masses) has been investigated in several works. For instance, the adaptation of the dynamical approach from balanced optimal transport~\cite{BenamouBrenier00} to the unbalanced case is discussed in \cite{49eee, RosPic, Rospicprop, ChiPeySchVia16}. Static extensions related to the Kantorovich formulation have been studied in \cite{CaffarelliMcCann06TR, FigalliGigli10}, while \cite{11uu} introduces the concept of optimal partial transport. A broad framework for unbalanced optimal transport problems is presented in \cite{SS24}, where methods from \cite{SS20} are employed to expand upon ideas initially explored in \cite{LMS18, CPSV18}.

\smallskip

We concentrate here on the distance \(\HK\), which was simultaneously introduced in \cite{LMS18, CPSV18, msg} and is referred to as the \emph{Hellinger–Kantorovich} or \emph{Wasserstein–Fisher–Rao} distance. This metric can be defined in multiple ways (see Section \ref{sec:hk} for one such approach) and is widely recognized as the natural extension of the Wasserstein distance to \(\mcM_+(X)\), the space of non-negative, finite Radon measures on \(X\).  

The metric space \((\mcM_+(X), \HK)\) is complete and separable, with a topology corresponding to the weak convergence of measures. It retains the geodesic and length-space properties of the underlying space while allowing the map \(\f{C}[X] \ni [x,r] \mapsto r \delta_x \in \mcM_+(X)\) to be an isometric embedding. Here, the \emph{geometric cone} \(\f{C}[X]\) is equipped with the standard cone distance  
\[
\sfd_{\f{C}}([x,r],[y,s]) \coloneqq \tparen{ r^2+s^2-2rs\cos(\sfd(x,y) \wedge \pi) }^{\frac{1}{2}},
\]
which serves as an appropriate setting for unbalanced optimal transport. This cone structure is simply the quotient space \(X \times [0,+\infty)\), where all points of the form \((x,0)\) are identified as a single vertex, representing the null measure. Further details can be found in Section \ref{sec:goecone}.  

A particularly intuitive way to understand the Hellinger–Kantorovich distance is through its \emph{dynamical formulation} in the Euclidean space \(\mathbb{R}^d\). Given \(\mu_0, \mu_1 \in \mcM_+(\mathbb{R}^d)\), we have  
\begin{equation} \label{eq:hkintro}
\HK^2(\mu_0, \mu_1) = \inf \set{ \int_0^1 \int_{\mathbb{R}^d} \left [ |\vv_t|^2 + \tfrac{1}{4}|w_t|^2\right ] \de \mu_t \de t } ,
\end{equation}
where the infimum is taken over all triplets \((\mu_\cdot, \vv_\cdot, w_\cdot)\) satisfying the boundary conditions \(\mu_t |_{t=i}=\mu_i\) for \(i=0,1\) and solving the \emph{continuity equation with reaction}  
\[
\partial_t \mu_t + \div (\vv_t \mu_t) = w_t \mu_t  \fstop  
\]
This formulation, which generalizes the classical Benamou–Brenier theorem~\cite{BenamouBrenier00}, highlights how the \(\HK\) geometry incorporates both transport of mass via \(\vv\) and a growth/shrinking effect represented by \(w\).
Indeed, on the one hand, we have that for the Wasserstein distance it holds
\begin{equation} \label{eq:bb}
\W_2^2(\mu_0, \mu_1) = \inf \set{ \int_0^1 \int_{\R^d}  |\vv_t|^2 \de \mu_t \de t }, \quad \mu_0, \mu_1 \in \prob_2(X) \,,
\end{equation}
where the infimum is taken among all pairs $(\mu_\cdot, \vv_\cdot)$ with $\mu_t |_{t=i}=\mu_i$ for~$i=0,1$, and solving
\[
\partial_t \mu_t + \div (\vv_t \mu_t) =0 \fstop
\]
On the other hand, a similar dynamical formulation holds also for a third classical distance on measures: the Hellinger distance (see also Section \ref{sec:he}):
\begin{equation} \label{eq:he}
\He^2_2(\mu_0, \mu_1)= \inf \set{ \int_0^1 \int_{\R^d} \tfrac{1}{4}|w_t|^2 \de \mu_t \de t }, \quad \mu_0, \mu_1 \in \meas_+(\R^d),
\end{equation}
where the infimum is taken among all pairs $(\mu_\cdot, w_\cdot)$ with $\mu_t |_{t=i}=\mu_i$ for~$i=0,1$, and solving
\begin{equation}\label{eq: defconedist_intro}
 \partial_t \mu_t = w_t \mu_t \fstop   
\end{equation}
It has been suggested in \cite[Remark 8.19]{LMS18} that the interplay between formulas \eqref{eq:hkintro}, \eqref{eq:bb}, \eqref{eq:he} can be described by saying that \textbf{the Hellinger--Kantorovich distance is the infimal convolution of the Hellinger and the Wasserstein distances}. This idea comes from the fact that, if $M$ is a Riemannian manifold and $\mathsf{g}_1, \mathsf{g}_2$ are two metric tensors on $M$, we can define the inf-convolution of the Riemannian distances associated to $\mathsf{g}_1$ and $\mathsf{g}_2$ as
\begin{equation}
    \begin{aligned}
    \sfd_{\nabla}^2(x_0, x_1) \coloneqq \inf \Bigl \{ &\int_0^1 \left [ |v_1(s)|_{\mathsf{g}_1(x(s))}^2 + |v_2(s)|_{\mathsf{g}_2(x(s))}^2 \right ] \de s \, : \\ &\qquad  x \in \rmC^1([0,1]; M), \, x(i)=x_i,\, i=0,1,  \\
    &\qquad \dot{x}(s)=v_1(s)+v_2(s) \text{ for a.e.~} s \in (0,1)\Bigr \}, \quad x_0, x_1 \in M.    
    \end{aligned}
\end{equation}
In order to give a \emph{purely metric} counterpart of this construction, in the same remark, the authors present a metric notion of infimal convolution between two distances $\sfd_1$ and $\sfd_2$ on a space $X$: for points $z_0,z_1 \in X$ they set
\begin{equation}\label{eq:hkinf}
(\sfd_1 \nabla \sfd_2)^2(z_0,z_1) \coloneqq \liminf_{N \to + \infty} \inf \left \{ N \sum_{i=1}^N \left ( \sfd_1^2(x^N_{i-1}, y^N_i) + \sfd_2^2(y_i^N, x_i^N) \right ) : x_i^N, y_i^N \in \mathscr{P}(z_0, z_1; N) \right \}\,,
\end{equation}
where $\mathscr{P}(z_0, z_1; N)$ is the set of $2N+1$ points of the form
\begin{equation}\label{eq:partition}
 \{z_0=x_0^N, y_1^N, x_1^N, y^N_2, \dots, y_{N-1}^N, x_{N-1}^N, y^N_N, x_N^N=z_1 \} \subset X.
\end{equation}
They finally conjecture that, applying the above construction to $\sfd_1 \coloneqq \He_2$ and $\sfd_2\coloneqq \W_{2}$, the resulting distance $\sfd_1 \nabla \sfd_2$ should be the Hellinger--Kantorovich distance $\HK$. 

The above ansatz, or at least the use of the term \emph{inf-convolution} to describe the Hellinger--Kantorovich distance, has appeared in a variety of papers \cite{LMS23, GM17, cit1, cit2, cit3, cit4, cit5, cit6, cit7, cit8, cit9} but, up to our knowledge, it has never been proved rigorously. The aim of the present paper is precisely to prove this conjecture, in the following sense:

\begin{theorem}\label{thm:main-intro} Let $(X, \sfd)$ be a complete, separable, and geodesic metric space. Then
\[ \HK(\mu_0, \mu_1) = (\He_2 \nabla \W_{2})(\mu_0, \mu_1) \quad \text{ for every } \mu_0, \mu_1 \in \meas_+(X).\]
\end{theorem}

A first step to see how in general the Hellinger and Wasserstein distances combine to produce the Hellinger--Kantorovich distance has appeared in \cite[Pages 1109-1111]{GM17} and we report it here since, in part, it is the starting point of our analysis. Suppose $\mu_0=r_0 \delta_{x_0}$ and $\mu_1 = r_1 \delta_{x_1}$ with $|r_0-r_1| \ll 1$ and $\sfd(x_0,x_1) \ll 1$ for some $r_0,r_1>0$ and $x_0,x_1 \in X$. Forcing to interpolate between $\mu_0$ and $\mu_1$ with another Dirac measure of the form $\tilde{\mu}_0=r_1 \delta_{x_0}$ and using a Taylor expansion, we have that 
\begin{align} \label{eq:infocnveq}
\He_2^2&(\mu_0,\tilde{\mu}_0)=|\sqrt{r_1}-\sqrt{r_0}|^2\,, \qquad \W_{2}^2(\tilde{\mu}_0, \mu_1) = r_1 \sfd^2(x_0, x_1)\,, \\
 \HK^2(\mu_0, \mu_1)&=\sfd^2_{\f{C}}([x_0,\sqrt{r_0}],[x_1,\sqrt{r_1}])=|\sqrt{r_1}-\sqrt{r_0}|^2 + 4\sqrt{r_0r_1} \sin^2 \left ( \frac{\sfd(x_0,x_1) \wedge \pi}{2} \right ) \\
 &= \He_2^2(\mu, \tilde{\mu}_0) + \W_{2}^2(\tilde{\mu}_0, \mu_1) + \mathcal{O} \left (  \sfd^2(x_0,x_1) |\sqrt{r_1}-\sqrt{r_0}| \right ).
\end{align}
Roughly speaking, this means that the two reaction and transport processes from $\mu_0$ to $\tilde{\mu}_0$ and from $\tilde{\mu}_0$ to $\mu_1$ can be considered as occurring simultaneously and independently at the infinitesimal level, following the words of \cite{GM17}, where it is also stated “justifying or quantifying the above discussion for general measures is an interesting question”. 

To answer this question, we start our analysis with the study of the problem
\begin{equation}\label{eq: WHe intro}
\W\He(\mu_0, \mu_1) \coloneqq \inf_{\nu \in \mathcal{M}_+(X)} \left \{ \He_2^2(\mu_0, \nu) + \W_{2}^2(\nu, \mu_1)\right \}, \quad \mu_0, \mu_1 \in \mathcal{M}_+(X),
\end{equation}
which corresponds to one elementary step in the construction of the inf-convolution defined in \eqref{eq:hkinf}, and fits in the category of \emph{marginal Entropy-Transport problems} analyzed in \cite[Section 3.3. E8]{LMS18}. We recast the above problem as a general \emph{Unbalanced Optimal Transport problem} (see Section \ref{sec:uot}) for a suitable cost function $\mathsf{H}: \f{C}[X]^2\to [0,+\infty]$ which can be already inferred by \eqref{eq:infocnveq}, namely
\begin{equation}\label{eq:hintro}
\mathsf{H}([x_0,r_0],[x_1,r_1]) = |\sqrt{r_0}-\sqrt{r_1}|^2 + r_1 \sfd^2(x_0, x_1) \nchi_{(0,+\infty)}(r_0), \quad [x_0,r_0], [x_1,r_1] \in \f{C}[X].   
\end{equation}
This function encodes the fact that both a change in density and a movement in space are allowed, with the latter being constrained to happen precisely at the same mass-level of the target point. The expression of $\sfH$ in \eqref{eq:hintro} resembles very much a discretized (and unbalanced) version of the metric speed of a curve in the cone $\f{C}[X]$: if $\f{y}$ is an absolutely continuous curve with values in the geometric cone, it holds
\begin{equation}\label{eq:speed-cone-intro}
|\f{y}'(t)|_{\sfd_{\f{C}}}^2 = |r_{\f{y}}'(t)|^2 + |r_{\f{y}}(t)|^2|x_{\f{y}}'(t)|_\sfd^2 \quad \text{ for a.e.~} t \in (0,1),
\end{equation}
where $|\f{y}'(t)|_{\sfd_{\f{C}}}^2$ is the metric speed of the curve $\f{y}$ w.r.t.~the cone distance $\sfd_{\f{C}}$, $r_\eta$ and $x_\eta$ are the projections of the curve on $[0,+\infty)$ and $X$, respectively, and  $|x_{\f{y}}'(t)|_\sfd^2$ is the metric speed of $x_{\f{y}}$ w.r.t.~$\sfd$. We are going to use this fact together with a useful characterization of the Hellinger--Kantorovich distance (see Theorem \ref{thm:dynhk}): if $(X, \sfd)$ is a complete, separable, and geodesic metric space, then, for every $\mu_0, \mu_1 \in \meas_+(X)$, it holds
\begin{equation}\label{eq:dynamic_HK-intro}
\HK^2(\mu_0,\mu_1) = \min  \int \mathcal{A}_2(\f{y}) \de\ppi(\f{y}) \quad \text{ where } \mathcal{A}_2(\f{y}) = {\int_0^1 |\f{y}'_t|^2_{\sfd_{\f{C}}} 
\de t}\,,
\end{equation}
where the minimum runs over all the plans $\ppi \in \prob(\AC^2([0,1];(\f{C}[X],\sfd_{\f{C}})))$ suitably connecting $\mu_0$ to $\mu_1$.

To exploit the form of $\sfH$ in \eqref{eq:hintro} and the characterization of $\HK$ in \eqref{eq:dynamic_HK-intro}, we consider the discretized action of a curve $\f{y} \in \AC^2([0,1];(\f{C}[X],\sfd_{\f{C}}))$ defined as
\[ \mathcal{A}^N_2(\f{y})\coloneqq N \sum_{i=1}^N \mathsf{H}(\f{y}((i-1)/N), \f{y}(i/N))\,,\]
and we use it in the following strategy to prove separately the two inequalities that give Theorem \ref{thm:main-intro}.
\begin{enumerate}
\item $\He_2 \nabla \W_2 \le \HK$: we consider an optimal plan $\ppi$ as in \eqref{eq:dynamic_HK-intro} and, for every $N \in \N$, we use it to induce an admissible partition as in \eqref{eq:partition} in a way to get that 
\[ (\He_2 \nabla \W_2)^2(\mu_0, \mu_1) \le \liminf_{N \to + \infty} N \sum_{i=1}^N \W\He (\mu_{i-1}^N, \mu_i^N) \le \liminf_{N \to + \infty}  \int \mathcal{A}^N_2 \de \ppi. \]
We will show that the limit inferior in the right-most term above can be estimated from above with the integral of the action of curves on the cone w.r.t.~$\ppi$, and thus recover the result. This requires comparing $\mathcal{A}^N_2$ with the time-integral of the squared metric derivative as in \eqref{eq:speed-cone-intro} and involves, in particular, the careful study of the properties of the support of $\ppi$ and of its interaction with $\mathcal{A}^N_2$.
\item $\He_2 \nabla \W_2 \ge \HK$: we consider a minimizing sequence of partitions as in \eqref{eq:hkinf} (denoting $x_i^N\coloneqq\mu_i^N$ and $y_i^N\coloneqq\sigma_i^N$) and, upon replacing each $\sigma_i^N$ with $\nu_i^N \in \argmin \W\He(\mu_{i-1}^N, \mu_i^N)$ for $i=1, \dots, N$, we suitably interpolate these points with a plan $\ppi^N \in \prob(\rmC([0,1]; (\f{C}[X], \sfd_{\f{C}})))$ which is ``piecewise" optimal for the primal formulation of the unbalanced optimal transport problem (see Theorem \ref{thm:omnibus}) related to $\mathsf{H}$,
so that 
\[ (\He_2 \nabla \W_2)^2(\mu_0, \mu_1) = \lim_{N \to + \infty}   \int \mathcal{A}^N_2 \de \pi^N.\]
Then we prove that the above limit can be bounded from below by the integral of the action of a curve on the cone w.r.t.~some admissible plan as in \eqref{eq:dynamic_HK-intro}, thus recovering the desired inequality. This is technically challenging since the lack of compactness of the underlying metric space and the fact that both $\mathcal{A}^N_2$ and $\ppi^N$ depend on $N$ make it difficult to infer the tightness of the family of plans $(\ppi^N)_N$ and also the behavior of the above limit. These difficulties are overcome by a compact approximation procedure and a fine study of the properties of $\mathcal{A}^N_2$.
\end{enumerate}
We refer the interested reader to the beginning of Sections \ref{sec:ineq1} and \ref{sec:ineq2} for a more detailed description of the strategy adopted for the proofs of the two inequalities.

\smallskip

The study carried out to prove Theorem \ref{thm:main-intro} has revealed the mathematical richness of this problem. Apart from the analysis strictly needed for the proof of the conjecture, in the paper we also investigate the abstract properties of the infimal convolution between general distances. Among other results, we show that this construction leads to a symmetric, non-negative function, though it can be degenerate (i.e., being null between different points) and may fail to satisfy the triangle inequality. A more structured behavior emerges in the case the inf-convolution is performed between (Hilbertian) norms, where the problem aligns naturally with the infimal convolution in convex analysis \cite{Rockafellar70}. This raises the compelling question of whether there exist other purely metric examples, similar to the Hellinger--Wasserstein case, where this construction yields meaningful distances.

In another direction, we also thoroughly examined the properties of the single-step minimization problem \eqref{eq: WHe intro} between the Wasserstein and Hellinger distances.  This problem corresponds to an iteration step of the associated Minimizing Movement scheme, making it interesting to explore how our analysis could be connected to the theory of gradient flows \cite{AGS08}. Notice that a peculiarity of our situation is that the Hellinger distance corresponds to a \emph{non-superlinear} entropy functional, and thus deviates from the conventional framework usually adopted in the huge literature originated by the JKO-approach \cite{Jordan-Kinderlehrer-Otto98}.

Finally, throughout the whole manuscript we restricted ourselves to the study of the inf-convolution between the 2-Hellinger and 2-Wasserstein distances, but it is natural to wonder whether a result analogous to Theorem \ref{thm:main-intro} holds for any arbitrary exponent $p \in (1,+\infty)$. Namely, if also the $p$-Hellinger--Kantorovich distance, introduced in \cite[Section I.2]{Chizat17} and far less studied than the $p=2$ counterpart, can be regarded as the ($p$-version of the) inf-convolution of $\He_p$ and $\W_{p}$. A positive answer in this direction would shed some light on the infinitesimal structure of $p$-Hellinger--Kantorovich distances. However, let us mention that the strategy sketched above for the case $p=2$ is not directly applicable when $p \neq 2$: in fact, in the former case we strongly leveraged the fine properties of $\HK$-geodesics, established in \cite{LMS18,LMS23} and not yet available for $p \neq 2$. Nonetheless, we believe that many of our techniques could be adapted to this latter case.

\medskip
\paragraph{\em\bfseries Plan of the paper} In \textbf{Section \ref{sec:prel}} we report a few preliminaries related to metric spaces and measures therein, with a particular focus on the geometric cone. In \textbf{Section \ref{sec:gencost}} we discuss the abstract properties of the infimal convolution construction. \textbf{Section \ref{sec:tdist}} is devoted to the presentation of the general framework of Unbalanced Optimal Transport and, in particular, to the definition of the Hellinger, Wasserstein, and Hellinger--Kantorovich distances, and to the proof of a few properties of these distances. In \textbf{Section \ref{sec: Whe}} we investigate the marginal Entropy-Transport problem between the Hellinger and the Wasserstein distances.  \textbf{Section \ref{sec:ineq1}} and \textbf{Section \ref{sec:ineq2}} contain the proofs of the two inequalities mentioned above, leading to the proof of Theorem \ref{thm:main-intro}. Finally, in \textbf{Appendix \ref{sec:measappendix}} we discuss the measurability of the geodesic selection map, and in \textbf{Appendix \ref{app:radius}} we report some computations related to geodesics on the metric cone.

\medskip
\paragraph{\em\bfseries Acknowledgments} The authors warmly thank G.~Savaré for proposing the problem and for his many valuable suggestions. LT is also thankful to T.~Titkos for useful discussions. NDP and LT acknowledge financial support by the INdAM-GNAMPA Project 2024 ``Mancanza di regolarità e spazi non lisci: studio di autofunzioni e autovalori'', CUP $\#E53C23001670001\#$. NDP is supported by the INdAM-GNAMPA Project ``Proprietà qualitative e regolarizzanti di equazioni ellittiche e paraboliche'', codice CUP $\#E5324001950001\#$.

\section{Preliminaries}\label{sec:prel}

In this section we collect some basic notions related to metric and measure spaces we are going to use in the paper, focusing then on the geometric structure of the cone on a metric space, and curves and measures therein.

\subsection{Metric spaces}
An \emph{extended metric space} is a pair $(X,\sfd)$ where $\sfd:X\times X\to [0,+\infty]$ is a symmetric function (called \emph{extended distance}) satisfying the triangle inequality and such that $\sfd(x,y)=0$ if and only if $x=y$. An extended metric space is called a metric space when $\sfd$ takes only finite values, and in such a case we refer to $\sfd$ as a \emph{distance}. A curve in a (possibly extended) metric space is a continuous map $\gamma:I\to X$, where $I\subset \mathbb{R}$ is an interval, and a curve is called \emph{absolutely continuous} if there exists $m\in L^1(I)$ such that
\begin{equation}\label{def:AC curve}
\sfd(\gamma(t_0),\gamma(t_1))\le \int_{t_0}^{t_1}m(t) \de t\quad \textrm{for }t_0,t_1\in I, t_0<t_1.
\end{equation}
We denote by $\rmC(I;(X,\sfd))$ (resp.~$\AC(I; (X,\sfd))$) the class of all curves (resp.~absolutely continuous curves) from $I$ to $(X,\sfd)$ endowed with the topology of uniform convergence.
The metric derivative of an absolutely continuous curve is the Borel function
\begin{equation}\label{eq:metrd}
|\gamma'(t)|_{\sfd}\coloneqq \limsup_{h\to 0}\frac{\sfd(\gamma(t+h),\gamma(t))}{|h|}, \quad t \in I.
\end{equation}
The class of all absolutely continuous curves $\gamma: I\to X$ with $|\gamma'(\cdot)|_{\sfd}\in L^2(I)$ will be denoted by $\AC^2(I; (X,\sfd))$; when $I$ is not compact, we also consider the local space $\AC^2_{loc}(I ; (X,\sfd))$, consisting of curves in $\rmC(I; (X, \sfd))$ whose restriction to every compact subinterval $J \subset I$ belongs to $\AC^2(J; (X, \sfd))$. If $(X, \sfd)$ is a complete and separable metric space, then $\AC^2([0,1]; (X,\sfd))$ is a Borel set in the space $\rmC([0,1]; (X,\sfd))$.
A curve $\gamma:[0, 1]\to X$ taking values in the extended metric space $(X,\sfd)$ is a (minimal, constant-speed) \emph{geodesic} if
\[
\sfd(\gamma(t_0),\gamma(t_1))=|t_0-t_1|\sfd(\gamma(0),\gamma(1)) \qquad \forall\, t_0,t_1\in [0,1].
\]
We denote by $\geo((X, \sfd))$ the space of minimal, constant-speed geodesics in the space $(X, \sfd)$, which is a closed subset of $\rmC([0, 1]; (X,\sfd))$. The extended metric space $(X,\sfd)$ is a geodesic space if for every $x,y\in X$ there exists a geodesic connecting them, while it is called \emph{length space} if the distance
between arbitrary pairs of points can be obtained as the infimum of the length of
the absolutely continuous curves $\gamma:[0,1]\to X$ connecting them, where the latter is defined as 
\[
\ell_\sfd(\gamma)\coloneqq \int_0^1|\gamma'(t)|_{\sfd}\,\de t\,.
\]
The length extended distance induced by the extended distance $\sfd$ on $X$ is the function $\sfd_\ell:X\times X\to [0,+\infty]$ defined as
\begin{equation}\label{eq:lenghtd}
\sfd_\ell(x,y)\coloneqq \inf \{ \ell_\sfd(\gamma) \ : \ \gamma \in \AC([0,1]; (X,\sfd)), \gamma(0)=x, \gamma(1)=y\},    
\end{equation}
and it can be proved that $(X,\sfd_\ell)$ is an extended metric space. 

We also introduce the \emph{evaluation map} $\sfe_t: \rmC(I;(X,\sfd))\to X$, $t\in I$, which is the map defined as $\sfe_t(\gamma)\coloneqq \gamma(t).$

\subsection{Measures}
Let $(X,\tau)$ be a Hausdorff topological space.  We denote by $\meas_+(X)$ the space of non-negative finite Radon measures on $X$, while $\mathcal{B}(X)$ is the Borel $\sigma$-algebra of $X$. By $\prob(X)\subset \meas_+(X)$ we denote the space of probability measures. 
We endow $\meas_{+}(X)$ with the narrow topology, i.e.~the coarsest topology such that all the maps $\mu\mapsto \int_X \varphi \de\mu$ are lower semicontinuous for every $\varphi\in \LSC_b(X)$, the space of bounded lower semicontinuous functions.  

We recall that whenever $(X,\tau)$ is a Polish space, then $\meas_+(X)$ coincides with the set of all non-negative and finite Borel measures and the narrow topology coincides with the weak topology induced by the duality with continuous and bounded functions. We write $\mu_n\rightharpoonup\mu$ to denote the weak convergence of $\mu_n$ to $\mu.$

The \emph{support} of a measure $\mu\in \meas_+(X)$ is the set 
\[
\supp(\mu)\coloneqq  \{x\in X \ | \ \forall B\in \tau \ : \ x\in B \Rightarrow \mu(B)>0\},
\]
and we say that $\mu$ is \emph{concentrated} on $A\in \mathcal{B}(X)$ if $\mu(B)=\mu(A\cap B)$ for every $B\in \mathcal{B}(X)$.

If $A\subset X$ is a Borel set and $\mu\in \meas_+(X)$, we denote by $\mu\mres A \in \meas_+(X)$ the restriction of $\mu$ to $A$.

If $\mu\in \meas_+(X)$ and $Y$ is another Hausdorff topological space, a map $T : X\to Y$ is Lusin $\mu$-measurable if for every $\varepsilon>0$ there exists a
compact set $K_{\varepsilon}\subset X$ such that $\mu(X \setminus K_{\varepsilon})\le \varepsilon$ and the restriction of $T$ to $K_{\varepsilon}$
is continuous. We denote by $T_\sharp \mu \in \meas_{+}(Y)$ the \emph{push-forward} measure defined by $T_\sharp \mu(B) \coloneqq  \mu(T^{-1}(B))$ for every $B\in \mathcal{B}(Y)$.  Recall that Borel measurable maps are Lusin $\mu$-measurable for every $\mu \in \meas_+(X)$, whenever $Y$ is Polish. 

For every $\mu, \nu \in \meas_+(X)$ there exist $\sigma \in L^1_+(X, \mu)$ and $\nu^\perp$ such that
\begin{equation}\label{eq:lebdec}
\nu= \sigma \mu + \nu^\perp, \quad \nu^\perp \perp \mu \,.  
\end{equation}

We call such a decomposition the \emph{Lebesgue decomposition} of the measures $\nu$ w.r.t.~$\mu$.

A \emph{Dirac measure} centered at a point $x\in X$ is the measure $\delta_x\in \meas_{+}(X)$ defined as 
\begin{equation}\label{eq:dirac}
    \delta_x(A)\coloneqq \begin{cases}
    1 \quad \textrm{if} \ x\in A,\\
    0 \quad \textrm{if} \ x\notin A,
\end{cases}\qquad A\in \mathcal{B}(X).
\end{equation}

\subsection{The geometric cone}\label{sec:goecone}

We define the \emph{geometric cone} $\f{C}[X]$ on $X$ as the set of pairs $(x,r)$, $x \in X$, $r \ge 0$, identifying all the pairs with $r=0$; more precisely, we set
\[ 
\f{C}[X]\coloneqq  X \times [0,+\infty) /\f{R}, \quad (x, r) \f{R} (y,s) \text{ if } \left [ ( x=y \text{ and } r=s) \text{ or } (r=s=0)\right ] .
\]
Points in the cone are usually denoted by gothic letters $\f{y}$ or equivalence classes $[x,r]$, $x \in X$, $r \ge 0$. The \emph{vertex} of the cone is the point $\f{o}$ corresponding to the equivalence class $[x,0]$, $x \in X$. The quotient map sending a point $(x,r) \in X \times [0,+\infty)$ to its equivalence class $[x,r]$ is denoted by $\f{p}$. On the other hand, the projections on $[0,+\infty)$ and $X$ are simply defined as 
\[
\sfr:\f{C}[X]\to [0,+\infty), \qquad    \sfr([x,r]) \coloneqq  r\,,
\]
and 
\[
\sfx:\f{C}[X]\to X, \qquad \sfx([x,r]) \coloneqq \begin{cases}
    x \ \textrm{ if }  r>0 \\
    \bar{x} \  \textrm{ if } r=0\,,
\end{cases} 
\]
where $\bar{x} \in X$ is some fixed point. We omit the dependence of $\sfx$ on $\bar{x}$ since in the constructions where $\sfx$ is involved this will be irrelevant. If $0 < \eps < R < +\infty$, we also introduce the sets
\begin{subequations}
\begin{equation}\label{eq:cone_r}
\f{C}_R[X] \coloneqq  \{ [x,r] \in \f{C}[X] \,:\, 0 \le r \le R \}\,,
\end{equation}
\begin{equation}\label{eq:cepsr}
\f{C}_{R, \eps}[X] \coloneqq  \{ [x,r] \in \f{C}[X] \,:\, \eps \le r \le R \}\,.   
\end{equation}
\end{subequations}

On the cone $\f{C}[X]$ there is a natural notion of topology: for a point $[x,r]$ with $x \in X$ and $r >0$ a fundamental system of neighbourhoods is given by the image through $\f{p}$ of a fundamental system of neighbourhoods at $(x,r)$ in $X \times [0,+\infty)$. A fundamental system of neighbourhoods at $\f{o}$ is given by the family of sets $\{ \f{C}_R[X], \, R>0\} \,.$ Equivalently, the topology of the geometric cone is induced, for every $a \in (0, \pi]$, by the distance $\sfd_{a,\f{C}}: \f{C}[X]\times \f{C}[X] \to [0, +\infty)$ defined as
\begin{align}\label{ss22:eq:distcone} \sfd_{a,\f{C}}([x,r],[y,s]) &\eqdef  \tparen{ r^2+s^2-2rs\cos(\sfd(x,y) \wedge a) }^{\frac{1}{2}}\\
\label{eq:cos_to_sin}
&\,=\left(|r-s|^2+4rs\sin^2\left(\frac{\sfd(x,y) \wedge a}{2}\right) \right)^{\frac{1}{2}}, \qquad [x,r], [y,s] \in \f{C}[X] \,.
\end{align}
The geometric cone $\f{C}[X]$ serves as a model for weighted Dirac measures in $X$, i.e.~measures of the form $r\delta_x$ where $x \in X$ and $r \ge 0$, see also \eqref{eq:dirac}; in fact, it can be checked that $\f{C}[X]$ is homeomorphic to this space when endowed with the weak topology (see \cite[Lemma 2.4]{SS24} for the precise statement). 
Separability and completeness of the space $(X, \sfd)$ are inherited by $(\f{C}[X], \sfd_{a, \f{C}})$. Whenever $(X, \sfd \wedge a)$ is a length (resp.~geodesic) space, also $(\f{C}[X], \sfd_{a, \f{C}})$ is a length (resp.~geodesic) space (e.g.~see~\cite[Theorem 3.6.17]{Burago-Burago-Ivanov01}). Moreover, notice that the canonical choice $\sfd_{\f{C}}\coloneqq  \sfd_{\pi, \f{C}}$ gives that $(\f{C}[X], \sfd_{\f{C}})$ is  a length (resp.~geodesic) space whenever $(X,\sfd)$ is a length (resp.~geodesic) space (\cite[Theorem 3.6.17]{Burago-Burago-Ivanov01}).

\subsection{Curves on the cone}\label{sec:curves_cone}

For a curve $\f{y} \in \rmC([0,1]; (\f{C}[X],\sfd_{\f{C}}))$ we write 
\[ 
r_{\f{y}}(t)\coloneqq \sfr(\f{y}(t)), \quad  x_{\f{y}}(t)\coloneqq \sfx(\f{y}(t)), \quad t \in [0,1]\,,
\]
and we denote by $r'_{\f{y}}(t)$ the derivative of $r_{\f{y}}$ at time $t$ and by $|x_{\f{y}}'(t)|_\sfd$ the $\sfd$-metric derivative (cf.~\eqref{eq:metrd}) of $x_{\f{y}}$ at time $t$,  $t\in [0,1]$. By \cite[Lemma 8.1]{LMS18} we have that
\begin{equation}\label{eq:AC-cone}
\f{y} \in \AC^2([0,1]; (\f{C}[X],\sfd_{\f{C}})) \quad \Longleftrightarrow \quad 
\left\{
\begin{array}{ll}
r_\f{y} & \in \AC^2([0,1];\R_+) \\
x_\f{y} & \in \AC^2_{loc}(r_\f{y}^{-1}(0,+\infty);(X,\sfd)),
\end{array}
\right.
\end{equation}
and moreover it holds
\begin{equation}\label{eq:speed-cone}
|\f{y}'(t)|_{\sfd_{\f{C}}}^2 = |r_{\f{y}}'(t)|^2 + |r_{\f{y}}(t)|^2|x_{\f{y}}'(t)|_\sfd^2 \quad \text{ for a.e.~} t \in (0,1). 
\end{equation}
Let us assume that $(X, \sfd)$ is a geodesic metric space, so that ($\f{C}[X]$, $\sfd_{\f{C}})$ is too. We now describe the structure of geodesics in the geometric cone:
\begin{enumerate}
    \item \label{item:geod_trivial} In the trivial case $\f{y}_0=\f{y}_1$, then a geodesic connecting $\f{y}_0$ to $\f{y}_1$ is given by $\f{y}(t)=\f{y}_0$, $t \in [0,1]$.
    \item \label{item:geod_cusp} If $\f{y}_0=\f{o}$ and $\f{y}_1\ne \f{o}$, then 
    \[ \f{y}(t)\coloneqq [\sfx(\f{y}_1), t \sfr(\f{y}_1)], \quad t \in [0,1]\]
    is a geodesic connecting $\f{y}_0$ to $\f{y}_1$; the case $\f{y}_1=\f{o}$ and $\f{y}_0 \ne \f{o}$ is analogous. 
    \item \label{item:geod_farapart} If both $\f{y}_0, \f{y}_1 \ne \f{o}$ and $\sfd(\sfx(\f{y}_0), \sfx(\f{y}_1)) \ge \pi$, then
    \[ \f{y}(t)\coloneqq \begin{cases} \left[\sfx(\f{y}_0), \sfr(\f{y}_0)-\big(\sfr(\f{y}_0)+\sfr(\f{y}_1)\big)t\right] \quad &\text{ if } 0 \le t \le \frac{\sfr(\f{y}_0)}{\sfr(\f{y}_0)+\sfr(\f{y}_1)}, \\
    \left[\sfx(\f{y}_1), \big(\sfr(\f{y}_0)+\sfr(\f{y}_1)\big)t-\sfr(\f{y}_0)\right] \quad &\text{ if } \frac{\sfr(\f{y}_0)}{\sfr(\f{y}_0)+\sfr(\f{y}_1)} \le t \le 1,
    \end{cases}\]
    is a geodesic connecting $\f{y}_0$ to $\f{y}_1$. 
    \item \label{item:geod_normal} If both $\f{y}_0, \f{y}_1 \ne \f{o}$, $\f{y}_0 \neq \f{y}_1$, and $0\le \sfd(\sfx(\f{y}_0), \sfx(\f{y}_1)) < \pi$, then a geodesic connecting $\f{y}_0$ to $\f{y}_1$ is given by
    \[ \f{y}(t) = [x(\theta(t)), r(t)] \quad t \in [0,1],\]
    where $x:[0,\sfd(\sfx(\f{y}_0), \sfx(\f{y}_1))]\to X$ is a constant speed geodesic in $X$ connecting $\sfx(\f{y}_0)$ to $\sfx(\f{y}_1)$ and $\theta:[0,1] \to [0,\sfd(\sfx(\f{y}_0), \sfx(\f{y}_1))]$ and $r:[0,1] \to \R_+$ are given by
    \begin{align*}
        r^2(t)&\coloneqq (1-t)^2 \sfr^2(\f{y}_0) + t^2 \sfr^2(\f{y}_1) +2t(1-t)\sfr(\f{y}_0)\sfr(\f{y}_1)\cos(\sfd(\sfx(\f{y}_0), \sfx(\f{y}_1))) &&\quad t \in [0,1],\\
        \cos(\theta(t))&\coloneqq \frac{(1-t)\sfr(\f{y}_0)+ t\sfr(\f{y}_1)\cos(\sfd(\sfx(\f{y}_0),\sfx(\f{y}_1)))}{r(t)} &&\quad t \in [0,1].
    \end{align*}
\end{enumerate}

\begin{remark}[Minimal radius of non-degenerate geodesics]\label{rem:geocone} 
In case $\f{y}_0=[x_0,r_0], \f{y}_1=[x_1,r_1]$ are both different from $\f{o}$ and $0\le d\coloneqq \sfd(\sfx(\f{y}_0), \sfx(\f{y}_1)) < \pi$, the radius $r_\f{y}$ of the geodesic $\f{y}$ connecting $\f{y}_0$ to $\f{y}_1$ is bounded from below by a strictly positive constant depending on $\f{y}_0$ and $\f{y}_1$. This is trivial if $\f{y}_0 = \f{y}_1$, as by Case \eqref{item:geod_trivial} above $r_\f{y}(t) = r_0$. If instead $\f{y}_0 \neq \f{y}_1$, then it is just a matter of computations (postponed to Appendix \ref{app:radius} in order not to burden the presentation) to show that $r_\f{y}$ attains its minimum at
\begin{equation}\label{eq:tmin}
t_{\min}(\f{y}_0, \f{y}_1) \coloneqq  
\begin{cases} 
0 \quad & \text{ if } \cos(d) \ge \frac{r_0}{r_1}, \\
\displaystyle{\frac{r_0^2-r_0r_1 \cos(d)}{\sfd^2_{\f{C}}(\f{y}_0, \f{y}_1)}} \quad & \text{ if } \cos(d) < \frac{r_0}{r_1} \wedge \frac{r_1}{r_0},\\
1 \quad & \text{ if } \cos(d) \ge \frac{r_1}{r_0},
\end{cases}
\end{equation}
with value

\begin{equation}\label{eq:rmin}
r_{\min}(\f{y}_0, \f{y}_1)\coloneqq r_{\f{y}}(t_{\min}(\f{y}_0, \f{y}_1)) = 
\begin{cases} 
r_0\wedge r_1 \quad & \text{ if } \cos(d) \ge \frac{r_0}{r_1}\wedge\frac{r_1}{r_0}, \\
\displaystyle{\frac{r_0r_1 \sin(d)}{\sfd_{\f{C}}(\f{y}_0, \f{y}_1)}} \quad & \text{ if } \cos(d) < \frac{r_0}{r_1} \wedge \frac{r_1}{r_0}.
\end{cases}
\end{equation}
We deduce that
\[ r_{\min}(\f{y}_0, \f{y}_1) >0 \quad \text{ if } r_0 r_1>0 \text{ and } 0 \le d <\pi.\]
We can also infer (see Appendix \ref{app:radius}) that 
\begin{equation}\label{eq:rmin_nicer}
r_{\min}(\f{y}_0, \f{y}_1) \ge \frac{1}{\sqrt{2}}(r_0\wedge r_1)  \quad \text{ if } r_0 r_1>0 \text{ and } 0\le d\le \pi/2.
\end{equation}
\end{remark}

\subsection{Measures and functions on the cone}\label{sec:fcone}
We define the set of measures on the cone with finite $q$-radial moment, $q \in [1,+\infty)$, as
\begin{equation*}
\mathcal{\f{M}}_+^q(\f{C}[X]) \eqdef  \left \{ \alpha \in \meas_+(\f{C}[X]) : \int_{\f{C}[X]} \sfr^q \de \alpha < + \infty  \right \}, 
\end{equation*}
and the map
\begin{equation*}
\f{h}^q : \mathcal{\f{M}}_+^q(\f{C}[X]) \to \meas_+(X), \quad \f{h}^q(\alpha) \coloneqq \sfx_\sharp (\sfr^q \alpha).
\end{equation*}
Note that the map $\f{h}^q$ does not depend on the point $\bar{x}$ in the definition of $\sfx$.\\
For $N \in \N_{\ge 1}$, we consider the product cone $\f{C}[X]^{N+1}$ endowed with the product topology. On the product cone we can consider the projections on the components $\pi^{i} : \f{C}[X]^{N+1} \to \f{C}[X]$ sending $([x_0, r_0], \dots,  [x_N, r_N])$ to $[x_i, r_i]$, the projections on the $N+1$ copies of $[0,+\infty)$, and on the $N+1$ copies of $X$ defined as $\sfr_i \eqdef  \sfr \circ \pi^{i}$ and $\sfx_i \eqdef  \sfx \circ \pi^{i}$, for $i=0, \dots, N$. Note that $\sfx_i$ depends on the choice of points $\bar{x}_i \in X$, but this will be irrelevant. We introduce the set of measures with finite $q$-the radial moment, $q\in[1,+\infty)$, in the product cone as
\begin{equation*}
\mathcal{\f{M}}_+^q(\f{C}[X]^{N+1}) \eqdef  \left \{ \aalpha \in \meas_+(\f{C}[X]^{N+1}) : \int (\sfr_0^q + \dots +\sfr_N^q ) \d \aalpha < + \infty  \right \}, 
\end{equation*}
and the maps
\begin{equation*}
\f{h}_i^q : \mathcal{\f{M}}_+^q( \f{C}[X]^{N+1} \nc) \to \meas_+(X_i), \quad \f{h}_i^q(\aalpha) = (\sfx_i)_\sharp (\sfr_i^q \aalpha) \quad i=0,\dots, N.
\end{equation*}
Note that the map $\f{h}_i^q$ does not depend on the point $\bar{x}_i \in X$ in the definition of $\sfx_i$.\\
Finally, when $N=1$, we set $\pc\coloneqq \f{C}[X]^2$ and we define, for every $(\mu_0, \mu_1) \in \meas_+(X) \times \meas_+(X)$, the set 
\begin{equation}\label{eq:hommarg}
\f{H}^q(\mu_0, \mu_1) \eqdef  \left \{ \aalpha \in  \mathcal{\f{M}}_+^q(\pc) : \f{h}_i^q(\aalpha) = \mu_i, \, i=0,1\right \}.
\end{equation}
If $\aalpha \in \f{H}^q(\mu_0, \mu_1)$, we say that $\mu_0$ and $\mu_1$ are the $q$-homogeneous marginals of $\aalpha$.\\
\quad \\
\noindent We say that a function $\mathsf H: \f{C}[X]^{N+1} \to [0,+\infty]$, $N \in \N_{\ge 1}$, is
\begin{enumerate}
\item radially $q$-homogeneous, $q \in [1,+\infty)$, if 
\[ \mathsf H([x_0,\lambda r_0], \dots, [x_N, \lambda r_N])= \lambda^q\mathsf H([x_0,r_0], \dots, [x_N,r_N]) \]
 for every $[x_0,r_0],\dots, [x_N, r_N] \in \f{C}[X],\, \lambda > 0$;
\item radially $q$-convex, if the map
\[ (r_0,\dots, r_N) \mapsto \mathsf H([x_0,r_0^{1/q}],\dots, [x_N,r_N^{1/q}]) \]
is convex in $[0,+\infty)^{N+1}$ for every $x_0, \dots, x_N \in X$.
\end{enumerate}

In the sequel, we will always assume to work with functions $\mathsf H$ proper and lower semicontinuous. Under these assumptions, if $\mathsf H$ is $q$-homogeneous, it satisfies $\mathsf H(\f{o}, \dots, \f{o})=0$.

Finally, let us also provide a simplified definition of dilation tailored to our purposes (see \cite{LMS18} for the general definition): given $N \in \N_{\geq 1}$, $q \in [1,+\infty)$, and a constant $\vartheta > 0$, the $(\vartheta,q)$-dilation of $\aalpha \in \mathcal{M}_+(\f{C}[X]^{N+1})$ is defined as
\begin{equation}\label{eq:dilation}
\mathcal{M}_+(\f{C}[X]^{N+1})\ni{\rm dil}_{\vartheta,q}(\aalpha) \coloneqq ({\rm prd}_\vartheta)_\sharp(\vartheta^q \aalpha) \,, 
\end{equation}
where
\[
{\rm prd}_\vartheta([x_0,r_0],\dots,[x_N,r_N]) \coloneqq ([x_0,r_0/\vartheta],\dots,[x_N,r_N/\vartheta]) \,.
\]
The marginal constraints have a natural scaling invariance w.r.t.\ dilation maps, in the sense that for every $\aalpha \in \mathcal{M}_+(\f{C}[X]^{N+1})$ it holds
\begin{equation}\label{eq:dil_invariance}
\f{h}_i^q({\rm dil}_{\vartheta,q}(\aalpha)) = \f{h}_i^q(\aalpha) \quad \textrm{ for every } i = 0,\dots,N\,,
\end{equation}
see \cite[Lemma 2.8]{SS24}. It is also immediate to see that, whenever $\sfH: \f{C}[X]^{N+1} \to [0,+\infty]$ is radially $q$-homogeneous, $N \in \N_{\ge 1}$, $q \in [1,+\infty)$, then
\begin{equation}\label{eq:hqhom}
\int \sfH \de \aalpha = \int \sfH \de \left ( {\rm dil}_{\vartheta,q}(\aalpha) \right ).
\end{equation}

\section{Inf-convolution between general costs}\label{sec:gencost}

In this section we study the inf-convolution generated by two general cost functions, giving sufficient conditions for the resulting cost to enjoy relevant properties. Since the main aim of this paper is to prove that, for specific choices of the involved costs, one gets the Hellinger--Kantorovich distance as inf-convolution, we focus in particular on the properties satisfied by a distance. We will show how, in general, finiteness, non-degeneracy and the triangle inequality may fail, even if one starts from very regular cost functions. 
This also highlights how nicely Hellinger and Wasserstein distances interact to produce the $\HK$ distance, and motivates the investigation of other possible well-behaved examples.

In order to avoid possible misunderstandings, let us state precisely the terminology we shall use throughout the whole section. Given a non-empty set $U$,
\begin{itemize}
    \item a \emph{cost} is a non-negative function $\sfc : U \times U \to [0,+\infty)$; if $\sfc$ is allowed to take the value $+\infty$, then we will refer to it as an \emph{extended cost}; an extended cost is \emph{null on the diagonal} if $\sfc(z,z)=0$ for all $z\in U$; a cost is \emph{symmetric} if $\sfc(z_0,z_1) = \sfc(z_1,z_0)$ for all $z_0,z_1\in U$.
    \item an \emph{(extended) semi-distance} satisfies all axioms of an (extended) distance but non-degeneracy; that is, for an extended semi-distance $\varrho$ it may happen $\varrho(z_0,z_1)=0$ for some $z_0 \neq z_1 \in U$;
    \item an \emph{(extended) semi-norm} satisfies all axioms of an (extended) norm but non-degeneracy; that is, for an extended semi-norm $\|\cdot\|$ on a vector space $V$ it may happen $\|v\|=0$ for some $0 \neq v \in V$.
\end{itemize}

Let $U$ be a set and let $\sfc_1$ and $\sfc_2$ be extended costs on $U$; let $z_0, z_1 \in U$ be fixed and let $N \in \N_{\ge 1}$. We say that $P=(x_0, x_1, \dots, x_N; y_1, \dots, y_N) \in U^{N+1} \times U^N$ is an $N$-path from $z_0$ to $z_1$ if $x_0=z_0$ and $x_N=z_1$. The collection of $N$-paths from $z_0$ to $z_1$ is denoted by $\mathscr{P}(z_0,z_1; N)$. The energy of an element $P \in \mathscr{P}(z_0,z_1; N)$ between the costs $\sfc_1$ and $\sfc_2$ is defined as 
\begin{equation}\label{eq:energy}
\mathcal{E}_N (P) \coloneqq N \sum_{i=1}^N \left ( \sfc_1^2(x_{i-1}, y_i) + \sfc_2^2(y_i, x_i) \right ).   
\end{equation}
Notice that $\mathcal{E}_N$ depends on $\sfc_1$ and $\sfc_2$ even if this is not explicit in the notation. Sometimes, we represent an $N$-path $P$ as in the diagram below, with the arrows giving a natural orientation to the points in the path: 
\[
\begin{tikzcd}[row sep= .2in, column sep = .3in]
P: &[-2em] & x_0 \arrow{r}{\sfc_1} & y_1 \arrow{r}{\sfc_2} & x_1 &[-2em] \dots &[-2em] x_{N-1} \arrow{r}{\sfc_1} & y_N \arrow{r}{\sfc_2} & x_N &
\end{tikzcd}
\]

\begin{definition}\label{def: inf conv}
The infimal convolution (inf-convolution, for short) between $\sfc_1$ and $\sfc_2$ is defined as the quantity
\[ 
(\sfc_1 \nabla \sfc_2)(z_0, z_1) \coloneqq \left ( \liminf_{N \to + \infty} \inf \left \{ \mathcal{E}_N(P) : P \in \mathscr{P}(z_0,z_1;N) \right \} \right )^{1/2}, \quad z_0, z_1 \in U \,.
\]
\end{definition}

For reasons that will be apparent in Section \ref{subsect:mt vs conv}, we will sometimes refer to the previous definition as the “\emph{metric} infimal convolution”. Let us also stress that, more precisely, \eqref{eq:energy} and $\sfc_1 \nabla \sfc_2$ should be referred to as 2-energy and 2-infimal convolution, respectively; the choice of the exponent 2 is motivated by the main example of infimal convolution discussed in this paper: the inf-convolution between the squared 2-Hellinger and 2-Wasserstein distances. However, a priori other choices of exponent are possible, so that more generally (and with obvious adaptations) one could define the $p$-infimal convolution between two extended costs.

First of all, let us discuss some basic properties of the infimal convolution.

\begin{proposition}\label{prop: infconvprop}
Let $\sfc_1,\sfc_2$ be extended costs on $U$. Then the inf-convolution $\sfc_1\nabla \sfc_2$ satisfies the following properties:
\begin{itemize}
    \item[(i)] Non-negativity: $(\sfc_1 \nabla \sfc_2)(z_0,z_1) \geq 0$ for all $z_0,z_1 \in U$.
    \item[(ii)] Monotonicity: if $\sfc_1'$ and $\sfc_2'$ are extended costs on $U$ such that $\sfc_i \le \sfc_i'$ for $i=1,2$, then $\sfc_1 \nabla \sfc_2 \le \sfc_1' \nabla\sfc_2'$. 
\end{itemize}
If $\sfc_1,\sfc_2$ are null on the diagonal, then it also holds:
\begin{itemize}
    \item[(iii)] $(\sfc_1 \nabla \sfc_2)(z,z)=0$ for any $z \in U$.
    \item[(iv)] Symmetry w.r.t.\ costs: $\sfc_1 \nabla \sfc_2 = \sfc_2 \nabla \sfc_1$.
\end{itemize}
If $\sfc_1,\sfc_2$ are null on the diagonal and symmetric, then it also holds:
\begin{itemize}
 \item[(v)] Symmetry w.r.t.\ arguments: $(\sfc_1 \nabla \sfc_2)(z_0,z_1) = (\sfc_1 \nabla \sfc_2)(z_1,z_0)$ for all $z_0,z_1 \in U$.
\end{itemize}
\end{proposition}

\begin{proof}
(i) Non-negativity is clear since for every $z_0,z_1 \in U$, $N\in \N$ and every $P \in \mathscr{P}(z_0,z_1;N)$ it holds $\mathcal{E}_N(P)\ge 0$.
\\(ii) The property follows since for every $z_0,z_1 \in U$, $N\in \N$ and every $P \in \mathscr{P}(z_0,z_1;N)$ it holds $\mathcal{E}_N(P)\leq \mathcal{E}'_N(P)$, where $\mathcal{E}'$ denotes the energy computed between the extended costs $\sfc_1',\sfc_2'$.
\\(iii) From (i) we know that $\sfc_1 \nabla \sfc_2(z,z) \ge 0.$ By considering for every $N\in \N$ the $N$-path $P_N = (z,z,\dots,z)\in U^{N+1}\times U^N$ we have that $\mathcal{E}_N(P_N)=0$ and thus $(\sfc_1 \nabla \sfc_2)(z,z)=0.$
\\(iv) Given $P_N \in \mathscr{P}(z_0,z_1;N)$, the idea is to add `costless' initial and final points to $P_N$ getting the $(N+1)$-path $P_{N+1}$ in the diagram below:
\[
\begin{tikzcd}[row sep= .2in, column sep = .3in]
P_N: &[-2em] & x_0 \arrow{r}{\sfc_1} & y_1 \arrow{r}{\sfc_2} & x_1 &[-2em] \dots &[-2em] x_{N-1} \arrow{r}{\sfc_1} & y_N \arrow{r}{\sfc_2} & x_N & \\
P_{N+1}: & x_0 \arrow{r}{\sfc_1} & x_0 \arrow{r}{\sfc_2} & y_1 \arrow{r}{\sfc_1} & x_1 & \dots & x_{N-1} \arrow{r}{\sfc_2} & y_N \arrow{r}{\sfc_1} & x_N \arrow{r}{\sfc_2} & x_N
\end{tikzcd}
\]
and observe that
\[
\mathcal{E}^{1,2}_N(P_N) = \frac{N}{N+1}\mathcal{E}^{2,1}_{N+1}(P_{N+1}) \geq \frac{N}{N+1}\inf\left\{\mathcal{E}^{2,1}_{N+1}(P) \,:\, P \in \mathscr{P}(z_0,z_1;N+1)\right\},
\]
where $\mathcal{E}^{1,2}_N, \mathcal{E}^{2,1}_N$ denote the $N$-energies associated with $(\sfc_1,\sfc_2)$ and $(\sfc_2,\sfc_1)$, respectively. This implies $(\sfc_1 \nabla \sfc_2)(z_0,z_1) \geq (\sfc_2 \nabla \sfc_1)(z_0,z_1)$ and switching the roles of $\sfc_1$, $\sfc_2$ the opposite inequality follows.
\\(v) The argument is similar to (iv). Indeed, for a given $N$-path $P_N \in \mathscr{P}(z_0,z_1;N)$ between two fixed points $z_0,z_1 \in U$, we define the path $P_{N+1}$ as in the previous point (iv) and then its `reverse' path $\tilde{P}_{N+1}$ 
\[
\begin{tikzcd}[row sep= .2in, column sep = .3in]
\tilde{P}_{N+1}: & \tilde{x}_{N+1} & \arrow{l}{\sfc_2} \tilde{y}_{N+1} & \arrow{l}{\sfc_1} \tilde{x}_N & \arrow{l}{\sfc_2} \tilde{y}_N & \dots & \tilde{y}_2 & \arrow{l}{\sfc_1} \tilde{x}_1 & \arrow{l}{\sfc_2} \tilde{y}_1 & \arrow{l}{\sfc_1} \tilde{x}_0\,.
\end{tikzcd}
\]
More explicitly, given the $N$-path $P_N = (x_0,\dots,x_N;y_1,\dots,y_N)$ with $x_0=z_0$ and $x_N = z_1$, we define $\tilde{P}_{N+1} \in \mathscr{P}(z_1,z_0;N+1)$ by setting $\tilde{x}_0 \coloneqq x_N = z_1$, $\tilde{x}_{N+1} \coloneqq x_0 = z_0$, $\tilde{x}_i \coloneqq y_{N+1-i}$ for $i=1,\dots,N$, and $\tilde{y}_i \coloneqq x_{N+1-i}$ for $i=1,\dots,N+1$. As a consequence, 
\[
\mathcal{E}_N(P_N) = \frac{N}{N+1}\mathcal{E}_{N+1}(\tilde{P}_{N+1}) \geq \frac{N}{N+1}\inf\left\{\mathcal{E}_{N+1}(P) \,:\, P \in \mathscr{P}(z_1,z_0;N+1)\right\}
\]
whence $(\sfc_1 \nabla \sfc_2)(z_0,z_1) \geq (\sfc_1 \nabla \sfc_2)(z_1,z_0)$. By swapping the roles of $z_0$ and $z_1$, the reverse inequality follows.
\end{proof}

In order for the inf-convolution to be a distance, it remains to discuss sufficient conditions to get finiteness, the triangle inequality, and non-degeneracy. Let us start from a sufficient condition to have a finite inf-convolution.

\begin{proposition}\label{prop: uppboundinfconv}
Let $\varrho_1,\varrho_2$ be two extended distances on $U$. Then 
\[\varrho_1\nabla \varrho_2\le \min\{\varrho_{1,\ell},\varrho_{2,\ell}\}\,,\]
where $\varrho_{i,\ell}$ denotes the length (possibly extended) distance induced by $\varrho_i$, see \eqref{eq:lenghtd}. In particular, $\varrho_1\nabla \varrho_2$ is finite-valued whenever $\varrho_i$ is a length distance, $i=1$ or $i=2$.
\end{proposition}
\begin{proof}
Using property (iv) of Proposition \ref{prop: infconvprop} it is enough to show $\varrho_1\nabla \varrho_2\le \varrho_{1,\ell}$.
Since $\varrho_1\leq \varrho_{1,\ell}$, using property (ii) of Proposition \ref{prop: infconvprop} we can also assume without loss of generality that $\varrho_1$ is an extended length distance. So, let $\varrho_1$ be an extended length distance, $z_0,z_1 \in U$ and let us prove $\varrho_1\nabla \varrho_2(z_0,z_1) \le \varrho_1(z_0,z_1)$. If $\varrho_1(z_0,z_1) = +\infty$ there is nothing to prove, so let us suppose $\varrho_1(z_0,z_1)<+\infty$, let $\varepsilon>0$ and let $\gamma:[0,1]\to U$ be a constant-speed rectifiable curve joining $z_0$ to $z_1$ and such that $\ell_{\rho_1}(\gamma) < \rho_1(z_0,z_1) + \varepsilon$. Since $\gamma$ is of constant-speed, notice that for every $0\le s<t\le 1$
\begin{equation}\label{eq:constrect}
    \frac{\varrho_1(\gamma(s),\gamma(t))}{t-s}\leq \frac{\ell_{\rho_1}(\gamma|_{[s,t]})}{t-s}=\ell_{\rho_1}(\gamma) \,.
\end{equation}
Let $N\in \N$ and let $(x_0, x_1, \dots, x_N)$ be $N+1$ points defined as $x_i\coloneqq \gamma(\frac{i}{N})$. Notice that $\tilde{P}\coloneqq (x_0, x_1, \dots, x_N; x_1, \dots x_N) \in U^{N+1} \times U^N$ is an $N$-path from $z_0$ to $z_1$. We thus have
\begin{align*} 
&\inf \left \{ \mathcal{E}_N(P) : P \in \mathscr{P}(z_0,z_1;N) \right \}\le \mathcal{E}_N(\tilde{P})=N\sum_{i=1}^N \varrho_1^2(x_{i-1}, x_i)\\
\leq &\sum_{i=1}^N \frac{\varrho_1^2(\gamma(\frac{i-1}{N}),\gamma(\frac{i}{N}))}{\frac{i}{N}-\frac{i-1}{N}}\le \ell_{\rho_1}(\gamma)\sum_{i=1}^N\varrho_1(x_{i-1}, x_i)\le \ell_{\rho_1}^2(\gamma)<(\varrho_1(z_0,z_1) + \varepsilon)^2\,,
\end{align*}
where we have used \eqref{eq:constrect} and the definition of length of a curve. The result follows by sending $\varepsilon\downarrow 0$ and by taking the limit inferior with respect to $N$. 
\end{proof}

However, even starting from finite distances, the resulting inf-convolution may be infinite, as the following simple example shows.

\begin{example}[Inf-convolution of distances may take the value $+\infty$]
Consider a space $U$ with at least $2$ distinct points and the discrete metrics on $U$, i.e.
$$
\varrho_1(z_0,z_1) = \varrho_2(z_0,z_1) =
\begin{cases}
    0 \qquad &\textrm{ if } z_0 = z_1 \,,\\
    1 &\textrm{ if } z_0 \neq z_1 \,.
\end{cases}
$$
Then $(\varrho_1 \nabla \varrho_2)(z_0,z_1)=+\infty$ for every $z_0 \neq z_1 \in U$. To see this, it is sufficient to notice that for every $N\in \N$ and every $N$-path $P=(x_0, x_1, \dots, x_N; y_1, \dots y_N) \in U^{N+1} \times U^N$ with $x_{i-1}\neq y_i$ or $y_i\neq x_i$ for some $i\in 1,\dots,N$, we have $\mathcal{E}_N(P)\geq N$. Since every $N$-path $P$ connecting two distinct points must be of this form, the proof is concluded.

\end{example}

We give now a sufficient condition for the inf-convolution of general costs to enjoy the triangle inequality. 

\begin{proposition}\label{prop: triangineqholds}
Let $\sfc_1,\sfc_2$ be extended costs on $U$. Assume that for every $z_0,z_1 \in U$, the infimal convolution $(\sfc_1 \nabla \sfc_2)(z_0, z_1)$ is realized as a limit, that is
\[
\text{ for every $z_0, z_1 \in U$ the limit } \lim_{N \to + \infty} \inf \left \{ \mathcal{E}_N(P) : P \in \mathscr{P}(z_0,z_1;N) \right \} \text{ exists} \,.
\]
Then $\sfc_1 \nabla \sfc_2$ satisfies the triangle inequality.
\end{proposition}

\begin{proof}
We will prove that for every $t\in (0,1)$ and for every $z_0,z_1,z_2\in U$ we have
\[
(\sfc_1 \nabla \sfc_2)^2(z_0,z_2) \le \frac{(\sfc_1 \nabla \sfc_2)^2(z_0,z_1)}{t} + \frac{(\sfc_1 \nabla \sfc_2)^2(z_1,z_2)}{1-t} \,,
\]
then obtaining the triangle inequality by optimizing with respect to $t\in (0,1).$\\
Let $t \in (0,1)$ and let us take $\{N_{1,k},P_{1,k}\}_{k=1}^\infty,\{N_{2,k},P_{2,k}\}_{k=1}^\infty$ be such that $P_{i,k}\in \mathscr{P}(z_{i-1},z_{i}; N_{i,k})$, $i=1,2$, 
\[
\lim_{k\to \infty}\mathcal{E}_{N_{i,k}}(P_{i,k}) = (\sfc_1 \nabla \sfc_2)^2(z_{i-1},z_i) \,, \quad i=1,2 \qquad \textrm{and} \qquad \lim_{k \to + \infty} \frac{N_{2,k}}{N_{1,k}} = \frac{1-t}{t} \,.
\]
Notice that if $P_{1,k}=(x_{0},x_{1},\dots,x_{N_{1,k}};y_{1},\dots,y_{N_{1,k}})$ and $P_{2,k}=(x'_{0},x'_{1},\dots,x'_{N_{2,k}}; y'_{1},\dots,y'_{N_{2,k}})$ then $\tilde{P}_k \coloneqq (x_{0},x_{1},\dots,x_{N_{1,k}},x'_{0},x'_{1},\dots,x'_{N_{2,k}}; y_{1},\dots,y_{N_{1,k}},y'_{1},\dots,y'_{N_{2,k}}) \in \mathscr{P}(z_0,z_2; N_{1,k}+N_{2,k})$,
so that 
\begin{align*}
    \inf\{\mathcal{E}_{N_{1,k}+N_{2,k}}(P):\mathscr{P}(z_0,z_2; N_{1,k}+N_{2,k})\}\le \mathcal{E}_{N_{1,k}+N_{2,k}}(\tilde{P}_k)\\
    =\frac{N_{1,k}+N_{2,k}}{N_{1,k}}\mathcal{E}_{N_{1,k}}(P_{1,k})+\frac{N_{1,k}+N_{2,k}}{N_{2,k}}\mathcal{E}_{N_{2,k}}(P_{2,k})
\end{align*}
and the conclusion follows by taking the limit as $k \to + \infty$ on both sides.
\end{proof}

We provide here an involved counter-example showing that, in general, even starting from distances, the triangle inequality may fail. In view of Proposition \ref{prop: triangineqholds}, it also provides an example where the inf-convolution is not realized as a limit.

\begin{example}[Triangle inequality may fail] We construct in steps a space $X$ and distances $\varrho_1$ and $\varrho_2$ such that the triangle inequality for the inf-convolutions fails.
\\\noindent\textbf{1st step.} (Definition of the framework) For every $n\in \N_{\ge 1}$, let $A_n\coloneqq \{a_{i,n}\}_{i=0}^{2n}$, $B_n\coloneqq \{b_{i,n}\}_{i=0}^{2n}$, and 
\[
\tilde{Y} \coloneqq \bigsqcup_{n\in\N_{\geq 1}}A_{n} \sqcup \bigsqcup_{n'\in\N_{\geq 1}}B_{n'} \,.
\]
We define $Y \coloneqq \tilde{Y}/\f{R}$, where the equivalence relation $\f{R}$ is defined as follows: given $x,y \in \tilde{Y}$, we declare $x \,\f{R}\, y$ provided
\begin{itemize}
    \item $x = a_{0,n}$ and $y = a_{0,m}$, for some $n,m\in \N_{\ge 1}$; the equivalence class of all these points will be denoted by $z_0$;
    \item $x = a_{2n,n}$ and $y = b_{0,n'}$ or vice versa, for some $n,n'\in \N_{\ge 1}$; the equivalence class of all these points will be denoted by $z_1$;
    \item $x = b_{2n',n'}$ and $y = b_{2m',m'}$, for some $n',m'\in \N_{\ge 1}$; the equivalence class of all these points will be denoted by $z_2$;
    \item $x=y$.
\end{itemize}
When it does not create confusion, we still denote by $a_{i,n}$ and $b_{j,n'}$ the equivalence classes $[a_{i,n}]$ and $[b_{j,n'}]\in Y$ of the elements $a_{i,n}$ and $b_{j,n'}\in \tilde{Y}$. 

We endow $Y$ with a graph structure by declaring the neighbors of each point (recall that the neighborhood relationship $\sim$ is symmetric):
\begin{itemize}
    \item $z_0 \sim a_{1,n}$, $n\in \N_{\geq 1}$;
    \item $z_1 \sim a_{2n-1,n}$ and $z_1 \sim b_{1,n'}$, $n,n'\in \N_{\geq 1}$;
    \item $z_2 \sim b_{2n'-1,n'}$, $n'\in \N_{\geq 1}$;
    \item $a_{i-1,n} \sim a_{i,n}$, $n\in \N_{\geq 2}$, $i=2,\dots,2n-1$;
    \item $b_{i-1,n'} \sim b_{i,n'}$, $n' \in \N_{\geq 2}$, $i=2, \dots,2n'-1$.
\end{itemize}

\smallskip
\begin{center}
\begin{tikzpicture}
\begin{scope}[every node/.style={}]
    \node (A) at (0,0) {$z_0$};
    \node (B) at (7,0) {$z_1$};
    \node (C) at (14,0) {$z_2$};
    \node (D) at (3.5,0) {$a_{1,1}$};
    \node (E) at (10.5,0) {$b_{1,1}$};
    \node (F) at (1.75,-1) {$a_{1,2}$};
    \node (G) at (3.5,-1) {$a_{2,2}$};
    \node (H) at (5.25,-1) {$a_{3,2}$};
    \node (I) at (8.75,-1) {$b_{1,2}$};
    \node (J) at (10.5,-1) {$b_{2,2}$};
    \node (K) at (12.25,-1) {$b_{3,2}$};
    \node (L) at (3.5,-1.7) {$\vdots$};
    \node (M) at (10.5,-1.7) {$\vdots$};
    \node (N) at (3.5,-2.5) {$a_{n,n}$};
    \node (O) at (10.5,-2.5) {$b_{n',n'}$};
    \node (P) at (3.5,-3.5) {$\vdots$};
    \node (Q) at (10.5,-3.5) {$\vdots$};
    \node (R) at (1,-2.5) {$a_{1,n}$};
    \node (S) at (6,-2.5) {$a_{2n-1,n}$};
    \node (T) at (8,-2.5) {$b_{1,n'}$};
    \node (U) at (13,-2.5) {$b_{2n'-1,n'}$};
\end{scope}

\begin{scope}[>={Stealth[black]},
              every node/.style={fill=white,circle},
              every edge/.style={draw=}]
    \path [-] (A) edge (D);
    \path [-] (D) edge (B);
    \path [-] (B) edge (E);
    \path [-] (E) edge (C);
    \path [-] (A) edge (F);
    \path [-] (F) edge (G);
    \path [-] (G) edge (H);
    \path [-] (H) edge (B);
    \path [-] (B) edge (I);
    \path [-] (I) edge (J);
    \path [-] (J) edge (K);
    \path [-] (K) edge (C);
    \path [-] (A) edge (R);
    \path [-] (S) edge (B);
    \draw[dashed] (R) -- (N);
    \draw[dashed] (N) -- (S);
    \path [-] (B) edge (T);
    \draw[dashed] (T) -- (O);
    \draw[dashed] (O) -- (U);
    \path [-] (U) edge (C);
\end{scope}
\end{tikzpicture}

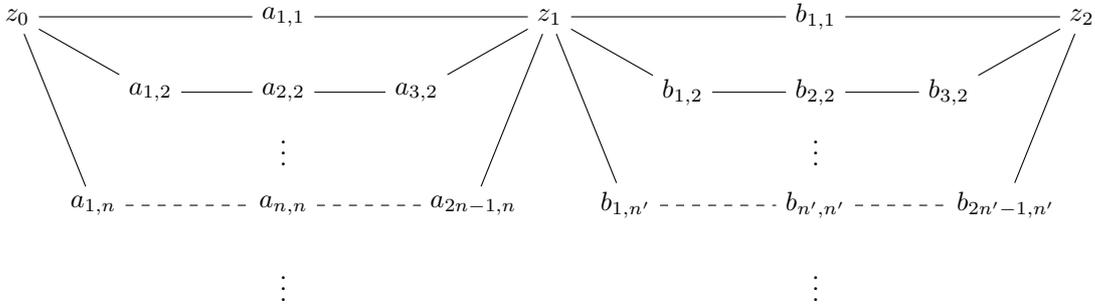
\captionof{figure}{A representation of the space $Y$.}
\end{center}

\medskip

\noindent\textbf{2nd step.} (Definition of the distances)
We now define two distances $\varrho_1,\varrho_2$ on $Y$ starting from costs $\tilde{\varrho}_1$ and $\tilde{\varrho}_2$.

We let $\tilde{\varrho}_u(y,y)=0$ for every $y \in Y, u=1,2$. We then define, for every $n\in \N_{\ge 1}$ and $i=1,\dots,2n$,

\[
\tilde{\varrho}_1(a_{i-1,n},a_{i,n})\coloneqq \begin{cases}
    \frac{1}{\sqrt{2}n} \quad &\textrm{if }i \textrm{ odd}\\
    1 \quad &\textrm{if }i \textrm{ even}
\end{cases} \,,
\qquad\qquad
\tilde{\varrho}_2(a_{i-1,n},a_{i,n})\coloneqq \begin{cases}
    1 \quad &\textrm{if }i \textrm{ odd}\\
    \frac{1}{\sqrt{2}n} \quad &\textrm{if }i \textrm{ even}
\end{cases} \,.
\]
In an analogous fashion, for every $n'\in \N$ and $i=1,\dots,2n'$ we define
\[
\tilde{\varrho}_1(b_{i-1,n'},b_{i,n'})\coloneqq \begin{cases}
    \frac{1}{\sqrt{2}n'} \quad &\textrm{if }i \textrm{ odd}\\
    1 \quad &\textrm{if }i \textrm{ even}
\end{cases} \,,
\qquad\qquad
\tilde{\varrho}_2(b_{i-1,n'},b_{i,n'})\coloneqq \begin{cases}
    1 \quad &\textrm{if }i \textrm{ odd}\\
    \frac{1}{\sqrt{2}n'} \quad &\textrm{if }i \textrm{ even}
\end{cases} \,.
\]
The distances $\varrho_1,\varrho_2$ between two arbitrary points in $Y$ are then defined by length minimization. Namely, for all $x,y \in Y$ and $u=1,2$,
\begin{equation}\label{def:chaindistance}
\varrho_u(x,y) \coloneqq \inf\left\{\sum_{j=1}^{m} \tilde{\varrho}_u(y_{j-1},y_{j}) \,:\, m \in \N_{\ge 1},\, y_0=x,\,y_m = y,\,y_{j-1} \sim y_{j} \,\forall j=1,\dots,m\right\} \,.
\end{equation}
We call \emph{chain} a collection of points $\{y_j\}_{j=0}^m$ connecting $x$ and $y$ as in the previous definition, and notice that every $x,y\in Y$ are connected by a chain and the infimum above is a minimum.

\medskip

\noindent\textbf{3rd step.} (Explicit value of $\varrho_u$) 
Let $x,y\in Y$ be such that all the following conditions are violated
\begin{enumerate}[label=\textbf{C.\arabic*}]
    \item \label{C.1} $x= y$;
    \item \label{C.2} $x\sim y$;
    \item \label{C.3} $x=a_{1,n}$ and $y= a_{1,m}$ for some $n \ne m\in \N_{\ge 1}$;
    \item \label{C.4}$x= a_{2n-1,n}$ and $y= a_{2m-1,m}$ for some $n \ne m\in \N_{\ge 1}$;
    \item \label{C.5}$x= b_{1,n'}$ and $y= b_{1,m'}$ for some $n' \ne m'\in \N_{\ge 1}$;
    \item \label{C.6} $x= b_{2n'-1,n'}$ and $y= b_{2m'-1,m'}$ for some $n' \ne m'\in \N_{\ge 1}$.
\end{enumerate}
We call \emph{distant} two points $x,y$ as above, and we claim that
\begin{equation}\label{claim dist trian}
\textrm{if }x,y\in Y \ \textrm{are distant, then }  \varrho_u(x,y)>1, \ u=1,2.
\end{equation}
To show this, first of all note that, if $x=a_{2n-1,n}$ and $y=b_{1,n'}$  (or the other way around) for some $n, n' \in \N_{\ge 1}$, then $x, y$ are distant and, for every chain $\{y_j\}_{j=0}^m$ connecting them, there are indexes $i^*, j^* \in \{1, \dots, m\}$ such that $\{y_{i^*-1}, y_{i^*}\} = \{x, z_1\}$ and $\{y_{j^*-1}, y_{j^*}\} = \{z_1, y\}$, so that $\varrho_u(x,y)>1$. 
If instead $x$ and $y$ do not have this form and are distant, for every chain $\{y_j\}_{j=0}^m$ connecting them, there exists an index $j^*\in\{1,\dots,m-1\}$ such that $\{y_{j^*-1}, y_{j^*}, y_{j^*+1}\}=\{a_{i-1,n},a_{i,n},a_{i+1,n}\}$ for some $n\in \N_{\ge 1}$ and $i\in \{1,\dots,2n-1\}$, or $\{y_{j^*-1}, y_{j^*}, y_{j^*+1}\}=\{b_{i-1,n'},b_{i,n'},b_{i+1,n'}\}$ for some $n'\in \N_{\ge 1}$ and $i\in \{1,\dots,2n'-1\}$. This proves \eqref{claim dist trian}.

Now we show that
\begin{equation}\label{eq dist for neigh_a}
{\varrho}_u(a_{i-1,n},a_{i,n})=\tilde\varrho_u(a_{i-1,n},a_{i,n}) \qquad \forall n\in\N_{\ge 1}, \ i=1,\dots,2n, \ u=1,2,
\end{equation}
and 
\begin{equation}\label{eq dist for neigh_b}
{\varrho}_u(b_{i-1,n'},b_{i,n'})=\tilde\varrho_u(b_{i-1,n'},b_{i,n'}) \qquad \forall n'\in\N_{\ge 1}, \ i=1,\dots,2n', \ u=1,2.
\end{equation}
We actually show the equality $\tilde{\varrho}_1(a_{i-1,n},a_{i,n})=\varrho_1(a_{i-1,n},a_{i,n})$, the other cases being analogous. By the definition of $\varrho_1,$ it is clear that $\tilde{\varrho}_1(a_{i-1,n},a_{i,n})\ge \varrho_1(a_{i-1,n},a_{i,n})$. For the converse inequality, notice that for every chain $\{y_j\}_{j=0}^m$ connecting $a_{i-1,n}$ to $a_{i,n}$ we must have an index $j^*\in\{1,\dots,m\}$ such that $y_{j^*-1}=a_{i-1,n}$ and $y_{j^*}=a_{i,n}$ or indexes $j^*,\overline{j}^*\in\{0,\dots,m\}$ such that $y_{j^*}=z_0$ and $y_{\overline{j}^*}=z_1$. In the first case the inequality is immediate, while in the second case we have
$\sum_{j=1}^m\tilde{\varrho}_1(y_{j-1},y_j) \ge \varrho_1(z_0, z_1) >1 \ge\tilde{\varrho}_1(a_{i-1,n},a_{i,n})$, where we used that $z_0, z_1$ are distant.

\medskip

\noindent\textbf{4th step. }(Bound on $\varrho_1\nabla\varrho_2(z_0,z_1)$ and $\varrho_1\nabla\varrho_2(z_1,z_2)$) We introduce the space $X$ as the subset of $Y$ defined as
\[
X \coloneqq \left(\bigsqcup_{k\in\N_{\geq 1}}A_{(2k)!} \sqcup \bigsqcup_{k'\in\N_{\geq 1}}B_{(2k'+1)!}\right)/\f{R} \,.
\]
We still denote by $\varrho_1,\varrho_2$ the restrictions of $\varrho_1$ and $\varrho_2$ to $X$. Notice that what we have proven so far still holds on $X$. 

Let us also introduce the sets
\[
X_A\coloneqq\left(\bigsqcup_{k\in\N_{\geq 1}}A_{(2k)!}\right)/\f{R}, \qquad X_B\coloneqq\left(\bigsqcup_{k'\in\N_{\geq 1}}B_{(2k'+1)!}\right)/\f{R}.
\]

We claim now that $\varrho_1\nabla\varrho_2(z_0,z_1)\le 1$ and $\varrho_1\nabla\varrho_2(z_1,z_2)\le 1$. We prove $\varrho_1\nabla\varrho_2(z_0,z_1)\le 1$, the other case being analogous. Notice that it is sufficient to prove that $\lim_{k\to \infty} \mathcal{E}_{(2k)!}(P_{(2k)!})=1$, where $\mathcal{E}_{(2k)!}$ is the energy between $\varrho_1$ and $\varrho_2$ and $P_{(2k)!}\in \mathscr{P}(z_0,z_1;(2k)!)$ is the $(2k)!$-path from $z_0$ to $z_1$ defined as 
\[
P_{(2k)!}=(z_0,a_{2,(2k)!},a_{4,(2k)!},\dots,z_1; \,a_{1,(2k)!},a_{3,(2k)!},\dots,a_{2(2k)!-1,(2k)!}) \,.
\]
We have
\begin{align*}
\mathcal{E}_{(2k)!}(P_{(2k)!})&=(2k)! \sum_{i=1}^{(2k)!} \left ( \varrho_1^2(a_{2i-2,(2k)!}, a_{2i-1,(2k)!}) + \varrho_2^2(a_{2i-1,(2k)!}, a_{2i,(2k)!})\right)\\
&=(2k)! \sum_{i=1}^{(2k)!} \left(\frac{1}{\sqrt{2}(2k)!}\right)^2+\left(\frac{1}{\sqrt{2}(2k)!}\right)^2=1 \,,
\end{align*}
for every $k\in \N_{\geq 1}$, where we have used \eqref{eq dist for neigh_a} and \eqref{eq dist for neigh_b}, and thus the claim is proved. 

\medskip

\noindent\textbf{5th step. }($\varrho_1\nabla\varrho_2(z_0,z_2)=+\infty$)
In this final step we will show that 
\begin{equation}\label{final claim trian}
\varrho_1\nabla\varrho_2(z_0,z_2)=+\infty \qquad \textrm{on } X,
\end{equation} so that
\[
\varrho_1\nabla\varrho_2(z_0,z_2)=+\infty>2\ge \varrho_1\nabla\varrho_2(z_0,z_1)+\varrho_1\nabla\varrho_2(z_1,z_2) \,,
\]
thus providing an example where the inf-convolution fails to satisfy the triangle inequality. 

By definition of inf-convolution, we can find $P_{N_l}=(x^l_0, x^l_1, \dots, x^l_{N_l}; y^l_1, \dots, y^l_{N_l}) \in \mathscr{P}(z_0,z_2;N_l)$ such that $N_l\to+\infty$ as $l\to+\infty$ and
\[
\varrho_1\nabla\varrho_2(z_0,z_2)=\lim_{l\to\infty}\mathcal{E}_{N_l}^{1/2}(P_{N_l}) \,.
\]
Let $l \in \N$ be fixed; it will be useful also to denote
\[ z_{2i}^l\coloneqq x_i^l, \quad i=0, \dots, N_l,\quad z_{2i-1}^\ell\coloneqq y_i^\ell, \quad i=1, \dots, N_l. \]
We call \emph{consecutive} two points $v,w\in P_{N_l}$ such that $v=z_i^l, w=z_{i+1}^l$ for some $i\in\{0,\dots,2N_l-1\}$. If $P_{N_l}$ contains two consecutive distant points, then from \eqref{claim dist trian} it immediately follows that $\mathcal{E}_{N_l}(P_{N_l})> N_l$. If $P_{N_l}$ contains only non-distant consecutive points, this means that for every pair of consecutive points $(v,w)$ in $P_{N_l}$ it holds that $(v,w)$ satisfies one among $\ref{C.1}-\ref{C.6}$. Therefore, only two cases are possible:
\begin{enumerate}
    \item Each pair of consecutive points $(v,w)$ in $P_{N_l}$ satisfies one between \ref{C.1}-\ref{C.2} (and thus does not satisfy any among conditions \ref{C.3}-\ref{C.6});
    \item $P_{N_l}$ contains only non-distant consecutive points and there exists at least one pair of consecutive points $(v,w)$ in $P_{N_l}$ that satisfies one among conditions \ref{C.3}-\ref{C.6}.
\end{enumerate}
We show that case (2) can be excluded: let us call \emph{bad} a pair of consecutive points satisfying one among conditions \ref{C.3}-\ref{C.6}; we show that any path $P_{N_l}$ as in (2) above can be replaced by a path $\tilde{P}_{N_l}$ with lower energy containing no consecutive bad or distant points, that is $\tilde{P}_{N_l}$ falls under case (1) above. We show how to do that in a very specific case, that is, when $P_{N_l}$ as in (2) contains a unique pair of bad points of the form $(v,w)=(z^l_{i^*},z^l_{i^*+1})$ for some index $i^* \in \{0, \dots, 2N_l-1\}$, satisfying \ref{C.3} (resp.~\ref{C.4}); the other cases \ref{C.5}-\ref{C.6} for a single point can be treated in a similar way, and the general case can be achieved with a simple induction argument. Thus, suppose that
\begin{itemize}
\item $(v,w)$ as above satisfies \ref{C.3}: we set
$\tilde{P}_{N_l}=(\tilde{z}^l_0, \tilde{z}^l_2, \dots, \tilde{z}^l_{ 2N_l}; \tilde{z}^l_1, \dots, \tilde{z}^l_{2N_l-1})\in\mathscr{P}(z_0,z_2;N_l)$ defined as $\tilde{z}^l_i=z_0$ for every index $i \in \{0,\dots, i^*\}$, and $\tilde{z}^l_i=z^l_i$ for every $i\in \{i^*+1,\dots, 2N_l\}$; to see that $\mathcal{E}_{N_l}(\tilde{P}_{N_l}) \le \mathcal{E}_{N_l}(P_{N_l})$ it is enough to observe that $\varrho_u(a_{1,n}, a_{1, m}) \ge \varrho_u(z_0, a_{1, m})$ for every $n\ne m \in \N_{\ge 1}$, $u=1,2$.
\item $(v,w)$ as above satisfies \ref{C.4}: let $j^*$ be the minimal value of $j$ such that $z^l_j \in X_B \setminus \{z_1\}$. We set $\tilde{P}_{N_l}=(\tilde{z}^l_0, \tilde{z}^l_2, \dots, \tilde{z}^l_{ 2N_l}; \tilde{z}^l_1, \dots, \tilde{z}^l_{2N_l-1})\in\mathscr{P}(z_0,z_2;N_l)$ defined as $\tilde{z}_i^l = z_i^l$ for every $i \in \{0, \dots, i^*\} \cup \{ j^*, \dots, 2N_l\}$, and $\tilde{z}_i^l = z_1$ for every $i \in \{ i^*+1, \dots, j^*-1\}$; to see that $\mathcal{E}_{N_l}(\tilde{P}_{N_l}) \le \mathcal{E}_{N_l}(P_{N_l})$ it is enough to observe that $\varrho_u(a_{2n-1,n}, a_{2m-1, m}) \ge \varrho_u(a_{2n-1, n},z_1)$ for every $n\ne m \in \N_{\ge 1}$, $u=1,2$.
\end{itemize}
We are left to consider only case (1). It is also easy to see that we can assume that the following does not occur: there are indexes $0 \le i^* < j^* < l^* \le 2N_l$ such that $z_{i^*}^l = z_{l^*}^l \ne z_{j^*}^l$. We call \emph{ordered} the paths avoiding this behaviour.
Therefore, we are left to consider the case in which $P_{N_l}$ is an ordered path for which every pair of consecutive points $(v,w)$ is such that $v\sim w$ or $v=w$. In particular, there exist indexes $k, k', i^*\in \N_{\ge 1}$ such that 
\[ z_i^l \in A_{(2k)!} \setminus \{z_1\} \text{ for every } i \in \{0, \dots, i^*-1\} \quad \text{ and } \quad z_i^l \in B_{(2k'+1)!} \text{ for every } i \in \{i^*, \dots, 2N_l\}.\]
Let $j_l, j'_l$ be the largest numbers in $\N_{\ge 1}$ such that $2(2j_l)! \le i^*$ and $2(2j'_l+1)! \le 2N_l-i^*$, and set $q_l\coloneqq (2j_l)!$, $q'_l\coloneqq (2j'_l+1)!$; note that $k\le j_l$, $k' \le j'_l$, and $q_l+q'_l \le N_l$. Using the values of $\varrho_u$ computed in the 3rd step, $u=1,2$, it holds
\begin{align*}
\mathcal{E}_{N_l}(P_{N_l})&= N_l \sum_{i=1}^{N_l} \left ( \varrho_1^2(x^l_{i-1}, y^l_i) + \varrho_2^2(y^l_i, x^l_i)\right) \ge (q_l+q'_l) \sum_{i=1}^{N_l} \left ( \varrho_1^2(x^l_{i-1}, y^l_i) + \varrho_2^2(y^l_i, x^l_i)\right) \\
& \ge (q_l+q'_l) \left [ \sum_{i=1}^{(2k)!} 2 \frac{1}{2((2k)!)^2} + \sum_{i=1}^{(2k'+1)!} 2 \frac{1}{2((2k'+1)!)^2}\right ] \\
& = (q_l+q'_l) \left [ \frac{1}{(2k)!} + \frac{1}{(2k'+1)!} \right ]\\
&\ge (q_l+q'_l)\left[\frac{1}{q_l}+\frac{1}{q'_l}\right]=2+\frac{(2j_l)!}{(2j'_l+1)!}+\frac{(2j'_l+1)!}{(2j_l)!}>2+2\max\{j_l,j'_l\}
\end{align*}
where the last inequality is elementary. Note that to pass from the first to the second line we have used that the part of $P_{N_l}$ in $X_A$ must contain at least $2(2k)!$ pairs of consecutive points $(v,w)$ with $v \ne w$ and $\varrho_u(v,w)$ is bounded from below by $\frac{1}{\sqrt{2} (2k)!}$ for every $u=1,2$. The same holds for the part of $P_{N_l}$ in $X_B$, replacing $2k$ with $2k'+1$.\\
We conclude that, in general, we have \[ \mathcal{E}_{N_l} (P_{N_l}) > \min \{ N_l, 2+2\max\{j_l, j'_l\} \} \to + \infty \quad \text{ as } l \to +\infty.\]
This shows \eqref{final claim trian}, and concludes the example.
\end{example}

We introduce now a definition of \emph{stability} for a cost, which will be relevant to understand when a single step of inf-convolution is already sufficient to determine the inf-convolution (see Lemma \ref{lem:criterion}). We also show that the prototypical example of a stable cost is given by a length distance (see Lemma \ref{lem:stability_length_semidist}).

\begin{definition}\label{def:good}
An extended cost $\sfc : U \times U \to [0,+\infty]$ is \emph{stable} provided
\begin{equation}\label{eq:stable}
\sfc^2(z_0,z_1) = \liminf_{N \to \infty} \inf \mathcal{F}_N(P) \,,
\end{equation}
where
\[
\mathcal{F}_N(P) \coloneqq N\sum_{i=1}^N \sfc^2(x_{i-1},x_i)
\]
and the infimum runs over all $P = (x_0,x_1,\dots,x_N) \in U^{N+1}$ such that $x_0 = z_0$ and $x_N = z_1$.
\end{definition}

\begin{lemma}\label{lem:criterion}
Let $\sfc_1,\sfc_2$ be two extended costs on $U$. If
\[
\sfc(z_0,z_1) \coloneqq \Big( \inf\left\{\mathcal{E}_1(P) \,:\, P \in \mathscr{P}(z_0,z_1;1)\right\} \Big)^{1/2}, \qquad z_0,z_1 \in U
\]
defines a stable extended cost (where $\mathcal{E}_1$ is the energy introduced in \eqref{eq:energy} for $N=1$ with costs $\sfc_1,\sfc_2$), then
\[
\sfc_1 \nabla \sfc_2 = \sfc.
\]
\end{lemma}

\begin{proof}
Let $z_0,z_1 \in U$ be fixed. The stability of $\sfc$ and its very definition yield the first and the second identity below, respectively
\begin{align*}
\sfc^2(z_0,z_1) & = \liminf_{N \to \infty} \inf \left\{ N\sum_{i=1}^N \sfc^2(x_{i-1},x_i)\right\} \\
& = \liminf_{N \to \infty} \inf \left\{ N\sum_{i=1}^N \inf_{y_i \in U}\left\{ \sfc_1^2(x_{i-1},y_i) + \sfc_2^2(y_i,x_i) \right\}\right\} \\
& = \liminf_{N \to \infty} \inf \inf_{\substack{y_i \in U, \\ i=1,\dots,N}} \left\{ N\sum_{i=1}^N \left\{ \sfc_1^2(x_{i-1},y_i) + \sfc_2^2(y_i,x_i) \right\}\right\} \,,
\end{align*}
where, as in Definition \ref{def:good}, the outermost infimum runs over all $(x_0,x_1,\dots,x_N) \in U^{N+1}$ such that $x_0 = z_0$ and $x_N = z_1$. It is now sufficient to remark that minimizing over all $(x_0,x_1,\dots,x_N) \in U^{N+1}$ as above and over all $(y_1,\dots,y_N) \in U^N$ is equivalent to minimizing over all $P \in \mathscr{P}(z_0,z_1;N)$, so that (taking also \eqref{eq:energy} into account) the last line above becomes
\[
\liminf_{N \to \infty} \inf\left\{\mathcal{E}_N(P) \,:\, P \in \mathscr{P}(z_0,z_1;N)\right\}
\]
and this is exactly $(\sfc_1 \nabla \sfc_2)^2(z_0,z_1)$.
\end{proof}

\begin{lemma}\label{lem:stability_length_semidist}
Let $\varrho$ be an extended length semi-distance on $U$. Then $\varrho$ is a stable extended cost.
\end{lemma}

\begin{proof}
Let $z_0,z_1 \in U$ be fixed. If $N \in \N_{\ge 1}$ and $P=(x_0, \dots, x_N) \in U^{N+1}$ is such that $x_0=z_0$, $x_N=z_1$, then 
\begin{equation}\label{eq:N_triangle_ineq}
\varrho^2(z_0,z_1) \leq \left ( \sum_{i=1}^N \varrho(x_{i-1}, x_i) \right )^2 \leq N \sum_{i=1}^N \varrho^2(x_{i-1}, x_i) \leq \mathcal{F}_N(P) \,.
\end{equation}
Taking the infimum over all $(N+1)$-tuples $P$ such that $x_0=z_0$, $x_N=z_1$ and then passing to the $\liminf$ as $N \to + \infty$, we get the `$\leq$' inequality in \eqref{eq:stable}.

For the converse inequality, by definition of length semi-distance we can find, for every $\eps>0$, a Lipschitz curve (which we may assume to be parametrized by arc-length) $\gamma^\eps :[0, \ell_{\varrho}(\gamma^\eps)] \to U$ such that $\ell_\varrho(\gamma^\eps) \leq \varrho(z_0,z_1) + \eps$. Let $N \in \N_{\ge 1}$ be arbitrary and let us consider points $x_i\coloneqq \gamma^\eps(i\ell/N)$ for $i=0, \dots, N$, where $\ell\coloneqq \ell_\varrho(\gamma^\eps)$ for ease of notation. Defining $P \coloneqq (x_0,\dots,x_N)$, we have
\begin{align*}
\mathcal{F}_N(P) & = N \sum_{i=1}^N \varrho^2(x_{i-1}, x_i) \leq N \sum_{i=1}^N \ell_\varrho(\restricts{\gamma^\eps}{[(i-1)\ell/N, i\ell/N]})^2 \\
& = N \sum_{i=1}^N \left( \int_{(i-1)\ell/N}^{i\ell/N} |\dot{\gamma}^\eps|(t) \de t \right)^2 = N \sum_{i=1}^N \left( \frac{\ell}{N}\right )^2 \\
& = \ell_\varrho(\gamma^\eps)^2 \leq (\varrho(z_0, z_1)+\eps)^2 \,.
\end{align*}
Considering again the infimum over all $(N+1)$-tuples $P$ such that $x_0=z_0$, $x_N=z_1$ and passing to the liminf as $N \to +\infty$, we deduce that 
\[ 
\liminf_{N \to +\infty}\inf \mathcal{F}_N(P) \leq (\varrho(z_0, z_1)+\eps)^2 \,.
\]
Since $\eps>0$ was arbitrary, we obtain the `$\geq$' inequality in \eqref{eq:stable} and thus the conclusion.
\end{proof}

From the combination of the previous two lemmas we can describe the inf-convolution of an extended (length) semi-distance with itself. The following result also gives a sufficient condition for the inf-convolution to be finite and, under stronger assumptions, to enjoy all the properties of a distance. 

\begin{corollary}\label{cor: infconvsamesemidist}
Let $\varrho$ be an extended semi-distance on $U$. Then $\frac{\varrho}{\sqrt{2}} \le \varrho \nabla \varrho$. If in addition $\varrho$ is a length semi-distance, then $\varrho \nabla \varrho = \frac{\varrho}{\sqrt{2}}$.
\end{corollary}

\begin{proof}
The first part of the statement follows from a variant of $\eqref{eq:N_triangle_ineq}$. If we fix $z_0,z_1 \in U$, take $N \in \N_{\ge 1}$ and $P=(x_0,x_1, \dots, x_N; y_1, \dots, y_N) \in \mathscr{P}(z_0,z_1; N)$, then 
\begin{equation}\label{eq:aqualung}
\varrho^2(z_0,z_1) \leq N \sum_{i=1}^N \varrho^2(x_{i-1}, x_i) \leq 2N \sum_{i=1}^N (\varrho^2(x_{i-1}, y_i) + \varrho^2(y_i, x_i)) = 2\mathcal{E}_N(P) \,.
\end{equation}
It is now sufficient to take again the infimum over $\mathscr{P}(z_0,z_1; N)$ and then pass to the $\liminf$ as $N \to + \infty$ to get $\varrho \leq \sqrt{2} \varrho \nabla \varrho$.

\smallskip

Assume now that $\varrho$ is a length semi-distance. It suffices to show that
\begin{equation}\label{eq:one_step}
\inf_{z \in U}\left\{ \varrho^2(z_0,z) + \varrho^2(z,z_1) \right\} = \frac12 \varrho^2(z_0,z_1).
\end{equation}
Indeed, by Lemma \ref{lem:stability_length_semidist} we know that (the square root of) the right-hand side above is a stable cost. Hence, if this identity holds true, then Lemma \ref{lem:criterion} ensures that the left-hand side is equal to $(\varrho \nabla \varrho)^2$, whence the conclusion. 

In order to prove \eqref{eq:one_step}, the `$\geq$' inequality is a particular case of \eqref{eq:aqualung} for $N=1$. As regards `$\leq$', we argue similarly to Lemma \ref{lem:stability_length_semidist}. By definition of length semi-distance we can find, for every $\eps>0$, a Lipschitz curve (which we may assume to be parametrized by arc-length) $\gamma^\eps :[0, \ell_{\varrho}(\gamma^\eps)] \to U$ such that $\ell_\varrho(\gamma^\eps) \leq \varrho(z_1,z_2) + \eps$. Let us consider the point $z \coloneqq \gamma^\eps(\ell/2)$, where $\ell\coloneqq \ell_\varrho(\gamma^\eps)$ for ease of notation. Setting $P \coloneqq (x_0,y_1,x_1) = (z_0,z,z_1)$, we have
\begin{align*}
\mathcal{E}_1(P) & = \varrho^2(x_0,y_1) + \varrho^2(y_1,x_1) \leq \ell_\varrho(\restricts{\gamma^\eps}{[0,\ell/2]})^2 + \ell_\varrho(\restricts{\gamma^\eps}{[\ell/2,\ell]})^2 \\
& = \left( \int_0^{\ell/2} |\dot{\gamma}^\eps|(t) \de t \right)^2 + \left( \int_{\ell/2}^{\ell} |\dot{\gamma}^\eps|(t) \de t \right)^2 = \left( \frac{\ell}{2}\right )^2 + \left( \frac{\ell}{2}\right )^2 \\
& = \frac12 \ell_\varrho(\gamma^\eps)^2 \leq \frac12 (\varrho(z_0, z_1)+\eps)^2.
\end{align*}
Taking the infimum over all $P \in \mathscr{P}(z_0,z_1;1)$, namely over all $z \in U$, and then letting $\eps \downarrow 0$ gives the conclusion.
\end{proof}

\subsection{Metric infimal convolution vs. convex infimal convolution}\label{subsect:mt vs conv}
In this section we compare our notion of (metric) infimal convolution with the \emph{convex} infimal convolution that we are going to recall. 

\begin{definition}
  Let $V$ be a vector space and $f,g:V\to (-\infty,+\infty]$ be extended real-valued functions. Their convex infimal convolution (convex inf-convolution, for short), denoted by $(f\ \Box\ g):V\to [-\infty, + \infty]$, is defined as
  \begin{equation}\label{def:convinfconv}
   (f\ \Box\ g)(v)\coloneqq \inf\{f(w)+g(z) \ : \ w,z\in V, w+z=v\}=\inf\{f(z)+g(v-z) \ : \ z\in V\}.   
  \end{equation}
\end{definition} 
 We use the terminology \emph{convex} infimal convolution to distinguish it from the \emph{metric} infimal convolution that we have introduced in Definition \ref{def: inf conv}. Notice that the convex inf-convolution corresponds to the classical infimal convolution in the sense of convex analysis (e.g., see \cite{Rockafellar70}), and in this field it is well known that the convex inf-convolution of two proper, lower semicontinuous and convex functions is again a proper, lower semicontinuous and convex function. The same operation is also called \emph{epi-sum} or \emph{epi-addition}, since the epigraph of the convex inf-convolution of two functions is the Minkowski sum of the epigraphs of those functions.

\begin{proposition}\label{prop: infconvseminorm}
Let $V$ be a vector space and $\|\cdot\|_1,\|\cdot\|_2$ be two (extended) semi-norms on $V$. Then $(\|\cdot\|^2_1\ \Box\ \|\cdot\|^2_2)^{1/2}:V\to \mathbb{R}$ is an (extended) semi-norm on $V$.
\end{proposition}
\begin{proof}
To simplify the notation, let us put $\|v\|_{\Box}\coloneqq (\|\cdot\|^2_1\ \Box\ \|\cdot\|^2_2)^{1/2}(v)$. 

It is clear that $\|\cdot\|_{\Box}$ is non-negative and finite in case the two norms are finite. Let $\lambda\in \R$ and $v\in V$, we have $\|\lambda v\|_{\Box}=|\lambda|\|v\|_{\Box}$ since the absolute homogeneity holds for $\|\cdot\|_1$ and $\|\cdot\|_2$, and 
$$\|\lambda v\|^2_{\Box}=\inf\{\|\lambda w\|^2_1+\|\lambda z\|^2_2 \ : w,z\in V, w+z=v\}.$$
It remains to prove that $\|\cdot\|_{\Box}$ satisfies the triangle inequality, which is equivalent to show that for any $v,w\in V$ and for any $t\in (0,1)$ it holds
\[
\frac{\|v\|_{\Box}^2}{t}+\frac{\|w\|_{\Box}^2}{1-t}\ge \|v+w\|_{\Box}^2.
\]
Let $\varepsilon>0$. By definition of $\|\cdot\|_{\Box}$ we can find $v_1,v_2,w_1,w_2\in V$, with $v_1+v_2=v$ and $w_1+w_2=w$, such that
\[
\frac{\|v\|_{\Box}^2+\varepsilon}{t}+\frac{\|w\|_{\Box}^2+\varepsilon}{1-t}\geq \frac{\|v_1\|_1^2+\|v_2\|_2^2}{t}+\frac{\|w_1\|_1^2+\|w_2\|_2^2}{1-t}\geq \|v_1+w_1\|_1^2+\|v_2+w_2\|_2^2\geq \|v+w\|_{\Box}^2\,,
\]
where in the second inequality we have used that $\|\cdot\|_1$ and $\|\cdot\|_2$ satisfy the triangle inequality, and again the definition of $\|\cdot\|_{\Box}$ in the last inequality since $v_1+w_1+v_2+w_2=v+w.$ The proof is concluded by sending $\varepsilon\downarrow 0$.
\end{proof}

In the next theorem we obtain that the convex infimal convolution and the metric infimal convolution coincide when computed between norms. 
\begin{theorem}\label{th: infconv=infmetric}
Let $V$ be a vector space and $\|\cdot\|_1,\|\cdot\|_2$ be two extended semi-norms on $V$. Then for every $v,w\in V$ it holds
$$(\|\cdot\|^2_1\ \Box\ \|\cdot\|^2_2)^{1/2}(v-w)=(\|\cdot\|_1\nabla\|\cdot\|_2)(v,w).$$    
\end{theorem}
\begin{proof}
By Proposition \ref{prop: infconvseminorm} we know that $(\|\cdot\|^2_1\ \Box\ \|\cdot\|^2_2)^{1/2}$ is an extended semi-norm, and thus the function 
$$(\|\cdot\|^2_1\ \Box\ \|\cdot\|^2_2)^{1/2}: V\times V\to [0,\infty], \qquad (\|\cdot\|^2_1\ \Box\ \|\cdot\|^2_2)^{1/2}(v,w)\coloneqq (\|\cdot\|^2_1\ \Box\ \|\cdot\|^2_2)^{1/2}(v-w)$$
is an extended semi-distance on $V$. It is also a length extended semi-distance since the curve
$$\gamma(t):[0,1]\to V, \qquad \gamma(t)\coloneqq v+t(w-v)$$
is a length minimizing geodesic connecting two arbitrary vectors $v,w\in V$, as a consequence of the absolute homogeneity of the extended semi-norm $(\|\cdot\|^2_1\ \Box\ \|\cdot\|^2_2)^{1/2}$. In particular, Lemma \ref{lem:stability_length_semidist} ensures that $(\|\cdot\|^2_1\ \Box\ \|\cdot\|^2_2)^{1/2}$ is a stable cost. We are thus in position to apply Lemma \ref{lem:criterion} which concludes the proof. 
\end{proof}

The following example is inspired by \cite{MOcommWen} and shows that the conclusion in Proposition \ref{prop: infconvseminorm} is sharp, in the sense that we cannot infer that the infimal convolution is a norm even when it is computed between two norms.

\begin{example}[Inf-convolution may be degenerate]\label{ex:degenerate}
Let $V$ be a vector space and $\|\cdot\|_1$, $\|\cdot\|_2$ be two complete but not equivalent norms on $V$ (one can take for instance $V=\ell^2$, $\|\cdot\|_1$ be the norm induced by the following scalar product
\[
(v,w)=\frac{1}{n}\sum_{n=1}^{\infty}v_nw_n \qquad v=(v_n)_{n=1}^{\infty},\ w=(w_n)_{n=1}^{\infty} \,,
\]
and $\|\cdot\|_2$ be the standard $\ell^2$-norm). Then $\|\cdot\|_1\nabla\|\cdot\|_2(v,w)=0$ for some $v\neq w\in V$. Indeed, suppose by contradiction that $\|\cdot\|_1\nabla\|\cdot\|_2$ is not degenerate and thus a length distance induced by a norm thanks to Proposition \ref{prop: infconvseminorm} and Theorem \ref{th: infconv=infmetric}. Let us consider the identity map $\mathsf{id}:(V,\|\cdot\|_1)\to (V,\|\cdot\|_2)$. By Proposition \ref{prop: uppboundinfconv}, the product topology of $(V,\|\cdot\|_1)\times (V,\|\cdot\|_2)$ is finer than the product topology of $(V,\|\cdot\|_1\nabla\|\cdot\|_2)^2$, which is Hausdorff. It follows that $\mathsf{id}$ has a closed graph, and thus it is continuous by the closed graph theorem. This gives the desired contradiction, since $\|\cdot\|_1$ and $\|\cdot\|_2$ were supposed to be not equivalent. 
\end{example}

In the next proposition we show, under suitable assumption, an explicit representation of the infimal convolution of two \emph{Hilbertian} norms. The formula we obtain is strongly connected to the so-called \emph{parallel addition} of two operators, introduced in the finite-dimensional setting in \cite{AndDuf69} and further extended to the case of infinite-dimensional Hilbert spaces in \cite{PekSmu76, EriLeu86}.

\begin{proposition}\label{prop: representationhilb}
Let $V$ be a vector space and $A,B:V\times V \to \R$ be two scalar products with induced norms $\|\cdot\|_A,\|\cdot\|_B$, respectively. Let us suppose that $(V,A)$ is a Hilbert space and that there exist $C_1,C_2>0$ such that
\begin{equation}
\|v\|_B\ge C_1\|v\|_A, \qquad  |B(v,w)|\le C_2\|v\|_A\|w\|_B \quad \forall v,w\in V.
\end{equation}
Then 
\begin{equation}\label{eq: explicitinfconvhilb}
    (\|\cdot\|_1\nabla\|\cdot\|_2)(v,w)=\|v-w\|_B-B(v-w,T^{A,B}(v-w)),
\end{equation}
where $T^{A,B}:V\to V$ is the unique map such that
\begin{equation}\label{eq: laxmilgramdef}
A(T^{A,B}v,w)+B(T^{A,B}v,w)=B(v,w) \quad \forall w\in V.   
\end{equation}
In particular, if $V=\R^d$ and we are identifying $A$ and $B$ with two symmetric and positive definite $d\times d$ matrices, we have 
$$(\|\cdot\|_1\nabla\|\cdot\|_2)(v,w)=(v-w)^{\intercal}(A^{-1}+B^{-1})^{-1}(v-w).$$
\end{proposition}

\begin{proof}
We start by noting that the (bijective and linear) map $T^{A,B}$ is well defined as a consequence of the Lax--Milgram lemma, since the form $b(v,w)\coloneqq A(v,w)+B(v,w)$ is continuous and strongly coercive and the map $w\mapsto B(v,w)$ is linear and continuous with respect to the normed space $(V,\|\cdot\|_A)$. \\
The last conclusion follows from \eqref{eq: explicitinfconvhilb} up to noticing the matrices identities
$$T^{A,B}=(A+B)^{-1}B, \qquad B-B(A+B)^{-1}B=(A^{-1}+B^{-1})^{-1},$$
where the second one is a consequence of the Woodbury matrix identity. 
\\It remains to prove \eqref{eq: explicitinfconvhilb}. By Theorem \ref{th: infconv=infmetric} it is sufficient to show 
$$\inf_{z\in V}A(z,z)+B(v-z,v-z)=B(v,v)-B(v,T^{A,B}v) \quad \forall v\in V$$ 
or, equivalently, that 
\begin{equation}\label{eq: proofminhilb}
\inf_{z\in V}A(z,z)-2B(v,z)+B(z,z)=-B(v,T^{A,B}v) \quad \forall v\in V,
\end{equation}
which follows easily since $z=T^{A,B}v$ is a minimizer in the left hand side of \eqref{eq: proofminhilb}.
\end{proof}

Notice that, under the assumptions of Proposition \ref{prop: representationhilb}, the resulting infimal convolution is always a norm since, up to a constant, it is bounded from below by the norm $\|\cdot\|_A$.

\section{Three distances on measures}\label{sec:tdist} 

In this whole section $(X, \sfd)$ is a metric space. We are going to introduce the Hellinger (or Matusita), the (Kantorovich--Rubinstein--)Wasserstein and the Hellinger--Kantorovich (or Wasserstein--Fisher--Rao) distances on the space of non-negative, finite and Radon measures $\meas_+(X)$ on $X$. To do so, we adopt the approach of \cite{SS24} to Unbalanced Optimal Transport in order to introduce these distances as particular instances of costs on measures.

\subsection{Unbalanced Optimal Transport}\label{sec:uot}

\begin{definition}[Unbalanced Optimal Transport cost] Let $q \in [1,+\infty)$ and let $\mathsf H: \pc \to [0,+\infty]$ be a proper, lower semicontinuous, and radially $q$-homogeneous function, see Section \ref{sec:fcone}. We define the $q$-\emph{Unbalanced Optimal Transport cost} associated to $\mathsf{H}$ as
\[ \mathsf{UOT}_{\mathsf{H},q}(\mu_0, \mu_1) \coloneqq \inf \left \{ \int_{\pc} \mathsf{H} \de \aalpha : \aalpha \in \f{H}^q(\mu_0, \mu_1) \right \}, \quad \mu_0, \mu_1 \in \meas_+(X). \]    
\end{definition}
\begin{remark}[$2$-homogeneity and $q$-homogeneity]\label{rem:tp} 
Given $q \in [1,+\infty)$, we define the map $\mathsf T_q: \pc \to \pc$ as
\begin{equation}
    \label{eq:tq}
\mathsf T_q([x_0,r_0],[x_1,r_1]) \coloneqq ([x_0,r_0^{2/q}], [x_1, r_1^{2/q}]).
\end{equation} 
It is easy to check that $(\mathsf T_q)_\sharp$ is a bijective transformation from $\f{H}^2(\mu_0, \mu_1)$ to $\f{H}^q(\mu_0, \mu_1)$ for any pair $(\mu_0, \mu_1) \in \meas_+(X) \times \meas_+(X)$. Moreover, if $\sfH_q: \pc \to [0,+\infty]$ is radially $q$-homogeneous, then $\sfH\coloneqq \sfH_q \circ \mathsf T_q$ is radially $2$-homogeneous and 
\begin{equation}
    \label{eq:qhom}
\mathsf{UOT}_{\sfH_q,q}= \mathsf{UOT}_{\sfH,2}.
\end{equation}
For this reason, we will simply use the notation $\mathsf{UOT}_{\sfH}$ in place of $\mathsf{UOT}_{\sfH, 2}$ and limit our exposition to $2$-homogeneous cost functions $\sfH$.
\end{remark}

The following result contains the main properties of the Unbalanced Optimal Transport cost \cite[Theorem 1.1]{SS24}.

\begin{theorem}\label{thm:omnibus} Let $\sfH:\pc\to [0, + \infty]$ be a proper, lower semicontinuous, and radially $2$-homogeneous function. 
\begin{enumerate}
    \item For every $(\mu_0 ,\mu_1) \in \meas_+(X)\times \meas_+(X)$ such that $\mathsf{UOT}_{\mathsf{H}}(\mu_0, \mu_1) <+\infty$, there exists an optimal $2$-homogeneous coupling
 $\aalpha \in \f{H}^2(\mu_0, \mu_1)$ such that 
\[ \mathsf{UOT}_{\mathsf{H}}(\mu_0, \mu_1) = \int_{\pc}\sfH \de \aalpha.\]
Such coupling can be chosen to be a probability concentrated on $\{ 0 \le \sfr_i \le R, \, i=0,1\}$ where $R>0$ only depends on $\mu_0(X)$ and $\mu_1(X)$.
\item $\mathsf{UOT}_{\mathsf{H}}$ is a lower semicontinuous, subadditive, convex function and satisfies 
\[\mathsf{UOT}_{\mathsf{H}}(r_0 \delta_{x_0}, r_1 \delta_{x_1}) = \cce{\sfH}([x_0,r_0],[x_1,r_1]) \le \sfH([x_0,\sqrt{r_0}], [x_1, \sqrt{r_1}])
\]
for every $x_0,x_1 \in X$ and every $r_0, r_1\ge 0$, where $\cce{\sfH}$ is the l.s.c.~convex envelope of $(r_0,r_1)\mapsto \sfH([x_0,\sqrt{r_0}],[x_1,\sqrt{r_1}])$.
  \end{enumerate}
\end{theorem}

\subsection{The three choices of cost function}

We now consider three specific choices of cost function $\sfH$ as above inducing the three distances that will play a crucial role in the sequel.

\subsubsection{The $L^p$-Wasserstein (extended) distance}\label{sec:wass}
Let $p \in [1,+\infty)$ and let $\sfH_{\W_{p,\sfd}}: \pc \to [0,+\infty]$ be the proper, lower semicontinuous, and radially $2$-homogeneous function defined as
\begin{equation}\label{eq:HW}
\sfH_{\W_{p,\sfd}}([x_0,r_0],[x_1,r_1])\coloneqq \begin{cases} r_0^2 \sfd^p(x_0,x_1) \quad &\text{ if } r_0=r_1, \\
+ \infty \quad &\text{ else},\end{cases} \quad [x_0, r_0],[x_1,r_1] \in \f{C}[X].    
\end{equation}
The resulting UOT cost induced by $\sfH_{\W_{p,\sfd}}$ is the $p$-th power of the $L^p$-Wasserstein (extended) distance $\W_{p, \sfd}$ \cite{AGS08, Villani09, santambrogio} which can be also characterized as
\[ \W_{p,\sfd}(\mu_0, \mu_1)= \left ( \inf \left \{ \int_{X \times X} \sfd^p \de \ggamma : \ggamma \in \Gamma(\mu_0, \mu_1) \right \} \right )^{1/p}, \quad \mu_0, \mu_1 \in \meas_+(X),\]
where $\Gamma(\mu_0, \mu_1)$ denotes the set of couplings between $\mu_0$ and $\mu_1$ defined as
\[ \Gamma(\mu_0, \mu_1) \coloneqq \{ \ggamma \in \meas_+(X \times X) : \pi^i_\sharp \ggamma = \mu_i, \, i=0,1 \},\]
being $\pi^i: X \times X \to X$ the projection $\pi^i(x_0, x_1)\coloneqq x_i$ for every $(x_0, x_1) \in X \times X$, $i=0,1$. Notice that, whenever $\mu_0(X) \ne \mu_1(X)$, then $\Gamma(\mu_0, \mu_1) = \emptyset$ so that $\W_{p,\sfd}(\mu_0, \mu_1)=+\infty$. When $\W_{p,\sfd}(\mu_0, \mu_1)<+\infty$ it is well known that the set of optimal couplings $\Gamma_o(\mu_0, \mu_1)$ of elements $\ggamma \in \Gamma(\mu_0, \mu_1)$ realizing the infimum above is not empty (convex and weakly compact), see e.g.~\cite[Theorem 3.5]{SS20}. When we consider the subset of $\prob(X)$ (the Borel probability measures in $X$) given by 
\[ 
\prob_p(X)\coloneqq \left\{ \mu \in \prob(X) : \int_X \sfd^p(x_0, x) \de \mu(x) < + \infty \text{ for some (hence for every) $x_0 \in X$} \right\} \,, 
\]
then $(\prob_p(X), \W_{p, \sfd})$ is a metric space. If $(X, \sfd)$ is separable, then $(\prob_p(X), \W_{p, \sfd})$ is separable. If, additionally, $(X, \sfd)$ is complete (resp.~a length space, resp.~a geodesic space), then $(\prob_p(X), \W_{p, \sfd})$ is complete (resp.~a length space, resp.~a geodesic space). In case $(X,\sfd)$ is complete and separable, the topology of $(\prob_p(X), \W_{p, \sfd})$ is induced by the following notion of convergence
$$\mu_n\rightarrow\mu \textrm{ in }\W_{p,\sfd}\Longleftrightarrow \mu_n \rightharpoonup\mu  \textrm{ and } \int_X\sfd^p(x_0, x) \de \mu_n(x)\rightarrow \int_X\sfd^p(x_0, x) \de \mu(x), \ x_0\in X.$$
For the above properties, see e.g.~\cite[Chapter 7]{AGS08}.

Finally, notice that $\sfH_{\W_{p, \sfd}}$ is radially $2$-convex, so that 
\[ \W_{p, \sfd}^p(r_0 \delta_{x_0}, r_1 \delta_{x_1}) = \begin{cases} r_0 \sfd^p(x_0, x_1) \quad & \text{ if } r_0 = r_1, \\ +\infty \quad &\text{ else,}\end{cases} \quad x_0,x_1 \in X, \, r_0, r_1 \ge 0.\]

\subsubsection{The $p$-Hellinger distance}\label{sec:he} Let $p \in [1,+\infty)$ and let $\sfH_{\He_p}: \pc \to [0,+\infty)$ be the proper, lower semicontinuous, and radially $2$-homogeneous function defined as
\begin{equation}\label{eq:HHe}
\sfH_{\He_p}([x_0,r_0],[x_1,r_1])\coloneqq \begin{cases} (r_0^{2/p}-r_1^{2/p})^p \quad &\text{ if } x_0=x_1, \\
r_0^2+r_1^2 \quad &\text{ if } x_0 \ne x_1,\end{cases} \quad [x_0, r_0],[x_1,r_1] \in \f{C}[X].
\end{equation}
The resulting UOT cost induced by $\sfH_{\He_p}$ is the $p$-th power of the $p$-Hellinger distance $\He_p$, which can be also characterized as
$$ \He_p(\mu_0, \mu_1)\coloneqq \left ( \int_{X} \left | \bigg( \frac{\de \mu_0}{\de \eta}\bigg)^{1/p} - \bigg(\frac{\de \mu_1}{\de \eta}\bigg)^{1/p} \right |^p \de \eta \right)^{1/p}, \quad \mu_0, \mu_1 \in \meas_+(X),
$$
where $\eta \in \meas_+(X)$ is any measure such that $\mu_i \ll \eta$ for $i=0,1$. Notice that the above characterization does not depend on $\eta$. The $p$-Hellinger distance is a complete and geodesic distance and induces the topology of the total variation on $\meas_+(X)$, see e.g.~\cite{Luise-Savare19}. By introducing the function 
\begin{equation}\label{eq:mp}
M_p:\mathbb{R}\to [0,+\infty] \qquad M_p(s)=\begin{cases}(s^{1/p}-1)^p &\textrm{if} \ s\ge 0,\\
+\infty &\textrm{otherwise,}
\end{cases}
\end{equation}
it is easy to show (e.g., see \cite[Lemma 3]{DePonti20}) that $\He_p^p$ can be equivalently defined as
\begin{equation}\label{def:eqHe}
\He_p^p(\mu_0,\mu_1)=D_{M_p}(\mu_0||\mu_1)\coloneqq \int_X M_p( \vartheta)\de \mu_1+\mu_0^{\perp}(X)\,,  \qquad \mu_0= \vartheta\mu_1+\mu_0^{\perp},
\end{equation}
where we have used the Lebesgue decomposition of $\mu_1$ w.r.t.~$\mu_0$, see \eqref{eq:lebdec}, and thus corresponds to the Csisz\'ar divergence generated by the convex and lower semicontinuous function $M_p$, $M_p$-divergence for short.

Finally notice that $\sfH_{\He_{p}}$ is $2$-radially convex, so that 
\[ \He_{p}^p(r_0 \delta_0, r_1 \delta_{x_1}) = \begin{cases} (r_0^{1/p}-r_1^{1/p})^p \quad &\text{ if } x_0=x_1, \\
r_0+r_1 \quad &\text{ if } x_0 \ne x_1,\end{cases} \quad x_0,x_1 \in X, \, r_0, r_1 \ge 0.\]

\subsubsection{The Hellinger--Kantorovich distance}\label{sec:hk} Let $\sfH_{\HK_\sfd}: \pc \to [0,+\infty)$ be the proper, lower semicontinuous, and radially $2$-homogeneous function defined as
\begin{equation}\label{eq:HHK}
\begin{split}
\sfH_{\HK_\sfd}([x_0,r_0],[x_1,r_1])&\coloneqq \sfd^2_{\pi/2, \f{C}}([x_0, r_0], [x_1, r_1])\\
&=\nc  r_0^2+r_1^2-2r_0 r_1\cos(\sfd(x_0,x_1) \wedge \pi/2),\quad [x_0, r_0],[x_1,r_1] \in \f{C}[X],  
\end{split}
\end{equation}
where $\sfd^2_{\pi/2, \f{C}}$ is as in \eqref{ss22:eq:distcone}.  
The resulting UOT cost induced by $\sfH_{\HK_\sfd}$ is the square of the Hellinger--Kantorovich distance $\HK_{\sfd}$. Whenever $(X,\sfd)$ is separable, then $(\meas_+(X),\HK_\sfd)$ is separable. If, additionally, $(X, \sfd)$ is complete (resp.~a length space, resp.~a geodesic space), then $(\meas_+(X),\HK_\sfd)$ is complete (resp.~a length space, resp.~a geodesic space). In case $(X,\sfd)$ is complete and separable, the topology of $(\meas_+(X),\HK_\sfd)$ coincides with the weak topology. The above properties can be found in \cite[Sections 7,8]{LMS18}. 

Notice that $\sfH_{\HK_\sfd}$ is radially $2$-convex, so that
\begin{equation}\label{eq:HK_truncated}
\HK^2(r_0 \delta_0, r_1 \delta_1) = \sfd_{\pi/2,\f{C}}^2([x_0, \sqrt{r_0}], [x_1, \sqrt{r_1}])\,.
\end{equation}
Moreover, it is not difficult to check that the squared canonical distance on the cone $\sfd^2_{\f{C}}$ is a proper, lower semicontinuous, and radially 2-homogeneous function whose l.s.c.~convex envelope $\cce{\sfd^2_{\f{C}}}$ coincides with the function
\[ \sfd^2_{\pi/2, \f{C}}([x_0, \sqrt{r_0}], [x_1, \sqrt{r_1}]),\]
so that (see e.g.~\cite[Corollary 3.18]{SS24} for a proof) the Hellinger--Kantorovich distance can be recovered minimizing w.r.t.~the canonical cone distance i.e.
\[ \HK^2_{\sfd}(\mu_0, \mu_1) = \inf \left \{ \int_{\pc} \sfd^2_{\f{C}} \de \aalpha : \aalpha \in \f{H}^2(\mu_0, \mu_1) \right \}, \quad \mu_0, \mu_1 \in \meas_+(X). \]

Let us also recall a useful dynamic formulation of $\HK$, consequence of \cite[Theorem 8.4]{LMS18}:

\begin{theorem}\label{thm:dynhk}
Let $(X,\sfd)$ be a length complete and separable metric space and let us define the lower semicontinuous functional (the 2-action functional) $\mathcal{A}_2 : \rmC([0,1];(\f{C}[X],\sfd_\f{C})) \to [0,+\infty]$ as
\begin{equation}\label{eq:a2def}
\mathcal{A}_2(\f{y}) \coloneqq   
\begin{cases} 
\displaystyle{\int_0^1 |\f{y}'_t|^2_{\sfd_{\f{C}}} 
\de t} \quad &\text{ if } \f{y} \in \AC^2([0,1];(\f{C}[X],\sfd_\f{C})), \\
+ \infty \quad &\text{ else,} 
\end{cases}
\end{equation}
where $|\f{y}'_t|^2_{\sfd_{\f{C}}} $ is the metric derivative of the curve $\f{y}$ at time $t$, see \eqref{eq:metrd} and \eqref{eq:AC-cone}.
Then, for all $\mu_0,\mu_1 \in \mathcal{M}_+(X)$ it holds
\begin{equation}\label{eq:dynamic_HK}
\HK_\sfd^2(\mu_0,\mu_1) = \inf  \int \mathcal{A}_2(\f{y}) \de\ppi(\f{y}) \,,
\end{equation}
where the infimum runs over all $\ppi \in \prob(\AC^2([0,1];(\f{C}[X],\sfd_{\f{C}}))$ such that $(\sfe_0, \sfe_1)_\sharp \ppi \in \f{H}^2(\mu_0, \mu_1)$. If $(X,\sfd)$ is additionally assumed to be geodesic, then the infimum is attained.
\end{theorem}

We conclude with a slight improvement of \cite[Theorem 8.6]{LMS18}.

\begin{theorem}\label{thm:toolbox} Let $(X, \sfd)$ be a geodesic, complete and separable metric space. For every (constant speed, length-minimizing) $\HK_{\sfd}$-geodesic $(\mu_t)_{t \in [0,1]}$ there exists a measure $\ppi \in \prob({\rm Geo}((\f{C}[X], \sfd_{\f{C}})))$ such that:
\begin{enumerate}[(a)]
\item $\f{h}^2((\mathsf{e}_t)_\sharp \ppi) = \mu_t$ for all $t \in [0,1]$;
\item it holds
\begin{equation}\label{eq:lifting_eq}
|\mu_t'|_{\HK_{\sfd}}^2 = \int |\f{y}'(t)|^2_{\sfd_{\f{C}}} \de \ppi(\f{y}) =\int \left (|r_{\f{y}}'(t)|^2 + |r_{\f{y}}(t)|^2|x_{\f{y}}'(t)|_\sfd^2 \right ) \de \ppi(\f{y}) \quad \text{ for a.e.~$t \in [0,1]$}\,;  
\end{equation}
\item $(\mathsf{e}_t)_\sharp \ppi$ is concentrated on $\f{C}_R[X] \coloneqq \{ [x,r] \in \f{C}[X] \,:\, 0 \leq r \leq R\}$, for all $t \in [0,1]$, where $R>0$ is a constant depending only on $\mu_0$ and $\mu_1$; 
\item $\ppi$ is concentrated on 
\begin{equation*} \left \{ \f{y} \in {\rm Geo}((\f{C}[X], \sfd_{\f{C}})) \,:\, r_{\f{y}}(0)r_{\f{y}}(1) >0 \,\Rightarrow\, \sfd(x_{\f{y}}(0), x_{\f{y}}(1)) \le \frac{\pi}{2} \right \}\,.
\end{equation*}
\end{enumerate}
\end{theorem}

\begin{proof} 
By \cite[Theorem 8.6]{LMS18}, we can find $\ppi \in \prob({\rm Geo}((\f{C}[X], \sfd_{\f{C})}))$ satisfying points (a)-(c) as above. We show that also point (d) must hold: we argue by contradiction, assuming that $\ppi(B)>0$, where
\[ 
B\coloneqq \left \{ \f{y} \in {\rm Geo}((\f{C}[X], \sfd_{\f{C}})) : r_{\f{y}}(0) r_{\f{y}}(1)>0 \text{ and } \sfd(x_{\f{y}}(0), x_{\f{y}}(1)) > \frac{\pi}{2} \right\}, 
\]
and we produce a curve $(\tilde{\mu}_t)_{t \in [0,1]} \in \AC^2([0,1]; (\meas_+(X),\HK_{\sfd}))$ connecting $\mu_0$ to $\mu_1$ whose length is strictly less than $\HK_{\sfd}(\mu_0, \mu_1)$. Let $T_0,T_1: {\rm Geo}((\f{C}[X], \sfd_{\f{C}})) \to {\rm Geo}((\f{C}[X], \sfd_{\f{C}}))$ be defined as
\[ T_0(\f{y})(t)\coloneqq [x_{\f{y}}(0), (1-t)\sqrt{2}r_{\f{y}}(0)], \quad T_1(\f{y})(t)\coloneqq [x_{\f{y}}(1), t\sqrt{2}r_{\f{y}}(1)], \quad t \in [0,1],\, \f{y} \in {\rm Geo}((\f{C}[X], \sfd_{\f{C}})).\]
Note that, by construction, $T_0(\f{y})$ is the geodesic connecting $\sqrt{2}\f{y}(0)$ to $\f{o}$ and $T_1(\f{y})$ is the geodesic connecting $\f{o}$ to $\sqrt{2}\f{y}(1)$ so that
\begin{equation}\label{eq:speed_T0T1}
|T_i(\f{y})'|^2_{\sfd_{\f{C}}}(t) = 2r_{\f{y}}^2(i), \quad \text{ for a.e.~} t \in (0,1), \quad i=0,1.
\end{equation}
Observe also that every element $\f{y} \in B$ satisfies
\begin{equation}\label{eq:strictly}
r_{\f{y}}^2(0)+r_{\f{y}}^2(1) < \sfd^2_{\f{C}}(\f{y}(0), \f{y}(1)) = |\f{y}'(t)|^2_{\sfd_{\f{C}}} \quad \text{for a.e.~} t \in (0,1).   
\end{equation}
We define
\[ \tilde{\ppi} \coloneqq \ppi \mres B^c + \frac{1}{2} (T_0)_\sharp (\ppi \mres B) + \frac{1}{2} (T_1)_\sharp (\ppi \mres B), \quad \tilde{\mu}_t\coloneqq  \f{h}^2((\mathsf{e}_t)_\sharp \tilde{\ppi}), \quad t \in [0,1]. \]
By construction $\tilde{\ppi} \in \prob({\rm Geo}((\f{C}[X], \sfd_{\f{C})}))$; we show that $\tilde{\mu}_i=\f{h}^2((\mathsf{e}_i)_\sharp \tilde{\ppi})=\mu_i$ for $i=0,1$: indeed, for every $\varphi \in \rmC_b(X)$, we have
\begin{align*}
\int_X \varphi \de \f{h}^2((\mathsf{e}_i)_\sharp \tilde{\ppi}) & = \int \varphi(x_{\f{y}}(i))r^2_{\f{y}}(i) \de \tilde{\ppi}(\f{y}) = \int_{B^c}\varphi(x_{\f{y}}(i))r^2_{\f{y}}(i) \de \ppi(\f{y}) \\
        & \quad + \frac{1}{2} \int_B \varphi(x_{T_0(\f{y})}(i))r^2_{T_0(\f{y})}(i) \de \ppi(\f{y})  + \frac{1}{2} \int_B \varphi(x_{T_1(\f{y})}(i))r^2_{T_1(\f{y})}(i)) \de \ppi(\f{y}) \\
        &= \int_{B^c} \varphi(x_{\f{y}}(i))r^2_{\f{y}}(i) \de \ppi(\f{y}) + \int_B \varphi(x_{\f{y}}(i)) r_{\f{y}}^2(i) \de \ppi(\f{y}) = \int_X \varphi \de \f{h}^2((\mathsf{e}_i)_\sharp \ppi).
\end{align*}
If $t \in (0,1)$ belongs to the (full measure) set of times where point (b) is satisfied for $\ppi$ and both \eqref{eq:speed_T0T1}, \eqref{eq:strictly} hold true, we have
\begin{align*}
    \int |\f{y}'(t)|^2_{\sfd_{\f{C}}} \de \tilde{\ppi}(\f{y}) &= \int_{B^c} |\f{y}'(t)|^2_{\sfd_{\f{C}}}\de \ppi (\f{y}) + \frac{1}{2} \int_{B} 2 r^2_{\f{y}}(0)  \de \ppi + \frac{1}{2} \int_{B} 2 r^2_{\f{y}}(1)  \de \ppi \\
    & < \int_{B^c} |\f{y}'(t)|^2_{\sfd_{\f{C}}} \de \ppi (\f{y}) + \int_{B} |\f{y}'(t)|^2_{\sfd_{\f{C}}}  \de \ppi(\f{y}) \\
    & = \int |\f{y}'(t)|^2_{\sfd_{\f{C}}} \de \ppi (\f{y}) = |\mu_t'|^2_{\HK_{\sfd}} = \HK_{\sfd}^2(\mu_0, \mu_1) \,.
\end{align*}
Integrating the above inequality in $[0,1]$ we deduce that $(\tilde{\mu}_t)_{t \in [0,1]} \in \AC^2([0,1]; (\meas_+(X), \HK_{\sfd}))$ (see the discussion above \cite[Formula (8.20)]{LMS18}) and that
\[ 
\ell^2_{\HK_{\sfd}}((\tilde{\mu}_t)_{t \in [0,1]}) \le \int_0^1 |\tilde{\mu}'_t|^2_{\HK_{\sfd}} \de t \le \int_0^1 \int |\f{y}'(t)|^2_{\sfd_{\f{C}}} \de \tilde{\ppi}(\f{y}) < \HK_{\sfd}^2(\mu_0, \mu_1) \,, 
\]
namely the sought contradiction.
\end{proof}

\subsection{Contraction and invariance properties of UOT}

The following is a simple contraction property of these three distances w.r.t.~Lipschitz transformations, see also \cite[Lemma 8.22]{LMS18}.

\begin{lemma}[Contraction property of UOT]\label{le:contraction} Let $(X, \sfd_X)$ and $(Y, \sfd_Y)$ be metric spaces and let $f:X \to Y$ be a $1$-Lipschitz function, i.e.
\[ \sfd_Y(f(x_0), f(x_1)) \le \sfd_X(x_0, x_1) \quad \text{ for every } x_0, x_1 \in X.\]
Then for every $\mu_0, \mu_1 \in \meas_+(X)$, we have
\begin{itemize}
    \item $\HK_{\sfd_Y}(f_\sharp \mu_0, f_\sharp \mu_1) \le \HK_{\sfd_X}(\mu_0, \mu_1)$,
    \item $\He_{p, Y}(f_\sharp \mu_0, f_\sharp \mu_1) \le \He_{p,X}(\mu_0, \mu_1)$,
    \item $\W_{p, \sfd_Y}(f_\sharp \mu_0, f_\sharp \mu_1) \le \W_{p,\sfd_X}(\mu_0, \mu_1)$,
\end{itemize}
where we added a subscript $X$ (resp.~$Y$) to denote the $p$-Hellinger distance in $\meas_+(X)$ (resp.~$\meas_+(Y)$).
\end{lemma}

\begin{proof}
    Let $\sfH_Z$ be any among $\sfH_{\W_{p,\sfd_Z}}$, $\sfH_{\He_{p,Z}}$ or $\sfH_{\HK_{\sfd_Z}}$ for $Z$ equal to $X$ or $Y$ (again we added a subscript in the notation for $\He_{p}$). Let $T_f: \f{C}[X,X] \to \f{C}[Y,Y]$ be defined as
    \[ T_f([x_0, r_0],[x_1, r_1]) \coloneqq ([f(x_0), r_0], [f(x_1), r_1]).\]
    It is easy to check that $\sfH_Y \circ T_f \le \sfH_X$ and, if $\aalpha \in \f{H}^2(\mu_0, \mu_1)$ for measures $\mu_0, \mu_1 \in \meas_+(X)$, then $(T_f)_\sharp \aalpha \in \f{H}^2(f_\sharp \mu_0, f_\sharp \mu_1)$. We deduce that for any $\aalpha \in \f{H}^2(\mu_0, \mu_1)$, we have
    \[ \mathsf{UOT}_{\sfH_Y} (f_\sharp \mu_0, f_\sharp \mu_1) \le \int_{\f{C}[Y,Y]} \sfH_Y \de [(T_f)_\sharp \aalpha] = \int_{\f{C}[X,X]} (\sfH_Y \circ T_f) \de \aalpha \le \int_{\f{C}[X,X]} \sfH_X \de \aalpha. \]
    Passing to the infimum in $\aalpha \in \f{H}^2(\mu_0, \mu_1)$, we conclude.
\end{proof}

As a consequence we get that any of the above three distances is invariant for isometric embeddings, when the metric spaces are complete.

\begin{proposition}\label{prop:iso} Let $(X, \sfd_X)$ and $(Y, \sfd_Y)$ be metric spaces, $(X, \sfd_X)$ complete, and let $\iota:X \to Y$ be an isometry onto its image, in the sense that
\[ \sfd_Y(\iota(x_0), \iota(x_1))= \sfd_X(x_0,x_1) \quad \text{ for every } x_0,x_1 \in X.\]
Then for every $\mu_0, \mu_1 \in \meas_+(X)$, we have
\begin{itemize}
    \item $\HK_{\sfd_Y}(\iota_\sharp \mu_0, \iota_\sharp \mu_1) = \HK_{\sfd_X}(\mu_0, \mu_1)$,
    \item $\He_{p, Y}(\iota_\sharp \mu_0, \iota_\sharp \mu_1) = \He_{p,X}(\mu_0, \mu_1)$,
    \item $\W_{p, \sfd_Y}(\iota_\sharp \mu_0, \iota_\sharp \mu_1) = \W_{p,\sfd_X}(\mu_0, \mu_1)$,
\end{itemize}
where we added a subscript $X$ (resp.~$Y$) to denote the $p$-Hellinger distance in $\meas_+(X)$ (resp.~$\meas_+(Y)$).
\end{proposition}

\begin{proof} 
As in the proof of Lemma \ref{le:contraction}, we denote by $\sfH_Z$ any among $\sfH_{\W_{p,\sfd_Z}}$, $\sfH_{\He_{p,Z}}$ or $\sfH_{\HK_{\sfd_Z}}$ for $Z$ equal to $X$,  $Y$ or $\iota(X) \subset Y$ (again we added a subscript in the notation for $\He_{p}$).
Since $\iota$ is $1$-Lipschitz, we have by Lemma \ref{le:contraction} that
\[ \mathsf{UOT}_{\sfH_Y}(\iota_\sharp \mu_0, \iota_\sharp \mu_1)\le \mathsf{UOT}_{\sfH_X}(\mu_0, \mu_1) \quad \text{ for every } \mu_0, \mu_1 \in \meas_+(X).\]
We are left to prove the reverse inequality. We observe that 
\begin{equation} \label{eq:inclusion}
\supp(\aalpha) \subset \f{C}[\iota(X),\iota(X)] \qquad \textrm{ for every }\aalpha \in \f{H}^2(\iota_\sharp \mu_0, \iota_\sharp \mu_1) \subset \meas_+(\f{C}[Y,Y]).
\end{equation}
Indeed, $\iota$ is an isometry and $(X,\sfd_X)$ is complete, so that it is a closed map, hence $\supp(\iota_\sharp \mu_i) \subset \iota(\supp(\mu_i))\subset \iota(X)$, $i=0,1$. Therefore, \eqref{eq:inclusion} follows if we prove $\supp(\aalpha)\subset V_0\times V_1$, where we have introduced
\[ V_i \coloneqq \{ [y,r], \, y \in \supp(\iota_\sharp\mu_i), \, r > 0 \} \cup \{\f{o}_Y\}, \quad i=0,1. \]
Let $(\f{y}_0, \f{y}_1) \in \supp(\aalpha)$ and let us prove w.l.o.g.~that $\f{y}_0 \in V_0$: if $\f{y}_0=\f{o}_Y$, then $\f{y}_0 \in V_0$ by definition. If instead $\f{y}_0 \ne \f{o}_Y$, then $\f{y}_0 = [y_0, r_0]$ for some $y_0 \in Y$ and some $r_0 >0$. Assume by contradiction, that $y_0 \notin \supp(\iota_\sharp \mu_0)$. Hence we can find a Borel function $\varphi: Y \to [0,1]$ and a radius $\eps>0$ such that $\varphi \equiv 1$ on the open ball $\rmB_{\sfd_Y}(y_0, \eps)$ and
\[ \int_Y \varphi \de (\iota_\sharp \mu_0)=0.\]
However, setting $A\coloneqq \f{p}(\rmB_{\sfd_Y}(y_0, \eps) \times (r_0/2, 3/2 r_0)) \times \f{C}[Y]$, we have
\[ 0= \int_{Y} \varphi \de (\iota_\sharp \mu_0) = \int_{\f{C}[Y,Y]} (\varphi \circ \sfx_0) \sfr_0^2 \de \aalpha \ge \int_{A} (\varphi \circ \sfx_0) \sfr_0^2 \de \aalpha \ge \frac{1}{4} r_0^2 \aalpha(A)>0, \]
where we have used that $A$ is an open neighbourhood of $(\f{y}_0, \f{y}_1) \in \supp(\aalpha)$. The proof of \eqref{eq:inclusion} is thus concluded and we can immediately derive that
\[
\mathsf{UOT}_{\sfH_Y}(\iota_{\sharp}\mu_0,\iota_{\sharp}\mu_1)=\mathsf{UOT}_{\sfH_{\iota(X)}}(\iota_{\sharp}\mu_0,\iota_{\sharp}\mu_1)\quad \text{ for every } \mu_0, \mu_1 \in \meas_+(X)\
\]
by definition of Unbalanced Optimal Transport cost.
We now consider the map $\iota^{-1}$ which is $1$-Lipschitz as a map from $(\iota(X),\sfd_Y)$ to $(X,\sfd_X)$ and thus, using again Lemma \ref{le:contraction}, for every $\mu_0,\mu_1\in \meas_+(X)$ we have
\[ \mathsf{UOT}_{\sfH_X}(\mu_0,\mu_1)=\mathsf{UOT}_{\sfH_X}(\iota^{-1}_{\sharp}(\iota_\sharp \mu_0), \iota^{-1}_{\sharp}(\iota_\sharp \mu_1))\le \mathsf{UOT}_{\sfH_{\iota(X)}}(\iota_{\sharp}\mu_0, \iota_{\sharp}\mu_1)=\mathsf{UOT}_{\sfH_Y}(\iota_{\sharp}\mu_0,\iota_{\sharp}\mu_1)\]
which gives the desired reverse inequality and concludes the proof.
\end{proof}

\section{Marginal entropy-transport problems}\label{sec: Whe}

In this section we assume a metric space $(X, \sfd)$ to be fixed and we use the notation $\He\coloneqq \He_{2}$, $\W\coloneqq \W_{2, \sfd}$, $\HK\coloneqq \HK_{\sfd}$. We devote this section to the study of the energy $\mathcal{E}_1$ in  \eqref{eq:energy} when $N=1$, $\sfc_1  = \He$ and $\sfc_2 = \W$. This leads to the following definition.

\begin{definition}[$\He-\W$ minimizing movement]
We define the marginal $\He-\W$ problem between $\mu_0,\mu_1\in \meas_+(X)$ as 
\begin{equation}\label{def:WHe}
\W\He(\mu_0,\mu_1)\coloneqq \inf_{\nu \in \meas_+(X)}  \He^2(\mu_0,\nu)+\W^2(\nu,\mu_1).
\end{equation}
\end{definition}

The terminology is inspired by \cite[Section 3.3 E.8]{LMS18}, since \eqref{def:WHe} enters in the framework of the marginal Entropy-Transport problems in the sense of Liero, Mielke, and Savar\'e, with transport cost given by $\sfd^2$ and relative entropy generated by the function $M_2$ as in \eqref{eq:mp}.

We collect in the following result a few immediate properties of $\W\He$.

\begin{proposition}\label{prop:whe} The function $\W\He$ has the following properties:
\begin{enumerate}
    \item $\W\He(\mu_0, \mu_1) <+\infty$ for every $\mu_0, \mu_1 \in \meas_+(X)$;
    \item if $\mu_i$ is the null measure, then $\W\He(\mu_0, \mu_1)=\mu_{1-i}(X)$ realized by $\nu=\mu_i$, for $i=0,1$;
    \item introducing $\sfH_{\W\He}:\pc \to [0,+\infty)$ defined as 
    \begin{equation}\label{eq:HWHe}
    \sfH_{\W\He}([x_0,r_0],[x_1,r_1]) \coloneqq \begin{cases} (r_0-r_1)^2 + r_1^2 \sfd^2(x_0,x_1) \quad &\text{ if } r_0>0, \\
    r_1^2 \quad &\text{ if } r_0=0,
    \end{cases}
    \end{equation}
    then $\sfH_{\W\He}$ is a well defined, proper, lower semicontinuous and radially $2$-homogeneous function such that $\W\He = \mathsf{UOT}_{\sfH_{\W\He}}$. In particular, for every pair $(\mu_0,\mu_1) \in \meas_+(X) \times \meas_+(X)$ there exists $\aalpha \in \f{H}^2(\mu_0, \mu_1)$ such that 
    \[ \W\He(\mu_0, \mu_1) = \int_{\pc}\sfH_{\W\He} \de \aalpha;\]
    \item for every pair $(\mu_0, \mu_1) \in \meas_+(X)\times \meas_+(X)$ the set of $\nu \in \meas_+(X)$ realizing the infimum in the definition of $\W\He$ is not empty.
\end{enumerate}
\end{proposition}

\begin{proof}
    Claim (1) follows by considering the competitor $\nu\coloneqq \mu_1$ and the finiteness of $\He$.
    
    Claim (2) is obvious from the definitions of $\He$ and $\W$. 
    
   To prove claim (3) we start by noticing that for $r_1=0$ we have $(r_0-r_1)^2 + r_1^2 \sfd^2(x_0,x_1)=r_0^2$. Since for $r_0=0$ $\sfH_{\W\He}$ is also independent on the spatial variables, this implies that $\sfH_{\W\He}$ is well defined on $\pc.$ It is clear that $\sfH_{\W\He}$ is proper, while the radially $2$-homogeneity can be easily checked since for every $\lambda>0$ we have
   \begin{align*}
   \sfH_{\W\He}([x_0,\lambda r_0],[x_1,\lambda r_1]) &=\begin{cases} (\lambda r_0-\lambda r_1)^2 + \lambda^2 r_1^2 \sfd^2(x_0,x_1) \quad &\text{ if } r_0>0 \\
    \lambda^2 r_1^2 \quad &\text{ if } r_0=0\end{cases} \\
    &=\begin{cases} \lambda^2(r_0-r_1)^2 + \lambda^2 r_1^2 \sfd^2(x_0,x_1) \hspace{6mm} &\text{ if } r_0>0 \\
    \lambda^2 r_1^2 \hspace{6mm} &\text{ if } r_0=0\end{cases}\\
    &=\lambda^2\sfH_{\W\He}([x_0,r_0],[x_1,r_1]).
   \end{align*}
To prove the lower semicontinuity of $\sfH_{\W\He}$, we note that in general
\begin{equation}\label{eq:fromab}
\sfH_{\W\He}([x_0,r_0],[x_1,r_1]) \ge |r_0-r_1|^2 \quad \text{ for every } [x_0,r_0],[x_1,r_1] \in \f{C}[X]    
\end{equation}
with equality if $r_0r_1=0$. We consider any sequence $(\f{y}_0^n, \f{y}_1^n)_n = ([x_0^n, r_0^n],[x_1^n, r_1^n])_n \subset \f{C}[X,X]$ converging to $(\f{y}_0, \f{y}_1)=([x_0,r_0],[x_1,r_1]) \in \f{C}[X,X]$ as $n \to + \infty$. Using the description of the topology on $\f{C}[X,X]$ we can infer that $r_i^n \to r_i$ as $n \to + \infty$, $i=0,1$. Thus, if $r_0r_1=0$, we can use \eqref{eq:fromab} and see that
\[ \liminf_{n \to + \infty} \sfH_{\W\He}(\f{y}_0^n, \f{y}_1^n) \ge \liminf_{n \to + \infty} |r_0^n-r_1^n|^2 =|r_0-r_1|^2 = \sfH_{\W\He}(\f{y}_0, \f{y}_1). \]
If, on the other hand, $r_0r_1>0$, for $n$ large enough we have $r_0^n r_1^n >0$ so that
\begin{align*}
 \lim_{n \to + \infty} \sfH_{\W\He}(\f{y}_0^n, \f{y}_1^n)&= \lim_{n \to + \infty} |r_0^n-r_1^n|^2 + (r_1^n)^2 \sfd^2(x_0^n, x_1^n) \\
 &=|r_0-r_1|^2 + r_1^2 \sfd^2(x_0, x_1) = \sfH_{\W\He}(\f{y}_0, \f{y}_1).    
\end{align*}
We now show that $\mathsf{UOT}_{\sfH_{\W\He}}=\W\He$. We prove separately the two inequalities, starting with $\mathsf{UOT}_{\sfH_{\W\He}}\le \W\He$. Let $\nu \in \meas_+(X)$; if $\W(\nu, \mu_1) = + \infty$  there is nothing to prove, so let us assume $\W(\nu, \mu_1) < + \infty$ and let 
\[\mu_0=\mu_0^{ac}+\mu_0^{\perp} = \vartheta \nu +\mu_0^{\perp}\] be the Lebesgue decomposition of $\mu_0$ w.r.t.~$\nu$. Let $\ggamma \in \Gamma_o(\nu, \mu_1)$ (cf.~Section \ref{sec:wass}) and let us define
\[ \aalpha\coloneqq ([\pi^0, \vartheta \circ \pi^0], [\pi^1, 1])_\sharp \ggamma + ([\text{id}_X,1],\f{o})_\sharp \mu_0^\perp,\]
where $\pi^i$ are the projections from $X \times X$ to $X$ and $\text{id}_X$ is the identity map on $X$. 
It is not difficult to check that $\aalpha \in \f{H}^2(\mu_0, \mu_1)$, so that
\begin{align*}
\mathsf{UOT}_{\sfH_{\W\He}}(\mu_0, \mu_1) &\le \int_{\pc} \sfH_{\W\He} \de \aalpha  \\
&=\int_{X \times X} \sfH_{\W\He}([x_0,\vartheta(x_0)],[x_1,1]) \de \ggamma(x_0,x_1) + \int_X 1 \de \mu_0^\perp \\
&= \int_{\{\vartheta >0\}} \left ( \left |\vartheta \circ \pi^0-1\right |^2 + \sfd^2 \right ) \de \ggamma + \int_{\{\vartheta =0\}} 1 \de \ggamma + \mu_0^\perp(X) \\
&= \int_X \left ( \vartheta-1 \right )^2 \de \nu + \mu_0^\perp(X) + \int_{\{\vartheta>0\}} \sfd^2 \de \ggamma \\\
&\le \He^2(\mu_0, \nu) + \W^2(\nu, \mu_1),   
\end{align*}
where we have also used \eqref{def:eqHe}. This gives the desired inequality since $\nu$ is arbitrary. To prove the converse inequality, let now $\aalpha \in \f{H}^2(\mu_0, \mu_1)$ be optimal for the definition of $\mathsf{UOT}_{\sfH_{\W\He}}(\mu_0, \mu_1)$ (cf.~Theorem \ref{thm:omnibus}(1)) and let us define
\begin{align}\label{eq:defnuopt}
\notag &\mu_1'\coloneqq (\sfx_1)_\sharp (\sfr_1^2 \aalpha \mres{\{\sfr_0=0\}})\,, \quad \ggamma \coloneqq (\sfx_0, \sfx_1)_\sharp (\sfr_1^2 \aalpha \mres{\{\sfr_0>0\}}) + (\operatorname{id}_X, \operatorname{id}_X)_\sharp \mu_1'\,,\\   
&\nu\coloneqq \pi^0_\sharp \ggamma\,, \quad \tilde{\aalpha} \coloneqq ([\sfx_0, \sfr_0],[\sfx_0, \sfr_1])_\sharp (\aalpha \mres{\{\sfr_0>0\}}) + (\f{o}, [\operatorname{id}_X, 1])_\sharp \mu_1'\,.
\end{align}
It is not difficult to check that $\ggamma \in \Gamma(\nu, \mu_1)$ and $\tilde{\aalpha} \in \f{H}^2(\mu_0, \nu)$. Moreover we have
\begin{equation}
    \int_{\f{C}[X,X]} \sfH_{\He} \de \tilde{\aalpha} = \int_{\f{C}[X,X]} |{\sfr_0}-{\sfr_1}|^2 \de \aalpha\,,
\end{equation}
where $\sfH_{\He}$ is as in \eqref{eq:HHe}; indeed:
\begin{align*}
\int_{\f{C}[X,X]} \sfH_{\He} \de \tilde{\aalpha} &= \int_{\{\sfr_0>0\}} \sfH_{\He}([x_0,r_0],[x_0,r_1]) \de \aalpha([x_0,r_0],[x_1,r_1]) + \int_{X} \sfH_{\He}(\f{o},[x,1]) \de \mu_1'(x)  \\
&= \int_{\{\sfr_0>0\}} |\sfr_0-\sfr_1|^2 \de \aalpha + \int_X 1 \de \left ( (\sfx_1)_\sharp (\sfr_1^2 \aalpha \mres{\{\sfr_0=0\}}) \right ) \\
& = \int_{\{\sfr_0>0\}} |\sfr_0-\sfr_1|^2 \de \aalpha + \int_{\{\sfr_0 =0\}} \sfr_1^2 \de \aalpha \\
& = \int_{\f{C}[X,X]} |\sfr_0-\sfr_1|^2 \de \aalpha.
\end{align*}
Then we have
\begin{align*}
    \mathsf{UOT}_{\sfH_{\W\He}}(\mu_0, \mu_1) &= \int_{\pc} \sfH_{\W\He} \de \aalpha \\
    & = \int_{\{\sfr_0>0\}} \left ( |r_0-r_1|^2 + \sfr_1^2 \sfd^2(\sfx_0, \sfx_1) \right ) \de \aalpha + \int_{\{\sfr_0 =0\}} \sfr_1^2 \de \aalpha \\
    & = \int_{\f{C}[X,X]} |r_0-r_1|^2 \de \aalpha + \int_{\{\sfr_0>0\}} \sfr_1^2 \sfd^2(\sfx_0,\sfx_1) \de \aalpha \\
    & = \int_{\f{C}[X,X]} \sfH_{\sfH_{\He}} \de \tilde{\aalpha} + \int_{X \times X} \sfd^2 \de \ggamma\\
    & \ge \He^2(\mu_0, \nu) + \W^2(\nu, \mu_1).
\end{align*}
Moreover, the above computations show that the measures $\nu$ defined in \eqref{eq:defnuopt} satisfies claim (4).
\end{proof}

\begin{remark}
We notice that the function $\sfH_{\W\He}$ defined in \eqref{eq:HWHe} is not the square of a distance on $\f{C}[X]$. Indeed, it is not symmetric since $\sfH_{\W\He}([x_0,r_0],[x_1,r_1])\neq \sfH_{\W\He}([x_1,r_1],[x_0,r_0])$ whenever $r_0=1, r_1=2$ and $\sfd(x_0,x_1)=1$. Also the triangle inequality fails: for this, one can consider a triple of points $[x_0,r_0], [x_1,r_1], [x_2,r_2]$ satisfying the admissible constraints
$$r_0=r_2>0, \quad r_1=0, \quad \sfd(x_0,x_2)\le\sfd(x_0,x_1)+\sfd(x_1,x_2), \quad   \sfd^2(x_0,x_2)>4.$$
It is immediate to see that for these points
$$\sqrt{\sfH_{\W\He}([x_0,r_0],[x_2,r_2])}>\sqrt{\sfH_{\W\He}([x_0,r_0],[x_1,r_1])}+\sqrt{\sfH_{\W\He}([x_1,r_1],[x_2,r_2])}.$$
\end{remark}

\begin{lemma}\label{le:double} Let $\sfH_{\HK}$ be as in \eqref{eq:HHK} and let $\sfH_{\W\He}$ be as in \eqref{eq:HWHe}. Then
\[ \sfH_{\HK} \le \sfd_{\f{C}}^2 \le 2 \sfH_{\W\He}\,.\]
\end{lemma}
\begin{proof} The first inequality is obvious, we only consider the second one.
If $r_0=0$ both functions are equal to $r_1$ so that we have only to consider the case $r_0>0$:
\begin{align*}
    \sfd^2_{\f{C}}([x_0,r_0],[x_1,r_1]) &\stackrel{\eqref{eq:cos_to_sin}}{=} |r_0-r_1|^2 + 4r_0r_1 \sin^2 \left (\frac{\sfd(x_0,x_1) \wedge \pi}{2} \right ) \\
    &\,\,= |r_0-r_1|^2 + 4(r_0-r_1+r_1)r_1 \sin^2 \left (\frac{\sfd(x_0,x_1) \wedge \pi}{2} \right ) \\
    &\,\,= |r_0-r_1|^2 + 2 \left (r_0-r_1 \right ) \left [ 2r_1^2\sin^2 \left (\frac{\sfd(x_0,x_1) \wedge \pi}{2} \right ) \right ] \\
    &\,\,\quad + 4r_1^2 \sin^2 \left (\frac{\sfd(x_0,x_1) \wedge \pi}{2} \right ) \\
    &\,\, \le |r_0-r_1|^2 + |r_0-r_1|^2 + 4r_1^2 \sin^4 \left (\frac{\sfd(x_0,x_1) \wedge \pi}{2} \right ) + r_1^2 \sfd^2(x_0,x_1)\\
    &\,\, \le 2 |r_0-r_1|^2 + 4r_1^2 \sin^2 \left (\frac{\sfd(x_0,x_1) \wedge \pi}{2} \right ) + r_1^2 \sfd^2(x_0,x_1) \\
    &\,\, \le 2 |r_0-r_1|^2 + r_1^2 \sfd^2(x_0,x_1) + r_1^2 \sfd^2(x_0,x_1) \\
    &\,\,= 2\sfH_{\W\He}([x_0,r_0],[x_1,r_1]).
\end{align*}
\end{proof}

Although the following results are not needed in the proof of the equality $\HK = \He \nabla \W$, we believe they are of interest for future investigations on Marginal Entropy-Transport problems and help to better clarify the properties of the $\W\He$ cost. The first result describes minimizers in the definition of $\W\He$ in \eqref{def:WHe}.

\begin{proposition}\label{prop:structure} Let $\mu_0, \mu_1 \in \meas_+(X)$ and let $\nu$ be a minimizer for $\W\He$, i.e.~such that
\[ \He^2(\mu_0, \nu) + \W^2(\nu, \mu_1) = \W\He(\mu_0, \mu_1).\]
Let $\nu=\nu^{ac}+\nu^{\perp}$ be the Lebesgue decomposition of $\nu$ w.r.t.~$\mu_0$. Then for every $\ggamma \in \Gamma_o(\nu, \mu_1)$ (that is, $\ggamma$ is optimal for $\W(\nu, \mu_1)$) we have that its disintegration $\{\ggamma_x\}_{x \in X}$ w.r.t.~$\nu$ satisfies
\[ \ggamma_x = \delta_x \quad \text{ for $\nu^\perp$-a.e.~$x \in X$}.\]
\end{proposition}

\begin{remark}
This is to say that the singular part of $\nu$ w.r.t.~$\mu_0$ does not need to be moved when matching $\nu$ to $\mu_1$ in an optimal way.
\end{remark}

\begin{proof} Let $\ggamma \in \Gamma_o(\nu, \mu_1)$; we assume by contradiction that there exists a Borel set $A \subset X$ such that $\nu^\perp(A)>0$ and $\ggamma_x \ne \delta_x$ for every $x \in A$. Let us denote by $\vartheta$ the density of $\nu^{ac}$ w.r.t.~$\mu_0$ and let us define
    \[ \tilde{\mu_1}\coloneqq \pi^1_\sharp \left ( \int_X \ggamma_x \de \nu^{\perp}(x)  \right ), \quad \tilde{\nu}\coloneqq \nu^{ac} + \tilde{\mu_1}, \quad \tilde{\mu_1} =: \vartheta' \mu_0 + \tilde{\mu_1}^\perp, \quad \tilde{\ggamma}\coloneqq \int_X \ggamma_x \de \nu^{ac}(x) + \int_X \delta_y \de \tilde{\mu_1}(y).\]
    We have for every $\varphi \in \rmC_b(X)$ that
    \begin{align*}
        \int_X \varphi \de (\pi^0_\sharp \tilde{\ggamma}) = \int_{X \times X} (\varphi \circ \pi^0) \de \tilde{\ggamma} &= \int_{X} \int_X \varphi(x) \de \ggamma_x(y) \de \nu^{ac}(x) + \int_X \int_X \varphi(x) \de \delta_y(x) \de \tilde{\mu_1}(y) \\
        &= \int_X \varphi \de \nu^{ac} + \int_X \varphi \de \tilde{\mu_1} 
         = \int_X \varphi \de \tilde{\nu},\\
        \int_X \varphi \de (\pi^1_\sharp \tilde{\ggamma}) = \int_{X \times X} (\varphi \circ \pi^1) \de \tilde{\ggamma} &= \int_{X} \int_X \varphi(y) \de \ggamma_x(y) \de \nu^{ac}(x) + \int_X \int_X \varphi(y) \de \delta_y(x) \de \tilde{\mu_1}(y) \\
        & = \int_{X} \int_X \varphi(y) \de \ggamma_x(y) \de \nu^{ac}(x) + \int_X \varphi \de \tilde{\mu_1} \\
        &= \int_{X} \int_X \varphi(y) \de \ggamma_x(y) \de \nu^{ac}(x) + \int_X \int_X \varphi(y) \de \ggamma_x(y) \de \nu^\perp(x) \\
        & = \int_{X \times X} \varphi(y) \de \ggamma(x,y) = \int_X \varphi \de \mu_1.
    \end{align*}
    We deduce that $\tilde{\ggamma} \in \Gamma(\tilde{\nu}, \mu_1)$. Notice that 
    \[ \int_X \sfd^2(x,y) \de \ggamma_x(y) = \int_{ \{y \ne x\}} \sfd^2(x,y) \de \ggamma_x(y) >0 \quad \text{ for every } x \in A,\]
    since $\ggamma_x( \{ y \ne x \}) = \ggamma_x(\{ y: \sfd^2(x,y)>0\})>0$ being $\ggamma_x \ne \delta_x.$ We deduce that
     \[ \int_X \int_X \sfd^2(x,y) \de \ggamma_x(y) \de \nu^{\perp}(x) \ge \int_A \int_X \sfd^2(x,y) \de \ggamma_x(y) \de \nu^\perp(x)>0.\]
     Therefore, using \eqref{def:eqHe}, we have
    \begin{align*}
     \He^2(\mu_0, \tilde{\nu})+\W^2(\tilde{\nu},\mu_1) & \le \int_X \left (1-\sqrt{\vartheta +\vartheta' } \right )^2 \de \mu_0 + \tilde{\mu_1}^\perp (X) +\int_{X \times X} \sfd^2 \de \tilde{\ggamma} \\
     & = \int_X \left ( 1+ \vartheta + \vartheta' -2\sqrt{\vartheta+\vartheta'} \right ) \de \mu_0 + \tilde{\mu_1}^\perp (X) \\
     & \quad + \int_X \int_X \sfd^2(x,y) \de \ggamma_x(y) \de \nu^{ac}(x) + \int_X \int_X \sfd^2(x,y) \de \delta_y(x) \de \tilde{\mu_1}(y) \\
     & \le \int_X \left ( 1+ \vartheta + \vartheta' -2\sqrt{\vartheta} \right ) \de \mu_0 +\tilde{\mu_1}^\perp (X) + \int_X \int_X \sfd^2(x,y) \de \ggamma_x(y) \de \nu^{ac}(x) \\
     & = \int_X \left ( 1+ \vartheta + \vartheta' -2\sqrt{\vartheta} \right ) \de \mu_0 +\tilde{\mu_1}^\perp (X) \\
     & \quad + \int_X \int_X \sfd^2(x,y) \de \ggamma_x(y) \de \nu(x) -\int_X \int_X \sfd^2(x,y) \de \ggamma_x(y) \de \nu^{\perp}(x) \\
     & < \int_X \left ( 1+ \vartheta + \vartheta' -2\sqrt{\vartheta} \right ) \de \mu_0 +\tilde{\mu_1}^\perp (X) + \int_X \int_X \sfd^2(x,y) \de \ggamma_x(y) \de \nu(x)\\
     & \le \int_X \left (1-\sqrt{\vartheta} \right )^2 \de \mu_0 + \int_X \vartheta' \de \mu_0 + \tilde{\mu_1}^\perp(X) + \int_X \int_X \sfd^2(x,y) \de \ggamma_x(y) \de \nu(x) \\
     & = \int_X \left (1-\sqrt{\vartheta} \right )^2 \de \mu_0 + \tilde{\mu_1}(X) + \W(\nu, \mu_1) \\
     & =\int_X \left (1-\sqrt{\vartheta} \right )^2 \de \mu_0 +   \W(\nu, \mu_1) + \int_X \int_X 1 \de \ggamma_x(y) \de \nu^\perp(x) \\
     & = \int_X \left (1-\sqrt{\vartheta} \right )^2 \de \mu_0 + \nu^\perp(X) + \W(\nu, \mu_1) \\
     & = \He^2(\mu_0, \nu) + \W^2(\nu, \mu_1).
    \end{align*}
This contradicts the minimality of $\nu$.
\end{proof}

Thanks to Proposition \ref{prop:structure} we can explicitly compute the optimal $\nu$ in case of Dirac measures.

\begin{proposition}\label{prop:atomic} Let $x_0,x_1 \in X$, $r_0, r_1 \in [0,+\infty)$ and let $\mu_0\coloneqq r_0 \delta_{x_0}$ and $\mu_1\coloneqq r_1 \delta_{x_1}$. Then the unique optimal $\nu$ for $\W\He(\mu_0, \mu_1)$ is given by
\[ \nu= s_0 \delta_{x_0} +(r_1-s_0) \delta_{x_1}, \quad s_0 \coloneqq \begin{cases} r_1 \wedge \frac{r_0}{ \sfd^4(x_0,x_1)} \quad &\text{ if } x_0 \ne x_1, \\
r_1 \quad & \text{ if } x_0=x_1.\end{cases}\]
In particular
\[  \W\He(r_0 \delta_{x_0}, r_1 \delta_{x_1})= \begin{cases} (\sqrt{r_0}-\sqrt{r_1})^2 + r_1 \sfd^2(x_0, x_1) \quad &\text{ if } r_1\sfd^4(x_0,x_1) \le r_0, \\
r_0+r_1- \frac{r_0}{\sfd^2(x_0,x_1)} \quad & \text{ if }r_1 \sfd^4(x_0,x_1)>r_0.\end{cases} \]
\end{proposition}
\begin{proof} We know that in case $r_0=0$ or $r_1=0$ the minimizer is $\nu\coloneqq \mu_1$; let us assume that both $r_0$ and $r_1$ are strictly larger than $0$. Let $\nu$ be a minimizer for $\W\He(\mu_0, \mu_1)$, let $(\nu^{ac},\nu^\perp)$ be its Lebesgue decomposition w.r.t.$\mu_0$ and notice that $\nu^{ac}=r \delta_{x_0}$ for some $r \in [0,r_1]$, since $\nu^{ac} \ll \mu_0$ and $\nu(X)=\mu_1(X)=r_1$. The unique optimal transport plan $\ggamma \in \Gamma(\nu, \mu_1)$ is simply $\ggamma = \nu \otimes \delta_{x_1}$ so that its disintegration $\{\ggamma_x\}_{x \in X}$ w.r.t.~$\nu$ is simply given by $\ggamma_x= \delta_{x_1}$ for every $x \in X$. By Proposition \ref{prop:structure}, we know that $\nu^\perp$ is concentrated on the set 
    \[ \{ x \in X : \ggamma_x = \delta_x \} = \{ x \in X : \delta_{x_1} = \delta_x\} = \{x_1\}.\]
    We deduce that $\nu^\perp = (r_1-r)\delta_{x_1}$. In case $x_0=x_1$ we thus get $s_0=r_1$; let us assume that $x_0 \ne x_1$: we can explicitly compute $\He^2(\mu_0, \nu)$ and $\W^2(\nu, \mu_1)$:
    \begin{align*}
        \He^2(\mu_0, \nu) &= \int_X \left ( 1 - \sqrt{\frac{\de \nu^{ac}}{\de \mu_0}} \right )^2 \de \mu_0 + \nu^\perp(X) = r_0 (1-\sqrt{r/r_0})^2 + (r_1-r) = r_0 + r_1 -2\sqrt{r_0 r},\\
        \W^2(\nu, \mu_1) &= \int_{X \times X} \sfd^2 \de \ggamma = \int_X \sfd^2(x,x_1) \de \nu(x)= r\sfd^2(x_0,x_1).
    \end{align*}
    We are then left with the minimization problem 
    \[ \text{ minimize } \quad r \sfd^2(x_0,x_1) - 2\sqrt{r_0 r} \quad \text{ subject to } \quad 0 \le r \le r_1,\]
    which gives the desired result.
\end{proof}

\begin{remark}
By Theorem \ref{thm:omnibus}(2), we have that the lower semicontinuous and convex envelope of the $1$-homogeneous version of $\sfH_{\W\He}$ is thus given by
\begin{equation}\label{eq:convenv}
 \cce{\sfH_{\W\He}}([x_0,r_0],[x_1,r_1]) = \begin{cases} (\sqrt{r_0}-\sqrt{r_1})^2 + r_1 \sfd^2(x_0, x_1) \quad &\text{ if } r_1\sfd^4(x_0,x_1) \le r_0, \\
r_0+r_1- \frac{r_0}{\sfd^2(x_0,x_1)} \quad & \text{ if }r_1 \sfd^4(x_0,x_1)>r_0.\end{cases} 
\end{equation}
The same result could be obtained directly: let $x_0, x_1 \in X$ and set $d\coloneqq \sfd(x_0,x_1) \ge 0$; by Fenchel--Moreau duality Theorem, we need to compute the (restriction to $[0,+\infty)^2$ of the) biconjugate of the function
\[ f_d(r_0, r_1) \coloneqq \begin{cases} (\sqrt{r_0}-\sqrt{r_1})^2 + r_1 d^2 \nchi_{(0,+\infty)}(r_0) \quad &\text{ if } r_0, r_1 \ge 0, \\ + \infty \quad &\text{ else}. \end{cases}\]
Since $f_d$ is $1$-homogeneous, it is not difficult to see that 
\[ f^*_d(s_0, s_1) \coloneqq \sup \left \{ r_0s_0+r_1s_1 -(\sqrt{r_0}-\sqrt{r_1})^2 + r_1 d^2 \nchi_{(0,+\infty)}(r_0) : r_0, r_1 \ge 0\right \} = I_{C_d}, \]
that is, $f^*_d$ is the indicator of the set 
\[ C_d \coloneqq \left \{ (s_0, s_1) : s_0<1, s_1 \le 1, s_1-1-d^2 \le \frac{1}{s_0-1} \right\}.\]
Therefore the biconjugate function $f^{**}_d$ is obtained by computing 
\[ f^{**}_d(r_0, r_1) \coloneqq \sup \left \{ r_0s_0+r_1s_1 : (s_0, s_1) \in C_d \right \} =  r_0+r_1 - \inf \{ r_0 x + r_1 y : x>0, y \ge 0, y + d^2 \ge 1/x \}. \]
The infimum above is clearly equal to $0$ in case $r_1r_0=0$, so that we can restrict to the case $r_0r_1>0$;
in this case, the constrained minimization problem in two variables is solved looking at the behaviour of the objective function on the boundary of the domain of minimization, that is, one has to minimize the function $b:(0, +\infty) \to \R$ given by 
\[ b(x)= \begin{cases} r_0x + \frac{r_1}{x} - r_1d^2 \quad &\text{ if } 0<x \le \frac{1}{d^2}, \\ r_0x \quad &\text{ if } x \ge \frac{1}{d^2}, \end{cases}\]
where we mean that, in case $d=0$, only the first case occurs. In case $d=0$, we have that the minimum is attained at $x=\sqrt{r_1/r_0}$ equal to $2\sqrt{r_0r_1}$. In case $d>0$, we have two cases:
\begin{enumerate}
    \item If $r_1 d^4 \le r_0$ , then $b$ is decreasing in $(0, \sqrt{r_1/r_0})$ and increasing in $(\sqrt{r_1/r_0}, + \infty)$. Thus the minimum is attained at $x= \sqrt{r_1/r_0}$ equal to $2\sqrt{r_0r_1} - r_1d^2$.
    \item  If $r_1 d^4 >r_0$, then $b$ is decreasing in $(0, 1/d^2)$ and increasing in $(1/d^2, + \infty)$. Thus the minimum is attained at $x= 1/d^2$ equal to $r_0/d^2$.
\end{enumerate}
Putting together this information, one obtains precisely the expression in \eqref{eq:convenv} for $f^{**}_d$.
\end{remark}

\section{\texorpdfstring{$\He \nabla \W\le \HK$}{HK > He d W2}}\label{sec:ineq1}

In this section we assume $(X, \sfd)$ to be a geodesic, complete and separable metric space;
we adopt the same notation $\He, \W, \HK, \W\He$ of the previous section, removing the dependence on the exponent $p=2$ and on the distance $\sfd$ when this does not cause confusion. In this section we also use the notation $\mathcal{E}_N$ as in \eqref{eq:energy} meaning that $\sfc_1= \He$ and $\sfc_2 = \W$.

Let $\mu_0, \mu_1 \in \meas_+(X)$ be fixed; we aim to show that 
\begin{equation}\label{eq:aimsec6}
(\He \nabla \W)^2(\mu_0, \mu_1) \le \HK^2(\mu_0, \mu_1).    
\end{equation}
By Theorem \ref{thm:toolbox}, given a (constant speed, length minimizing) $\HK$-geodesic $(\mu_t)_{t \in [0,1]}$ connecting $\mu_0$ to $\mu_1$, we can find a plan $\ppi \in \prob({\rm Geo}((\f{C}[X], \sfd_{\f{C})}))$ such that
\begin{equation}\label{eq:ppi}
\HK^2(\mu_0, \mu_1) = \int \mathcal{A}_2 \de \ppi,   
\end{equation}
where $\mathcal{A}_2$ is as in \eqref{eq:a2def}; therefore, the natural strategy is to consider, for every $N \in \N_{> 1}$, a family of $N-1$ equidistributed points along $(\mu_t)_{t \in [0,1]}$ and, then, to construct an $N$-path by interpolating those points with $\W\He$ minimizers: we aim then to show that the corresponding energy $\mathcal{E}_N$ can be controlled in the limit by $\HK^2(\mu_0, \mu_1)$. More in detail, we set
\begin{equation}\label{eq:sigmai}
\sigma_i^{N} \coloneqq  \f{h}^2((\mathsf{e}_{i/N})_\sharp\ppi)\in \meas_+(X), \qquad i=0,\dots,N    
\end{equation}
and
\[
\aalpha_i^{N} \coloneqq (\mathsf{e}_{\frac{i-1}{N}},\mathsf{e}_{\frac{i}{N}})_\sharp \ppi\in \meas_+(\pc), \qquad i=1,\dots,N
\]
and, for every $i=1, \dots, N$, we select $\nu_i^{N} \in \meas_+(X)$ minimizer for $\W\He(\sigma_{i-1}^{N},\sigma_i^{N})$, whose existence is given by Proposition \ref{prop:whe}(4). We thus define $P_{N} \coloneqq (\sigma_0^{N}, \sigma_1^{N}, \dots \sigma_N^{N}; \nu_1^{N}, \dots, \nu_N^{N})$ and note that $P_{N} \in \mathscr{P}(\mu_0, \mu_1; N)$. 

Observe that it holds
\begin{equation*}
\begin{split}
\mathcal{E}_N(P_{N}) & = N\sum_{i=1}^N\left(\He^2(\sigma^{N}_{i-1},\nu^{N}_i) + \W^2(\nu^{N}_i,\sigma^{N}_i)\right) \\
& = N\sum_{i=1}^N \W\He(\sigma^{N}_{i-1},\sigma^{N}_i) \leq N\sum_{i=1}^N \int_{\pc} \sfH_{\W\He}\de\aalpha^{N}_i,
\end{split}
\end{equation*}
where $\sfH_{\W\He}$ is as in \eqref{eq:HWHe}. By definition of infimal convolution, therefore, the final inequality \eqref{eq:aimsec6} is achieved if we prove that 
\begin{equation}\label{eq:intermediatee}
(\He \nabla \W)^2(\mu_0, \mu_1) \le \limsup_{N \to + \infty} N\sum_{i=1}^N \int_{\pc} \sfH_{\W\He}\de\aalpha^{N}_i  \le \int \mathcal{A}_2 \de \ppi,    
\end{equation}
where we also used \eqref{eq:ppi}. Rewriting the term in the middle in terms of $\ppi$ we have
\begin{align*}
N\sum_{i=1}^N \int_{\pc} \sfH_{\W\He}\de\aalpha^{N}_i &= \int N\sum_{i=1}^N \sfH_{\W\He} \left ( \f{y} \left (\frac{i-1}{N} \right ), \f{y} \left (\frac{i}{N} \right ) \right ) \de\ppi(\f{y}) \\ 
& \le \iint_0^1 \left [ |r_\f{y}'(t)|^2 + \left| r_\f{y}\left(\frac{\lceil tN \rceil}{N} \right) \right|^2 |x_\f{y}'(t)|_\sfd^2 \right ]\de t\de\ppi(\f{y}) \\
& =: \iint_0^1 D_2^N(t, \f{y}) \de t \de \ppi=:I_N,
\end{align*}
where we have employed the explicit expression of $\sfH_{\W\He}$ (cf.~\eqref{eq:HWHe}), and we have also used Jensen's inequality (more details will be given in the proof of Theorem \ref{thm:firstineq}). What makes $D_2^N$ different from the metric derivative on the cone (cf.~\eqref{eq:speed-cone}) is the fact that the term multiplying  $|x_\f{y}'(t)|_\sfd^2$ should be replaced by $r_\f{y}(t)^2$, but one can infer that, as $N \to + \infty$ and by continuity of $\f{y}$, the latter will indeed appear, so that 
\[ \limsup_{N \to + \infty} I_N = \limsup_{N \to + \infty} \iint_0^1 D_2^N(t, \f{y}) \de t \de \ppi(\f{y}) \le \iint_0^1 \left [ |r_\f{y}'(t)|^2 + r_{\f{y}}^2(t) |x_\f{y}'(t)|_\sfd^2 \right ]\de t \de \ppi =  \int \mathcal{A}_2 \de \ppi, \]
which leads precisely to \eqref{eq:intermediatee}. To make this strategy rigorous, however, one would need a (uniformly in $N$) integrable control on the ratio 
\begin{equation}\label{eq:ratio}
\frac{r_\f{y}\left(\frac{\lceil tN \rceil}{N} \right)}{r_\f{y}(t)}, 
\end{equation}
in order to apply e.g.~the Lebesgue dominated convergence theorem (notice that a constant bound on $r_\f{y}\left(\frac{\lceil tN \rceil}{N} \right)$ would not suffice since we do not know whether $(\f{y},t) \mapsto |x_\f{y}'(t)|_\sfd^2$ belongs to $L^1(\ppi \otimes \mathcal{L}^1 \mres [0,1])$). To control the ratio in \eqref{eq:ratio}, we consider the set of `non-degenerate' geodesics
\begin{equation}\label{eq:seta}
A \coloneqq \left\{ \f{y} \in {\rm Geo}((\f{C}[X], \sfd_{\f{C}})) \,:\, 0 < \sfd(x_\f{y}(0),x_\f{y}(1)) \leq \frac{\pi}{2} \,\textrm{ and }\, r_\f{y}(0) r_\f{y}(1) > 0 \right\}
\end{equation}
and, for any $\eps>0$, that of `$\eps$-good' geodesics
\begin{equation}\label{eq:setae}
A_\eps \coloneqq \left\{\f{y} \in A \,:\, \frac{|r_\f{y}(t)|}{|r_\f{y}(s)|} \leq 2 \textrm{ for all } t,s \in [0,1] \textrm{ such that } |t-s| < \eps \right\}.
\end{equation}
Therefore, on the set $A_\eps$ the dominated convergence theorem can be applied, obtaining 
\begin{equation}\label{eq:thefirste}
\limsup_{N\to +\infty} \int_{A_\eps} \int_0^1 D_2^N(t, \f{y}) \de t \de \ppi(\f{y}) \le \int_{A_\eps} \mathcal{A}_2 \de \ppi.    
\end{equation}
We will observe that on the set ${\rm Geo}((\f{C}[X], \sfd_{\f{C}}))\setminus A$ both the metric derivative and $D_2^N$ are equal to $|r'_{\f{y}}(t)|^2$ so that no limit in $N$ is actually necessary, hence
\begin{equation}\label{eq:theseconde}
\limsup_{N\to +\infty} \int_{{\rm Geo}((\f{C}[X], \sfd_{\f{C}}))\setminus A} \int_0^1 D_2^N(t, \f{y}) \de t \de \ppi(\f{y}) = \int_{{\rm Geo}((\f{C}[X], \sfd_{\f{C}}))\setminus A} \mathcal{A}_2 \de \ppi.    
\end{equation}
It remains to control the integral of $D_2^N$ on the set $A \setminus A_\eps$: unfortunately, no uniformly integrable control of $D_2^N$ is available on this set, so that we need to modify the plan $\ppi$ into an `approximated' plan $\ppi_\eps$ defined as
\begin{equation}\label{eq:pi-eps}
\ppi_\eps \coloneqq \ppi \mres (A_\eps \cup ({\rm Geo}((\f{C}[X], \sfd_{\f{C}}))\setminus A)) + T_\sharp (\ppi \mres (A \setminus A_\eps)),
\end{equation}
where 
$T: {\rm Geo}((\f{C}[X], \sfd_{\f{C}})) \to \AC^2([0,1];(\f{C}[X],\sfd_{\f{C}}))$, $\f{y} \mapsto T(\f{y})$, is defined as
\[
T(\f{y})(t) \coloneqq 
\left\{
\begin{array}{ll}
[x_\f{y}(0),(1-2t)r_\f{y}(0)] & \textrm{if } t \in [0,1/2], \\[5pt]

[x_\f{y}(1),(2t-1)r_\f{y}(1)] & \textrm{if } t \in (1/2,1].
\end{array}
\right.
\]
We can thus repeat the above construction from \eqref{eq:sigmai} to \eqref{eq:theseconde} word by word replacing $\ppi$ with $\ppi_\eps$ (and thus $\sigma_i^N$ with $\sigma_i^{N, \eps}$, $\aalpha_i^N$ with $\aalpha_i^{N, \eps}$, $\nu_i^N$ with $\nu_i^{N,\eps}$, $P_N$ with $P_{N, \eps}$, and $I_N$ with $I_N^\eps$): the argument to reach the inequalities \eqref{eq:thefirste} and \eqref{eq:theseconde} on the sets $A_\eps$ and ${\rm Geo}((\f{C}[X], \sfd_{\f{C}}))\setminus A)$, respectively, is indeed the same. On the other hand, on the set $T(A\setminus A_\eps)$, we can show that $D_2^N(t, \f{y})$ is bounded by $|r'_{\f{y}}(t)|^2 + o(N)$, and that the metric derivative coincides with $|r'_{\f{y}}(t)|^2$, therefore obtaining also
\[ \limsup_{N \to + \infty} \int_{T(A \setminus A_\eps)} \int_0^1 D_2^N(t, \f{y}) \de t \de \ppi_\eps(\f{y}) \le \int_{T(A \setminus A_\eps)} \mathcal{A}_2 \de \ppi_\eps. \]
Putting together this latter inequality with \eqref{eq:thefirste} and \eqref{eq:theseconde} written for $\ppi_\eps$ one finally arrives to 
\[ 
(\He \nabla \W)^2(\mu_0, \mu_1) \le \limsup_{N \to + \infty} I_N^\eps = \limsup_{N \to + \infty} \iint_0^1 D_2^N(t, \f{y}) \de t \de \ppi_\eps \le \int \mathcal{A}_2 \de \pi_\eps.
\]
It is then sufficient to show that $\lim_{\eps \downarrow 0} \int \mathcal{A}_2 \de \ppi_\eps = \int \mathcal{A}_2 \de \ppi$, see Lemma \ref{lem:convergence-2action-pi}.

We now detail the above strategy: as said at the beginning, we fix measures $\mu_0, \mu_1 \in \meas_+(X)$ and we select a (constant speed, length-minimizing) $\HK$-geodesic $(\mu_t)_{t \in [0,1]}$ connecting $\mu_0$ to $\mu_1$. By Theorem \ref{thm:toolbox} there exists $\ppi \in \prob({\rm Geo}((\f{C}[X], \sfd_{\f{C}})))$ satisfying conditions (a)-(d) in the statement of the same theorem.

Let us point out that the definitions of $A$ and $A_\eps$ in \eqref{eq:seta} and \eqref{eq:setae} are meaningful  since for a geodesic $\f{y} \in A$ the radius $r_{\f{y}}(t)$ is positive at any $t \in [0,1]$, as explained in Remark \ref{rem:geocone}. As a consequence of the same observation we have that $A_\eps \subset A_{\eps'} \subset A$ whenever $0<\eps'<\eps$ and $\cup_{\eps>0}A_\eps=A$. Finally, we briefly remark that both $A$ and $A_\eps$ are Borel sets: indeed, $A$ can be expressed as intersection of the pre-images of the intervals $(0,\frac{\pi}{2}]$ and $(0,+\infty)$ via the Borel maps $\sfd(\sfx \circ \sfe_0,\sfx \circ \sfe_1)$ and $\sfr \circ \sfe_0 \cdot \sfr \circ \sfe_1$ respectively; on the other hand, since $t \mapsto r_\f{y}(t)$ is continuous on $[0,1]$, 
\[
A_\eps = \bigcap_{\substack{t,s \in \mathbb{Q} \cap [0,1] \\ |t-s|<\eps}}\left\{\f{y} \in A \,:\, \frac{|r_\f{y}(t)|}{|r_\f{y}(s)|} \leq 2 \right\}
\]
and each set on the right-hand side is Borel being the pre-image of $[0,2]$ via a well-defined (by Remark \ref{rem:geocone}) Borel map, thus settling the measurability of $A_\eps$.

As a first step, let us prove that $\ppi_\eps$ (when projected) still joins $\mu_0$ to $\mu_1$.

\begin{lemma} Let $\ppi_\eps$ be defined as in \eqref{eq:pi-eps}; then 
$\ppi_\eps \in \prob(\AC^2([0,1];(\f{C}[X],\sfd_{\f{C}})))$ and $\f{h}^2((\mathsf{e}_i)_\sharp \ppi_\eps) = \mu_i$ for $i=0,1$.
\end{lemma}

\begin{proof}
The fact that $\ppi_\eps$ is concentrated on $\AC^2([0,1]; (\f{C}[X],\sfd_{\f{C}}))$ follows by the fact that $\ppi$ is concentrated on ${\rm Geo}((\f{C}[X], \sfd_{\f{C}}))$ and \eqref{eq:AC-cone}. Moreover, since the total mass is preserved by push-forward, we also see that $\ppi_\eps (\AC^2([0,1];(\f{C}[X],\sfd_{\f{C}})))=1$. 

As for $\f{h}^2((\mathsf{e}_i)_\sharp \ppi_\eps) = \mu_i$, fix $\varphi \in \rmC_b(X)$ and note that, for $i=0,1$,
\[
\begin{split}
\int_X \varphi\de\f{h}^2((\mathsf{e}_i)_\sharp \ppi_\eps) & = \int_{\f{C}[X]} (\varphi \circ \sfx)\sfr^2 \de(e_i)_\sharp \ppi_\eps = \int \varphi(x_\f{y}(i)) r_\f{y}^2(i) \de\ppi_\eps(\f{y}) \\
& = \int_{A_\eps \cup A^c} \varphi(x_\f{y}(i)) r_\f{y}^2(i) \de\ppi(\f{y}) + \int \varphi(x_\f{y}(i)) r_\f{y}^2(i)\de T_\sharp(\ppi \mres (A \setminus A_\eps))(\f{y}) \\
& = \int_{A_\eps \cup A^c} \varphi(x_\f{y}(i)) r_\f{y}^2(i) \de\ppi(\f{y}) + \int_{A \setminus A_\eps} \varphi(x_{T(\f{y})}(i)) r_{T(\f{y})}^2(i)\de \ppi(\f{y}) \\
& = \int_{A_\eps \cup A^c} \varphi(x_\f{y}(i)) r_\f{y}^2(i) \de\ppi(\f{y}) + \int_{A \setminus A_\eps} \varphi(x_\f{y}(i)) r_\f{y}^2(i)\de \ppi(\f{y}) \\
& = \int \varphi(x_\f{y}(i)) r_\f{y}^2(i) \de\ppi(\f{y}) = \int_{\f{C}[X]}(\varphi \circ \sfx)\sfr^2 \de(\mathsf{e}_i)_\sharp \ppi \\
& = \int_X \varphi \de\f{h}^2((\mathsf{e}_i)_\sharp \ppi) = \int_X \varphi\de\mu_i,
\end{split}
\]
where on the fourth line we used the fact that $x_{T(\f{y})}(i) = x_\f{y}(i)$ and $r_{T(\f{y})}(i) = r_\f{y}(i)$ by construction and on the last line the fact that $\f{h}^2((\mathsf{e}_i)_\sharp \ppi) = \mu_i$.
\end{proof}

\begin{lemma}\label{lem:convergence-2action-pi}
Let $\mathcal{A}_2 : \rmC([0,1];(\f{C}[X],\sfd_\f{C})) \to [0,+\infty]$ be defined as in \eqref{eq:a2def} and let $\ppi_\eps$ be as in \eqref{eq:pi-eps}.  Then
\[
\lim_{\eps \downarrow 0}\int \mathcal{A}_2\de\ppi_\eps = \int \mathcal{A}_2 \de\ppi.
\]
\end{lemma}

\begin{proof} Since $\ppi$ and $\ppi_\eps$ are concentrated on $\AC^2([0,1];(\f{C}[X],\sfd_\f{C}))$, $\mathcal{A}_2$ only takes finite values on the support of these probabilities.
We then rewrite
\[
\int \mathcal{A}_2\de\ppi_\eps = \int_{A_\eps \cup A^c} \mathcal{A}_2\de\ppi + \int \mathcal{A}_2\de T_\sharp(\ppi \mres (A \setminus A_\eps))
\]
and observe that the first term on the right-hand side converges to $\int \mathcal{A}_2 \de\ppi$ as $\eps \downarrow 0$ by monotone convergence: indeed, $\mathcal{A}_2 \geq 0$ and $A_\eps \uparrow A$ as $\eps \downarrow 0$.

Thus, we are only left to prove that the second term vanishes as $\eps \downarrow 0$. To this end, noticing that $T(\f{y}) \in \AC^2([0,1]; (\f{C}[X],\sfd_{\f{C}}))$ with $|r_{T(\f{y})}'(t)| = 2|r_\f{y}(0)|$ on $[0,1/2]$, $|r_{T(\f{y})}'(t)| = 2|r_\f{y}(1)|$ on $(1/2,1]$, and $|x_{T(\f{y})}'(t)|_\sfd = 0$ a.e.\ in $[0,1]$, yields
\[
\begin{split}
\int \mathcal{A}_2\de T_\sharp(\ppi \mres (A \setminus A_\eps)) & = \int_{A \setminus A_\eps}\int_0^1 \Big(|r_{T(\f{y})}'(t)|^2 + |r_{T(\f{y})}(t)|^2|x_{T(\f{y})}'(t)|^2_\sfd\Big)\de t \de\ppi(\f{y}) \\
& = 2\int_{A \setminus A_\eps}\big(|r_\f{y}(0)|^2 + |r_\f{y}(1)|^2\big)\de\ppi(\f{y}) \leq 4R^2 \ppi(A \setminus A_\eps),
\end{split}
\]
where in the last inequality we used the fact that, by Theorem \ref{thm:toolbox}(c), $(\mathsf{e}_0)_\sharp \ppi, (\mathsf{e}_1)_\sharp \ppi$ are concentrated on $\f{C}_R[X]$. The fact that $\ppi(A \setminus A_\eps) \to 0$ as $\eps \downarrow 0$ then follows again by $A_\eps \uparrow A$ as $\eps \downarrow 0$.
\end{proof}

We are now in the position to prove the main result of this section.

\begin{theorem}\label{thm:firstineq} Let $(X,\sfd)$ be a complete, separable and geodesic metric space and let $\mu_0, \mu_1 \in \meas_+(X)$. Then
\[ (\He_{2} \nabla \W_{2, \sfd})(\mu_0, \mu_1) \le \HK_{\sfd}(\mu_0, \mu_1).\]
\end{theorem}

\begin{proof}
Fix $\eps>0$ and consider $\ppi_\eps$ as defined in \eqref{eq:pi-eps}. For every $N \in \N_{\ge 1}$ with $N > 1/\eps$ we also define
\[
\sigma_i^{N,\eps} \coloneqq  \f{h}^2((\mathsf{e}_{i/N})_\sharp\ppi_\eps)\in \meas_+(X), \qquad i=0,\dots,N
\]
and
\[
\aalpha_i^{N,\eps} \coloneqq (\mathsf{e}_{\frac{i-1}{N}},\mathsf{e}_{\frac{i}{N}})_\sharp \ppi_\eps\in \meas_+(\pc), \qquad i=1,\dots,N
\]
and, for every $i=1, \dots, N$, we select $\nu_i^{N,\eps} \in \meas_+(X)$ minimizer for $\W\He(\sigma_{i-1}^{N,\eps},\sigma_i^{N,\eps})$, whose existence is given by Proposition \ref{prop:whe}(4). We thus define $P_{N,\eps} \coloneqq (\sigma_0^{N,\eps}, \sigma_1^{N,\eps}, \dots \sigma_N^{N,\eps}; \nu_1^{N,\eps}, \dots, \nu_N^{N,\eps})$ and note that $P_{N,\eps} \in \mathscr{P}(\mu_0, \mu_1; N)$. 

After this premise, we start observing that by construction
\begin{equation}\label{eq:bounding_EN}
\begin{split}
\mathcal{E}_N(P_{N,\eps}) & = N\sum_{i=1}^N\left(\He^2(\sigma^{N,\eps}_{i-1},\nu^{N,\eps}_i) + \W^2(\nu^{N,\eps}_i,\sigma^{N,\eps}_i)\right) \\
& = N\sum_{i=1}^N \W\He(\sigma^{N,\eps}_{i-1},\sigma^{N,\eps}_i) \leq N\sum_{i=1}^N \int_{\pc} \sfH_{\W\He}\de\aalpha^{N,\eps}_i,
\end{split}
\end{equation}
where $\sfH_{\W\He}$ is as in \eqref{eq:HWHe}. If we introduce for ease of notation the functions $\sfH',\sfH'' : \pc \to [0,+\infty)$ defined as
\[
\sfH'([x_0,r_0],[x_1,r_1]) \coloneqq (r_0-r_1)^2, \qquad  \sfH''([x_0,r_0],[x_1,r_1]) \coloneqq  
\left\{
\begin{array}{ll}
r_1^2 \sfd^2(x_0,x_1) & \textrm{if } r_0>0 \\
0 & \textrm{if } r_0=0
\end{array},
\right.
\]
so that $\sfH_{\W\He} = \sfH' + \sfH''$, then we can rewrite the last term above as
\begin{equation}\label{eq:three-terms}
\begin{split}
N\sum_{i=1}^N \int_{\pc} \sfH_{\W\He}\de\aalpha^{N,\eps}_i & = \int N\sum_{i=1}^N \sfH' \left ( \f{y} \left (\frac{i-1}{N} \right ), \f{y} \left (\frac{i}{N} \right ) \right ) \de\ppi_\eps(\f{y}) \\
& \quad + \int_{A_\eps} N\sum_{i=1}^N \sfH'' \left ( \f{y} \left (\frac{i-1}{N} \right ), \f{y} \left (\frac{i}{N} \right ) \right ) \de\ppi(\f{y}) \\
& \quad + \int_{A^c} N\sum_{i=1}^N \sfH'' \left ( \f{y} \left (\frac{i-1}{N} \right ), \f{y} \left (\frac{i}{N} \right ) \right ) \de\ppi(\f{y}) \\
& \quad + \int N\sum_{i=1}^N \sfH'' \left ( \f{y} \left (\frac{i-1}{N} \right ), \f{y} \left (\frac{i}{N} \right ) \right ) \de T_\sharp(\ppi \mres (A \setminus A_\eps))(\f{y}).
\end{split}
\end{equation}
Let us now discuss separately the four integrals, that for sake of brevity we shall denote $I_1^N$, $I_2^N$, $I_3^N$, and $I_4^N$ respectively. As regards the first one, by \eqref{eq:AC-cone} first and Jensen's inequality then we obtain
\[
\begin{split}
I_1^N & = \int N\sum_{i=1}^N \left|r_\f{y}\Big( \frac{i}{N} \Big) - r_\f{y}\Big(\frac{i-1}{N} \Big) \right|^2 \de\ppi_\eps(\f{y}) \leq \int N\sum_{i=1}^N \left|\int_{\frac{i-1}{N}}^{\frac{i}{N}}|r_\f{y}'(t)|\de t \right|^2 \de\ppi_\eps(\f{y}) \\
& \leq \int \sum_{i=1}^N \int_{\frac{i-1}{N}}^{\frac{i}{N}}|r_\f{y}'(t)|^2\de t \de\ppi_\eps(\f{y}) = \int \int_0^1 |r_\f{y}'(t)|^2\de t\de\ppi_\eps(\f{y})
\end{split}
\]
and observe that the right-hand side does not depend on $N$ any longer, so that
\[
\limsup_{N \to +\infty} I_1^N \leq \int \int_0^1 |r_\f{y}'(t)|^2\de t\de\ppi_\eps(\f{y}).
\]
As for $I_2^N$, note that for any geodesic $\f{y} \in A$, and a fortiori in $A_\eps$, both $\f{y}_0, \f{y}_1 \neq \f{o}$, so that by Remark \ref{rem:geocone} it holds $\min_{t \in [0,1]} r_\f{y}(t) > 0$. This fact, together with \eqref{eq:AC-cone} and Jensen's inequality again, implies that
\[
\begin{split}
I_2^N & = \int_{A_\eps} N\sum_{i=1}^N \left|r_\f{y}\Big(\frac{i}{N}\Big)\right|^2 \sfd^2\left(x_\f{y}\Big( \frac{i}{N} \Big),x_\f{y}\Big( \frac{i-1}{N} \Big)\right)\de\pi(\f{y}) \\
& \leq \int_{A_\eps} N\sum_{i=1}^N \left|r_\f{y}\Big(\frac{i}{N}\Big)\right|^2 \left|\int_{\frac{i-1}{N}}^{\frac{i}{N}}|x_\f{y}'(t)|_\sfd \de t\right|^2 \d\ppi(\f{y}) \\
& \leq \int_{A_\eps} \sum_{i=1}^N \left|r_\f{y}\Big(\frac{i}{N}\Big)\right|^2 \int_{\frac{i-1}{N}}^{\frac{i}{N}}|x_\f{y}'(t)|^2_\sfd \de t \d\ppi(\f{y}) \\
& = \int_{A_\eps}\int_0^1 \left| r_\f{y}\left(\frac{\lceil tN \rceil}{N} \right) \right|^2 |x_\f{y}'(t)|_\sfd^2 \de t\de\ppi(\f{y}).
\end{split}
\]
Dividing and multiplying the integrand function by $|r_\f{y}(t)|^2$ (which is possible since, as already pointed out, for any geodesic $\f{y} \in A_\eps$ we have $\min_{t \in [0,1]} r_\f{y}(t) > 0$), we can exploit the definition of $A_\eps$, because $|\frac{\lceil tN \rceil}{N} - \frac1N| < \frac1N < \eps$ for every $t\in [0,1]$, and this yields
\[
\left| r_\f{y}\left(\frac{\lceil tN \rceil}{N} \right) \right|^2 |x_\f{y}'(t)|_\sfd^2 \leq 4 \chi_{A_\eps} |r_\f{y}(t)|^2 |x_\f{y}'(t)|_\sfd^2, \qquad \forall t \in [0,1],\,\forall \f{y} \in A_\eps.
\]
Since $\chi_{A_\eps} |r_\f{y}(t)|^2 |x_\f{y}'(t)|_\sfd^2 \leq |r_\f{y}(t)|^2 |x_\f{y}'(t)|_\sfd^2 \in L^1(\ppi \otimes \mathcal{L}^1 \mres [0,1])$ as a consequence of \eqref{eq:lifting_eq}, by Lebesgue's dominated convergence theorem we conclude that
\[
\limsup_{N \to +\infty} I_2^N \leq \int_{A_\eps} \int_0^1 |r_\f{y}(t)|^2 |x_\f{y}'(t)|^2_\sfd \de\ppi(\f{y}).
\]
For the third integral in \eqref{eq:three-terms}, by Theorem \ref{thm:toolbox}(d), if $\f{y} \in A^c$, then only two situations may occur:
\begin{itemize}
\item either $r_\f{y}(0) r_\f{y}(1) = 0$, so that $x_{\f{y}}(t)$ is constant in time (since $\f{y}$ falls within case \eqref{item:geod_trivial} or case \eqref{item:geod_cusp} discussed in Section \ref{sec:curves_cone});
\item or $r_\f{y}(0) r_\f{y}(1) >0$ and $\sfd(x_\f{y}(0),x_\f{y}(1))=0$, so that $x_{\f{y}}(t)$ is constant in time.
\end{itemize}
In both cases 
\[ r^2_{\f{y}}(t) |x_{\f{y}}'(t)|^2_{\sfd}= \sfH'' \left ( \f{y} \left (\frac{i-1}{N} \right ), \f{y} \left (\frac{i}{N} \right ) \right ) =0 \qquad \text{for every } i=1, \dots, N, \text{ and a.e.~} t \in (0,1), \]
so that
\[ 
\limsup_{N \to +\infty} I_3^N = 0 = \int_{A^c} \int_0^1 |r_\f{y}(t)|^2 |x_\f{y}'(t)|^2_\sfd \de\ppi(\f{y}).
\]
Finally, the fourth integral in \eqref{eq:three-terms} is given by
\[
I_4^N = \int_{A \setminus A_\eps} N\sum_{i=1}^N \sfH'' \left ( T(\f{y}) \left (\frac{i-1}{N} \right ), T(\f{y}) \left (\frac{i}{N} \right ) \right )\de\ppi(\f{y})
\]
and we will now discuss separately the case when $N$ is even and when $N$ is odd. In the latter, we first remark that for all $\f{y} \in A$ (which ensures $r_{T(\f{y})}(0) = r_\f{y}(0) > 0$ and $r_{T(\f{y})}(1) = r_\f{y}(1) > 0$) it holds $r_{T(\f{y})}(i/N) > 0$ for all $i=0,\dots,N$ by definition of $T$. Thus, by definition of $\sfH''$
\[
\begin{split}
I_4^N & = \int_{A \setminus A_\eps} N\sum_{i=1}^N \left|r_{T(\f{y})}\Big(\frac{i}{N}\Big)\right|^2 \sfd^2\left(x_{T(\f{y})}\Big( \frac{i}{N} \Big),x_{T(\f{y})} \Big( \frac{i-1}{N} \Big)\right) \de\ppi(\f{y}) \\
\end{split}
\]
and we observe that all but one terms $\sfd(x_{T(\f{y})}(i/N),x_{T(\f{y})}((i-1)/N))$ equal 0; more precisely, since $N$ is odd we have exactly one non-zero summand which corresponds to the unique index $i^*$ such that $(i^*-1)/N < 1/2 < i^*/N$, namely the unique index $i^*$ such that $x_{T(\f{y})}$ is discontinuous in $[\frac{i^*-1}{N},\frac{i^*}{N})$. Therefore
\begin{equation}\label{eq:bounding_r2d2}
\begin{split}
N\sum_{i=1}^N \left|r_{T(\f{y})}\Big(\frac{i}{N}\Big)\right|^2 & \sfd^2\left(x_{T(\f{y})}\Big( \frac{i}{N} \Big),x_{T(\f{y})} \Big( \frac{i-1}{N} \Big)\right) \\ 
& = N\left| r_{T(\f{y})} \Big(\frac{i^*}{N}\Big) \right|^2 \sfd^2\left(x_{T(\f{y})} \Big( \frac{i^*}{N} \Big),x_{T(\f{y})} \Big( \frac{i^*-1}{N} \Big)\right) \\
& \leq N \frac{4R^2}{N^2} \frac{\pi^2}{4} = \frac{R^2\pi^2}{N},
\end{split}
\end{equation}
where the inequality is due to Theorem \ref{thm:toolbox}(d) and to the fact that, by definition of $T$ and Theorem \ref{thm:toolbox}(c),
\[
\left|r_{T(\f{y})}\Big(\frac{i^*}{N}\Big)\right| \leq \left| \frac{2i^*}{N}-1 \right| \max\{r_\f{y}(0),r_\f{y}(1)\} \leq \frac{2R}{N}.
\]
If we now plug \eqref{eq:bounding_r2d2} into the expression of $I_4^N$, we get the estimate $I_4^N \leq \frac1N R^2\pi^2$ for $N$ odd. 
\\If instead $N$ is even, there exists exactly one index $i^*$ such that $(i^*-1)/N = 1/2$, which implies $r_{T(\f{y})}((i^*-1)/N) = 0$; for all the other indexes $i=0,\dots,N$, $i\neq i^*$, we have $\sfd(x_{T(\f{y})}(i/N),x_{T(\f{y})}((i-1)/N))=0$ and thus $I_4^N=0$ for $N$ even. 
\\Overall,
\[
\limsup_{N \to +\infty} I_4^N \leq 0 = \int_{A \setminus A_\eps}\int_0^1 |r_{T(\f{y})}(t)|^2 |x_{T(\f{y})}'(t)|^2_\sfd \de t \de\ppi(\f{y}),
\]
since $|x_{T(\f{y})}'(t)|_\sfd = 0$ a.e.\ in $[0,1]$, for any $\f{y} \in A \setminus A_\eps$. 
\\In conclusion, combining the above estimates on $I_1^N,I_2^N,I_3^N,I_4^N$ with \eqref{eq:three-terms} and \eqref{eq:bounding_EN} we deduce that
\[
\begin{split}
(\He \nabla \W)^2(\mu_0,\mu_1) & = \liminf_{N \to +\infty} \inf \big\{\mathcal{E}_N(P) \,:\, P \in \mathscr{P}(\mu_0,\mu_1;N)\big\} \leq \liminf_{N \to +\infty} \mathcal{E}_N(P_{N,\eps}) \\
& \leq \liminf_{N \to +\infty}\big\{I_1^N+I_2^N+I_3^N+I_4^N\big\} \leq \limsup_{N \to +\infty}\big\{I_1^N+I_2^N+I_3^N+I_4^N\big\} \\
& \leq \limsup_{N \to +\infty} I_1^N + \limsup_{N \to +\infty} I_2^N + \limsup_{N \to +\infty} I_3^N + \limsup_{N \to +\infty} I_4^N \\
& \leq \int \int_0^1 |r_\f{y}'(t)|^2\de t\de\ppi_\eps(\f{y}) + \int_{A_\eps} \int_0^1 |r_\f{y}(t)|^2 |x_\f{y}'(t)|^2_\sfd \de\ppi(\f{y}) \\
& \quad + \int_{A^c} \int_0^1 |r_\f{y}(t)|^2 |x_\f{y}'(t)|^2_\sfd \de\ppi(\f{y}) + \int\int_0^1 |r_\f{y}(t)|^2 |x_\f{y}'(t)|^2_\sfd \de T_\sharp(\ppi \mres(A \setminus A_\eps)) \\
& = \int_{A_\eps} \mathcal{A}_2 \de\ppi + \int_{A^c} \mathcal{A}_2 \de\ppi + \int \mathcal{A}_2 \de T_\sharp(\ppi \mres (A \setminus A_\eps)) \\
& = \int \mathcal{A}_2 \de\ppi_\eps.
\end{split}
\]
As this inequality holds true for any fixed $\eps>0$, passing to the limit as $\eps \downarrow 0$ and applying Lemma \ref{lem:convergence-2action-pi} yield
\[
\begin{split}
(\He \nabla \W)^2(\mu_0,\mu_1) & \leq \lim_{\eps \downarrow 0}\int \mathcal{A}_2 \de\ppi_\eps = \int \mathcal{A}_2 \de\ppi = \int\int_0^1 |\f{y}'(t)|_{\sfd_{\f{C}}}^2 \de t\de \ppi(\f{y}) \\
& = \int_0^1 |\mu_t'|_\HK^2 \de t = \HK^2(\mu_0,\mu_1),
\end{split}
\]
thus concluding the proof.
\end{proof}

\section{\texorpdfstring{$\HK \le \He \nabla \W$}{HK < He d W2}}\label{sec:ineq2}

Aim of this section is to prove the converse inequality to \eqref{eq:aimsec6}, namely the following

\begin{theorem}\label{thm:main1}
Let $(X,\sfd)$ be a complete and separable metric space and let $\mu_0, \mu_1 \in \meas_+(X)$. Then
\[ 
\HK_{\sfd}(\mu_0, \mu_1) \leq (\He_{2} \nabla \W_{2, \sfd})(\mu_0, \mu_1)\,.
\]
\end{theorem}
Unlike in Theorem \ref{thm:firstineq}, notice that for this inequality we do not even require the space to be geodesic.

For ease of notation, in what follows we drop the dependence on the exponent $p=2$ and on the distance $\sfd$ in $\He, \HK, \W$. Moreover, as the proof of Theorem \ref{thm:main1} is rather involved, we present here the strategy. Again, we denote by $\mathcal{E}_N$ the energy as in \eqref{eq:energy} where $\sfc_1= \He$ and $\sfc_2=\W$, unless otherwise stated.
In comparison with the previous section, to prove it we need to adopt a rather opposite approach. Indeed, in Section \ref{sec:ineq1} we started with a $\HK$-geodesic between $\mu_0$ and $\mu_1$ and by time discretization we were naturally led to an optimal family of $N$-paths between the same measures. Now we start with an optimal family of $N$-paths and we need to find a way to interpolate between them: this will first require a gluing procedure and a subsequent lifting to $\prob(\rmC([0,1];(\f{C}[X],\sfd_{\f{C}})))$. 

More explicitly, assuming without loss of generality that $(\He \nabla \W)(\mu_0, \mu_1) < + \infty$, by definition of $\He \nabla \W$ there exist a sequence $(N_k)_{k \in \N}$ and $N_k$-paths $(P'_k)_{k \in \N}$ such that 
\[ 
P'_k \in \mathscr{P}(\mu_0, \mu_1; N_k), \quad N_k \uparrow + \infty, \quad \mathcal{E}_{N_k}(P'_k) \to (\He \nabla \W)^2(\mu_0, \mu_1) \quad \text{ as } k \to + \infty\,.
\]
By Proposition \ref{prop:whe}(4) we can replace every $N_k$-path $P'_k = (\mu_0^k, \mu_1^k, \dots, \mu_{N_k}^k; \tilde{\nu}_1^k, \dots, \tilde{\nu}_{N_k}^k)$ with $P_k \coloneqq (\mu_0^k, \mu_1^k, \dots, \mu_{N_k}^k; \nu_1^k, \dots, \nu_{N_k}^k)$, where $\nu_i^k$ is such that
\[ 
\W\He(\mu_{i-1}^k, \mu_i^k) = \He^2(\mu_{i-1}^k, \nu_i^k) + \W^2(\nu_i^k, \mu_i^k) \quad \text{ for every } 1 \le i \le N_k, \, k \in \N\,.
\]
For the new paths it clearly still holds 
\begin{equation}\label{eq:energy_to_HeW}
\mathcal{E}_{N_k}(P_k) \to (\He \nabla \W)^2(\mu_0, \mu_1) \quad \text{ as } k \to + \infty\,, 
\end{equation}
so that the inequality $\HK(\mu_0,\mu_1) \le (\He \nabla \W)(\mu_0,\mu_1)$ is proven if we show that $\HK^2(\mu_0,\mu_1) \le \lim_{k \to +\infty} \mathcal{E}_{N_k}(P_k)$. Moreover, by \eqref{eq:energy_to_HeW} there exists a constant $C \in (0, +\infty)$ such that 
\begin{equation}\label{eq:cbound}
\mathcal{E}_{N_k}(P_k) \le C  \quad \text{ for every } k \in \N\,.
\end{equation}
By Jensen's inequality this implies
\[
\left( \sum_{i=1}^{N_k} \sqrt{\W\He(\mu_{i-1}^k, \mu_i^k)}\right)^2 \le \mathcal{E}_{N_k}(P_k) \le C \qquad \text{ for every } k \in \N\,,
\]
whence
\begin{equation}\label{eq:theta}
\sqrt{\mu_0(X)} +\sum_{i=1}^{N_k} \sqrt{\W\He(\mu_{i-1}^k, \mu_i^k)} \le \sqrt{\mu_0(X)} + \sqrt{C} =: \Theta \,,
\end{equation}
so that the left-hand side is bounded by a finite positive constant independent of $k$. By (a simple adaptation of) \cite[Lemma 7.11]{LMS18} we can thus find probability measures $\aalpha^k \in \prob(\f{C}_\Theta[X]^{N_k+1})$, see \eqref{eq:cone_r}, such that
\begin{equation}\label{eq:aalpha_k}
\left\{
\begin{array}{l}
\aalpha_i^k\coloneqq \pi^{i-1,i}_\sharp \aalpha^k \in \f{H}^2(\mu_{i-1}^k, \mu_i^k) \\[10pt]
\displaystyle{\int_{\pc} \sfH_{\W\He} \de \aalpha_i^k = \W\He(\mu_{i-1}^k, \mu_i^k)}
\end{array}
\right. \quad \textrm{ for every } i=1, \dots, N_k,\, k \in \N\,,
\end{equation}
where $\sfH_{\W\He}$ is as in \eqref{eq:HWHe}. The energy $\mathcal{E}_{N_k}$ then rewrites as
\begin{equation}\label{eq:pinguino}
\mathcal{E}_{N_k}(P_k) = N_k\sum_{i=1}^{N_k} \int_{\pc} \sfH_{\W\He} \de \aalpha_i^k \,,
\end{equation}
which resembles a discretization of the action $\mathcal{A}_2$, see \eqref{eq:a2def}. We are thus naturally led to lift $\aalpha^k$ to $\ppi^k \in \prob(\AC^2([0,1];(\f{C}[X],\sfd_{\f{C}})))$, because if we prove that
\begin{equation}\label{eq:desiderata}
\int \mathcal{A}_2 \de\ppi \leq \liminf_{k \to +\infty}\int  N_k\sum_{i=1}^{N_k} \sfH_{\W\He} \circ (\sfe_{\frac{i-1}{N_k}},\sfe_{\frac{i}{N_k}})\de \ppi^k = \liminf_{k \to + \infty} \mathcal{E}_{N_k}(P_k)\,,
\end{equation}
for some $\ppi \in \prob(\AC^2([0,1];(\f{C}[X],\sfd_{\f{C}})))$ with $(\sfe_0, \sfe_1)_\sharp\ppi \in \f{H}^2(\mu_0, \mu_1)$, then the desired inequality $ \HK(\mu_0,\mu_1)\le (\He \nabla \W)(\mu_0,\mu_1)$ follows from the dynamical representation \eqref{eq:dynamic_HK} of $\HK$. 

However, for the above inequality to hold we need tightness of the measures $\ppi^k$, so that $\ppi^k \rightharpoonup \ppi$ for some $\ppi$, and some sort of uniform convergence of the integrands towards $\mathcal{A}_2$. To overcome the first obstacle, we assume $X$ to be compact: this is not restrictive by Proposition \ref{prop:compred}. The second obstacle is instead more involved. If the supports of the measures $\ppi^k$ and $\ppi$ only contain curves with radii uniformly bounded away from 0, then \eqref{eq:desiderata} holds true: the idea is to bound, as in \eqref{eq:caimano},
\[
\sfH_{\W\He} \circ (\sfe_{\frac{i-1}{N_k}},\sfe_{\frac{i}{N_k}})(\f{y}) \geq \left(1 \wedge \frac{r_{\f{y}}(i/N_k)}{r_{\f{y}}((i-1)/N_k)}\right) \sfd_{\f{C}}^2(x_{\f{y}}((i-1)/N_k),x_{\f{y}}(i/N_k))
\]
and, if $\ppi^k$ is concentrated between the time steps $\frac{i-1}{N_k}$ and $\frac{i}{N_k}$ on constant-speed geodesics, then $\sfd_{\f{C}}(x_{\f{y}}((i-1)/N_k),x_{\f{y}}(i/N_k))$ coincides with the rescaled metric speed $N_k^{-1} |\f{y}'(t)|_{\sfd_{\f{C}}}$. Integrating in time the right-hand side above thus yields (up to a multiplicative factor tending to 1 as $k \to +\infty$) the action $\mathcal{A}_2$. Handling the liminf as $k \to +\infty$ when the above inequality is integrated w.r.t.\ $\ppi^k$ is discussed in Proposition \ref{prop:conva2}.

Incidentally, the above discussion suggests how to actually lift $\aalpha^k$ to $\prob(\AC^2([0,1];(\f{C}[X],\sfd_{\f{C}})))$. We define a geodesic interpolation map $\Gamma^{N_k} : \f{C}[X]^{N_k+1} \to \AC^2([0,1];(\f{C}[X],\sfd_{\f{C}}))$ that sends the $(N_k+1)$-tuple $(\f{y}_0,\dots,\f{y}_{N_k})$ into an absolutely continuous curve that, on every interval $[(i-1)/N_k,i/N_k]$, $i=1,\dots,N_k$, is a reparametrized constant-speed geodesic connecting $\f{y}_{i-1}$ and $\f{y}_i$ with metric speed equal to $N_k \sfd_{\f{C}}(\f{y}_{i-1},\f{y}_i)$. Then, $\ppi^k \coloneqq \Gamma^{N_k}_\sharp \aalpha^k$.

\medskip

However, we cannot exclude the limit measure $\ppi$ to concentrate on curves that either start or end in $\f{o}$. This motivates a first splitting of $\ppi^k$ and $\ppi$. The idea, although slightly different from the one adopted in the proof of Theorem \ref{thm:main1} (we will perform the splitting on $\ssigma^k \coloneqq (\sfe_0,\sfe_1)_\sharp \ppi^k$ and $\ssigma \coloneqq (\sfe_0,\sfe_1)_\sharp \ppi$), is the following: we single out the curves that either start or end in $\f{o}$, namely $\f{D}_{\f{o}} \coloneqq \{\f{y} \in \AC^2([0,1];(\f{C}[X],\sfd_{\f{C}})) \,:\, \f{y}(0) = \f{o} \textrm{ or } \f{y}(1) = \f{o}\}$ and accordingly split $\ppi$ into
\[
\ppi_{\f{o}} \coloneqq \ppi \mres \f{D}_{\f{o}} \qquad \textrm{and} \qquad \ppi_> \coloneqq \ppi \mres \f{D}_{\f{o}}^c \,.
\]
Then, by Lemma \ref{lem:decomposition}, we build measures $\ppi_>^k,\ppi_{\f{o}}^k \in \meas_+(\AC^2([0,1];(\f{C}[X],\sfd_{\f{C}})))$ such that $\ppi^k = \ppi_{\f{o}}^k + \ppi_>^k$ and $\ppi_{\f{o}}^k \rightharpoonup \ppi_{\f{o}}$, $\ppi_>^k \rightharpoonup \ppi_>$ as $k \to +\infty$. It is not difficult to check that 
\begin{equation}\label{eq:liminf_pok}
\begin{split}
\liminf_{k \to +\infty} \int N_k\sum_{i=1}^{N_k} \sfH_{\W\He} \circ (\sfe_{\frac{i-1}{N_k}},\sfe_{\frac{i}{N_k}})\de \ppi_{\f{o}}^k & \geq \int_{\pc} \sfd_{ \pi/2,\f{C}}^2 \,\de (\sfe_0,\sfe_1)_\sharp\ppi_{\f{o}} \\
& \ge \HK^2(\f{h}((\sfe_0)_\sharp \ppi_{\f{o}}), \f{h}((\sfe_1)_\sharp \ppi_{\f{o}}))
\end{split}
\end{equation}
as a consequence of the definition of the cone distance, see the first step in the proof of Theorem \ref{thm:main1}.
A further splitting concerns the measures $\ppi^k_>$: we single out a `good' set of curves $\f{D}_{N_k} \subset \f{D}_{\f{o}}^c$ where it holds
\[ \f{y} \in \f{D}_{N_k} \quad \Longrightarrow \quad N_k \sum_{i=1}^{N_k}\sfH_{\W\He} \left ( \f{y} \left ( \frac{i-1}{N_k}\right ), \f{y}\left (\frac{i}{N_k}\right) \right )\ge \sfd_{ \pi/2,\f{C}}^2 (\f{y}(0), \f{y}(1)) \,,
\]
so that, setting $\ppi^k_g \coloneqq \ppi^k_> \mres \f{D}_{N_k}$, we also get
\begin{equation}\label{eq:liminf_pgk}
\begin{split}
\liminf_{k \to +\infty} \int N_k\sum_{i=1}^{N_k} \sfH_{\W\He} \circ (\sfe_{\frac{i-1}{N_k}},\sfe_{\frac{i}{N_k}})\de \ppi_g^k & \geq \int_{\pc} \sfd_{ \pi/2,\f{C}}^2 \,\de (\sfe_0,\sfe_1)_\sharp\ppi_g \\
&\ge \HK^2(\f{h}((\sfe_0)_\sharp \ppi_{g}), \f{h}((\sfe_1)_\sharp \ppi_{g})),
\end{split}
\end{equation}
where $\ppi_g$ is the weak limit of $\ppi_g^k$ as $k \to + \infty$. 

It remains to handle the \emph{energy of the plan} $\ppi^k_b$, that is the integral
\begin{equation}\label{eq:discrete_energy_bad}
\int N_k\sum_{i=1}^{N_k} \sfH_{\W\He} \circ (\sfe_{\frac{i-1}{N_k}},\sfe_{\frac{i}{N_k}})\de \ppi^k_b,    
\end{equation}
where $\ppi^k_b \coloneqq  \ppi^k_> \mres (\f{D}_{\f{o}}^c \setminus \f{D}_{N_k})$ with weak limit $\ppi_b$ as $k \to +\infty$. Recalling that if $\ppi^k_b$ and $\ppi_b$ are concentrated on curves with radii uniformly bounded away from zero then \eqref{eq:desiderata} (written for $\ppi^k_b$ and $\ppi_b$) holds true, by a cut-off argument we introduce a new family of measures $\ppi_b^{k,m}$ concentrated on curves $\f{y}$ with $1/m < r_{\f{y}}(0), r_{\f{y}}(1) \leq \Theta$ with a lower energy, and also satisfying
\begin{equation}\label{eq:samehk}
\lim_{m \to + \infty} \HK^2(\f{h}((\sfe_0)_\sharp \ppi_{b}^m),\f{h}((\sfe_1)_\sharp \ppi_{b}^m)) = \HK^2(\f{h}((\sfe_0)_\sharp \ppi_{b}),\f{h}((\sfe_1)_\sharp \ppi_{b})) \,.
\end{equation}
The cut-off makes the radii bounded away from zero at the initial and final time. We cannot avoid the radii to become arbitrarily small or vanish at intermediate times. However, we can show that if this happens, then for any $k \in \N$ we can find a better plan $\tilde{\ppi}^{k,m}_b$ (that is, with a lower energy) with the same initial and final 2-homogeneous marginals, and additionally concentrated on curves with radii uniformly bounded away from zero at all times: this is achieved by exploiting the specific form of the set $\f{D}_{N_k}$, the structure of geodesics on the cone, building a map $G_{N_k}$ lifting those curves whose radii are too close to $0$ (see Theorem \ref{thm:thmdinicolo} for the precise result), and setting $\tilde{\ppi}^{k,m}_b \coloneqq (G_{N_k})_\sharp \ppi_b^{k,m}$. The latter map is constructed studying the properties of minimizers of the function 
\[ 
N_k\sum_{i=1}^{N_k} \sfH_{\W\He}(\f{y}_{i-1}, \f{y}_i) \quad \text{on the set} \quad \f{C}[X]^{N_k+1}\,,
\]
and it is at the core of the arguments used in the proof of Theorem \ref{thm:main1}. We can finally apply Proposition \ref{prop:compred} and obtain \eqref{eq:desiderata} for the plan $\tilde{\ppi}^{k,m}_b$, that is 
\begin{align*}
\HK^2(\f{h}((\sfe_0)_\sharp \ppi_{b}^m),\f{h}((\sfe_1)_\sharp \ppi_{b}^m)) &= \HK^2(\f{h}((\sfe_0)_\sharp \tilde{\ppi}_{b}^m),\f{h}((\sfe_1)_\sharp \tilde{\ppi}_{b}^m)) \le \int \mathcal{A}_2 \de\tilde{\ppi}^m_b \\
&\leq \liminf_{k \to +\infty}\int  N_k\sum_{i=1}^{N_k} \sfH_{\W\He} \circ (\sfe_{\frac{i-1}{N_k}},\sfe_{\frac{i}{N_k}})\de \tilde{\ppi}^{k,m}_b \,,
\end{align*}
where $\tilde{\ppi}^m_b$ is the weak limit of $\tilde{\ppi}^{k,m}_b$ as $k \to + \infty$, and the first inequality follows from the dynamical formulation of $\HK$, see \eqref{eq:dynamic_HK}. Since, as discussed above, the energy of $\tilde{\ppi}^{k,m}_b$ lower bounds the one of $\ppi^{k,m}_b$ and, in turn, this is less than the one of $\ppi^k_b$ for every $m \in \N$, we conclude that 
\[ \HK^2(\f{h}((\sfe_0)_\sharp \ppi_{b}^m),\f{h}((\sfe_1)_\sharp \ppi_{b}^m)) \le \liminf_{k \to + \infty} N_k\sum_{i=1}^{N_k} \sfH_{\W\He} \circ (\sfe_{\frac{i-1}{N_k}},\sfe_{\frac{i}{N_k}})\de \ppi^{k}_b \quad \text{ for every } m \in \N.\]
Using \eqref{eq:samehk}, a final passage to the limit as $m \to + \infty$ gives
\begin{equation}\label{eq:pibm}
    \HK^2(\f{h}((\sfe_0)_\sharp \ppi_{b}),\f{h}((\sfe_1)_\sharp \ppi_{b})) \le \liminf_{k \to + \infty}\int N_k\sum_{i=1}^{N_k} \sfH_{\W\He} \circ (\sfe_{\frac{i-1}{N_k}},\sfe_{\frac{i}{N_k}})\de \ppi^{k}_b.
\end{equation}
Together with \eqref{eq:pinguino}, \eqref{eq:liminf_pok}, \eqref{eq:liminf_pgk}, and the subadditivity of $\HK^2$, this yields $\HK \leq \He \nabla \W$.

\subsection{A priori estimates for the minimization problem}\label{sec:apriori}

In this section we study, \emph{at a discrete level}, the minimization problem associated with the inf-convolution between the Hellinger and the Wasserstein distance. We obtain some a priori bounds on the minimizers that will be crucial for the convergence results for general measures contained in the next subsection.

Let $N\in \mathbb{N}_{>1}$, and $r_0,r_N>0$. Let $f_N(r_0, r_N; \cdot): \mathbb{R}^{2N-1} \to [0,+\infty)$ be the function
\begin{equation}\label{eq:deff_N}
  f_N(r_0,r_N;r_1,\dots,r_{N-1},d_1,\dots,d_N)\coloneqq N\sum_{i=1}^N\left(|r_i-r_{i-1}|^2+r_i^2d_i^2\right). 
\end{equation}
Notice that $f_N(r_0, r_N; \cdot)$ is clearly continuous and symmetric w.r.t.~the origin in the variables $d_i$. Given $d \ge 0$, we also introduce the sets $\Omega_{\ge d}^N, \Omega_d^N\subset \mathbb{R}^{2N-1}$ defined as
\begin{equation}\label{def:Omega_dN}
\Omega_{\ge d}^N\coloneqq \left \{ r_i >0, \, d_i \ge 0, \, \sum_{i=1}^N d_i \ge d \right \}, \qquad \Omega_d^N\coloneqq  \left \{ r_i >0, \, d_i \ge 0, \, \sum_{i=1}^N d_i = d \right \}.
\end{equation}

\begin{lemma}\label{lem: properties-f}
Let $N\in \mathbb{N}_{>1}$ and let us set $Y\coloneqq (0,+\infty)^2 \times [0,+\infty)$. The following properties hold:
\begin{enumerate}
\item For every $(r_0,r_N,d)\in Y$, the set
\begin{equation}\label{eq:defPsi}
\Psi(r_0,r_N,d)\coloneqq \argmin_{\Omega_{\ge d}^N} f_N(r_0,r_N;\cdot)
\end{equation}
is non-empty. 
\item For every $(r_0,r_N,d)\in Y$, if 
\[\min_{\Omega_{\ge d}^N} f_N(r_0, r_N; \cdot) \le |r_0-r_N|^2 + r_0r_N (d \wedge \pi/2)^2\]
then 
\begin{equation}\label{eq:radialandspatialbounds}
 \Psi(r_0,r_N,d)\subset    \left[\Big(1-\frac{\pi}{4}\Big)r_0\wedge r_N,r_0\vee r_N\right]^{N-1}\times \left[0,\frac{1}{\sqrt{N}}\frac{\pi}{2-\frac{\pi}{2}}\frac{\sqrt{r_0r_N}}{r_0\wedge r_N}\right]^N\,.
 \end{equation} 
\item There exists a Borel measurable map $U_N:Y\to \mathbb{R}^{2N-1}$ such that $U_N(r_0,r_N,d)\in \Psi(r_0,r_N,d)\cap \Omega_{d}^N$ for every $(r_0,r_N,d)\in Y.$
\end{enumerate}
\end{lemma}

\begin{proof}
For every $r_0, r_N>0$, the function $f_N(r_0, r_N; \cdot)$ is increasing in the $d_i$-components in $[0,+\infty)$, hence it is immediate to see that the minimization problems of $f_N(r_0, r_N; \cdot)$ on $\Omega_{\ge d}^N$ or on $\Omega_d^N$ are equivalent (in the sense that they have the same minimizers and the same minimal value); we thus restrict ourselves to study the problem in the latter set.
For every $(r_0,r_N,0)$ in the Borel set $Y_0\coloneqq (0,+\infty)^2\times \{0\}$, setting $r_i\coloneqq r_0+i/N(r_N-r_0)$ for $i=1,\dots,N-1$, it is also immediate to see that the point $p=(r_1,\dots,r_{N-1},0,\dots,0)\in \Omega_0^N$ is the only minimizer of $f_N(r_0, r_N; \cdot)$ in $\Omega_0^N$ and all the statements trivially hold with $Y$ replaced by $Y_0$. It is thus not restrictive to assume $d>0$ in the following, and construct a Borel measurable map $U_N:Y\setminus Y_0\to \R^{2N-1}$.

(1) We fix $(r_0, r_N, d) \in Y \setminus Y_0$, we let 
\[\overline{\Omega}_d^N=\left \{ r_i \ge0, \, d_i \ge 0, \, \sum_{i=1}^N d_i = d \right \}
\]
be the closure of $\Omega_d^N$ in $\R^{2N-1}$, and we set
\begin{equation}\label{eq:defsetCandK}
C\coloneqq \left[0, \frac{N(r_0^2+r_N^2)^{1/2}}{d}\right]^{N-1}\times \mathbb{R}^{N}\subset \R^{2N-1}, \qquad K\coloneqq \overline{\Omega}_d^N\cap C\,.  
\end{equation}
Since $f_N(r_0,r_N;\cdot)$ is continuous and $K$ is compact, $f_N(r_0,r_N;\cdot)$ has at least a minimum point in $K$. For every $z=(r_1,\dots,r_{N-1,}d_1,\dots,d_N)\in \overline{\Omega}_d^N\cap C^c$ let $j\in \{1,\dots,N\}$ be an index such that $d_j\ge d/N$ (notice that such an index $j$ must exist); we have 
\[f(z)\ge Nr_j^2d_j^2> \frac{N^3}{d^2}(r_0^2+r_N^2)\frac{d^2}{N^2}=N(r_0^2+r_N^2)=f(c)\,,\]
where $c=(0,\dots,0,d,0,\dots,0)\in K$. Thus 
\[
\argmin_{K}f_N(r_0,r_N;\cdot)=\argmin_{\overline{\Omega}_d^N}f_N(r_0,r_N;\cdot).
\] 
To conclude, we show that any minimizer $z^{*}\coloneqq (r_1^*,\dots,r^*_{N-1},d^*_1,\dots,d^*_N)$ of $f_N(r_0,r_N;\cdot)$ in $\overline{\Omega}_d^N$ is such that $r^*_i>0$ for every $i=1,\dots,N-1$, so that
\begin{equation}\label{eq:argminonKandOmega}
\argmin_{K}f_N(r_0,r_N;\cdot)=\argmin_{\Omega_d^N}f_N(r_0,r_N;\cdot).
\end{equation}
Indeed, for any $z\coloneqq (r_1,\dots,r_{N-1},d_1,\dots,d_N)\in \overline{\Omega}_d^N$ the map
$$t\mapsto f_N(r_0,r_N;t,r_2,\dots,r_{N-1},d_1,\dots,d_N)$$
is strictly decreasing in $[0,r_0/(2+d_1^2))$ as an immediate consequence of its expression, which implies $r_1^*>0$. Moreover, if $z\coloneqq (r_1,\dots,r_{N-1},d_1,\dots,d_N)\in \overline{\Omega}_d^N$ is such that $r_j>0$ for every $1\le j\le i-1$, with $1<i<N$, then the map
$$t\mapsto f_N(r_0,r_N;r_1,\dots,r_{i-1},t,r_{i+1},\dots,r_{N-1},d_1,\dots,d_N)$$
is strictly decreasing in $[0,r_{i-1}/(2+d_i^2))$ which implies $r^*_i>0$ and concludes the proof of the first point of the statement. 
\smallskip

$(2)$ We fix $(r_0, r_N, d) \in Y \setminus Y_0$; notice that $f_N(r_0,r_N;\cdot)$ is a smooth function and the constraint $\sum_{i=1}^Nd_i=d$ is admissible, thus every $z=(r_1,\dots,r_{N-1},d_1,\dots,d_N)\in \Psi(r_0,r_N,d)$ satisfies the following system of equations given by the Lagrange multipliers theorem (here we also used the symmetry in the variables $d_i$ in order to avoid an additional constraint for $d_i \ge 0$):
\[
\begin{cases}
    r_i\left(2+d_i^2\right)=r_{i-1}+r_{i+1} \qquad &\forall i=1,\dots,N-1,\\
    2r_i^2d_i=\lambda_N \qquad &\forall i=1,\dots,N,\\
    \sum_{i=1}^N d_i=d\,,
\end{cases}
\]
where $\lambda_N$ is the Lagrange multiplier. By the equations in the first line of the system we infer $2r_i\le r_{i-1}+r_{i+1}$, which implies that one (and only one) of the two following situations occurs:
\begin{itemize}
    \item the map $i:\{0,\dots,N\}\to \R$, $i\mapsto r_i$ is monotone (increasing or decreasing);
    \item there exists an index $j\in \{1,\dots,N-1\}$ such that the map $i\mapsto r_i$ is monotone decreasing in $\{0,\ldots,j\}$ and monotone increasing in $\{j,\dots,N\}$. In particular, $r_j=\min\{r_i : i=1,\ldots,N-1\}$.
\end{itemize}
In the first case it is obvious that $r_i\in [(1-\pi/4)r_0\wedge r_N,r_0\vee r_N]$ for every $i=1,\dots,N-1$. 
\\In the second case we reason as follows:  it is clear that $r_i\le r_0\vee r_N$ for every $i=1,\dots,N-1$; moreover, thanks to the upper bound in the assumption and Jensen's inequality we have
\begin{equation}\label{eq:prooflowerboundradialpart}
\begin{aligned}
|r_0-r_N|^2 + r_0r_N (d \wedge \pi/2)^2 & \ge \min_{\Omega_d^N} f_N(r_0, r_N; z) \ge N\sum_{i=1}^N|r_i-r_{i-1}|^2\\
&\ge \left(\sum_{i=1}^N|r_i-r_{i-1}|\right)^2=(r_0+r_N-2r_j)^2,
\end{aligned}
\end{equation}
whence
\begin{equation}\label{eq:proofboundminradius}
    r_j\ge \frac{r_0+r_N-\sqrt{r_0^2+r_N^2+r_0r_N(\frac{\pi^2}{4}-2)}}{2}.
\end{equation}
Introducing the function $g:[1,\infty)\to \R$, $g(t)=1+t-\sqrt{1+t^2+t(\pi^2/4-2)}$, by elementary considerations one can prove that $g$ is minimized at $t=1$ with value $g(1)=2-\pi/2$. This fact, together with the bound \eqref{eq:proofboundminradius}, gives that 
\begin{equation}\label{eq:prooffinalboundradius}
 r_j\ge \Big(1-\frac{\pi}{4}\Big)r_0\wedge r_N\,.
\end{equation}
It remains to prove that 
\[
d_k\coloneqq \max\{d_i : i=1,\ldots,N\}\le \frac{1}{\sqrt{N}}\frac{\pi}{2-\frac{\pi}{2}}\frac{\sqrt{r_0r_N}}{r_0\wedge r_N}\,.
\]
To prove this bound we can reason as in the proof of \eqref{eq:prooflowerboundradialpart} to get
\[
|r_0-r_N|^2+r_0r_N(d\wedge \pi/2)^2\ge (r_0-r_N)^2+N\sum_{i=1}^Nr_i^2d_i^2
\]
which implies, using \eqref{eq:prooffinalboundradius},
\[r_0r_N\frac{\pi^2}{4}\ge N\left(1-\frac{\pi}{4}\right)^2(r_0\wedge r_N)^2\sum_{i=1}^Nd_i^2 \ge N\left(1-\frac{\pi}{4}\right)^2(r_0\wedge r_N)^2d_k^2\,,\]
which gives the desired bound on $d_k$.

\smallskip
(3) The Borel map $U_N: Y \setminus Y_0 \to \R^{2N-1}$ is constructed as a selection of the multi-valued map $\Psi:(0,+\infty)^3\to 2^{\R^{2N-1}}$ defined in \eqref{eq:defPsi}. In order to ensure the Borel measurability, we appeal to \cite[6.9.3]{Bogachev07} and we have to show that for every $U\subset \R^{2N-1}$ open the set $\hat{\Psi}(U)$ is Borel, where we have introduced the map 
\[\hat{\Psi}:2^{\R^{2N-1}}\to 2^{(0,+\infty)^3}, \qquad\hat{\Psi}(U)\coloneqq \{(r_0,r_N,d)\in (0,+\infty)^3\ :\ \Psi(r_0,r_N,d)\cap U\neq \emptyset\}\,.\]
It sufficient to show that $\hat{\Psi}(C)$ is closed for every $C\subset \R^{2N-1}$ closed, since every open set $U\subset \R^{2N-1}$ is a countable union of closed sets $\cup_n C_n=U$ and $\hat{\Psi}(\cup_n C_n)=\cup_n \hat{\Psi}(C_n)$. So let $y^j\coloneqq (r^j_0,r^j_N,d^j)$ be a sequence of points in $\hat{\Psi}(C)$ such that $y^j\to y\coloneqq (r_0,r_N,d)\in(0,+\infty)^3$; we aim to show that $y\in \hat{\Psi}(C)$. Since $(y^j)_j$ is converging, we can assume it is contained in a compact set $K'=[a,b]^3\subset (0,+\infty)^3$ and from the proof of point $(1)$ we know that every sequence $(x^j)_j$ such that $x^j\in \Psi(y^j)\cap C$ for every $j$ is contained in a compact set $K\subset \mathbb{R}^{2N-1}$ of the form (recalling \eqref{eq:defsetCandK} and \eqref{eq:argminonKandOmega})
\[
K = \left\{(r_1,\dots,r_{N-1},d_1,\dots,d_N) \,:\, r_i\in\left[0,\frac{\sqrt{2}Nb}{a}\right],\, d_i\ge 0,\, \sum_{i=1}^Nd_i\in[a,b]\right\}\,.
\]
Up to a non-relabeled subsequence, we can thus assume that $x^j$ is converging to some $x\in C$; let $p=(s_1,\dots,s_{N-1},c_1,\dots,c_N)\in \Omega_d^N$. We can consider the sequence $p^j\coloneqq (s_1,\dots,s_{N-1},c_1d^j/d,\dots,c_Nd^j/d) \in \Omega_{d^j}^N$ so that $p^j\to p$ and 
\[f_N(r_0^j,r_N^j;x^j)\le f_N(r_0^j,r_N^j;p^j)\]
since $x^j\in \Psi(y^j)$. Using the continuity of $f_N(\cdot,\cdot\,;\cdot):\R^{2N+1}\to [0,+\infty)$, passing to the limit in the last inequality we obtain
\[f_N(r_0,r_N;x)\le f_N(r_0,r_N;p)\]
which shows that $x\in \Psi(y)$ due to the arbitrariness of $p$, so that $y\in \hat{\Psi}(C)$.
\end{proof}

\vspace{.5cm}

Let $N\in \mathbb{N}_{\ge 1}$. We introduce the function $\sfH_N: \f{C}[X]^{N+1} \to [0,+\infty)$ defined as
\begin{equation}\label{eq:NsumH}
    \sfH_N(\f{y}_0, \f{y}_1, \dots, \f{y}_{N})\coloneqq N \sum_{i=1}^N \sfH_{\W\He}(\f{y}_{i-1}, \f{y}_i)\,.
\end{equation}
We also introduce
\begin{equation}\label{eq: fakeconedistance}
\tilde{\sfd}_{\f{C}}: \f{C}[X]^2 \to [0,+\infty), \qquad \tilde{\sfd}^2_{\f{C}}([x,r],[y,s]) \coloneqq |r-s|^2 + rs (\sfd(x,y) \wedge \pi/2) ^2 \,.
\end{equation}
Notice that $\tilde{\sfd}_{\f{C}}$ is well defined on $\f{C}[X]^2$ with 
\begin{equation}\label{eq: Handtildedonvertex}
\sfH_{\W\He}(\f{o}, [y,s])= \tilde{\sfd}^2_{\f{C}}(\f{o}, [y,s])=s^2 \quad \textrm{and} \quad \sfH_{\W\He}( [x,r],\f{o})= \tilde{\sfd}^2_{\f{C}}([x,r],\f{o})=r^2.
\end{equation}
Using formulation \eqref{eq:cos_to_sin}, it follows immediately that $\tilde{\sfd}_{\f{C}} \ge \mathsf{d}_{\pi/2, \f{C}}$ where $\mathsf{d}_{\pi/2, \f{C}}$ was defined in \eqref{ss22:eq:distcone}. \label{text:d_tilde}

In the next statement
\[
\begin{array}{ccccc}
\pi^i & : & \f{C}[X]^{N+1} & \to & \f{C}[X] \\[5pt]
& & (\f{y}_0,\dots,\f{y}_N) & \mapsto & \f{y}_i
\end{array}\,, 
\qquad i=0,\dots,N
\]
denote the projection maps on the cone and $\pi^{i,j} \coloneqq (\pi^i,\pi^j)$, $0 \leq i < j \leq N$. We use the notation $\mathcal{B}(Z)$ for the Borel sigma-algebra of a complete and separable metric space $Z$, and $\mathcal{B}(Z)^*$ for the sigma-algebra of \emph{universally measurable} subsets of $Z$. We refer the reader to Appendix \ref{sec:measappendix} for the details about universally measurable subsets. 

\begin{theorem}\label{thm:thmdinicolo} Let $(X, \sfd)$ be a complete, separable and geodesic metric space, and let $N \in \N_{\ge 1}$.
There exists a map $G_N : \f{C}[X]^{N+1} \to \f{C}[X]^{N+1}$ which is $\mathcal{B}(\f{C}[X]^{N+1})^*$-$\mathcal{B}(\f{C}[X]^{N+1})$ measurable such that 
\begin{equation}\label{eq:fermo_ai_bordi}
\pi^0 \circ G_N = \pi^0, \quad \pi^N \circ G_N = \pi^N \quad \textrm{and}\quad \sfH_N \circ G_N \le \sfH_N \  \textrm{on} \ \f{C}[X]^{N+1},  
\end{equation}
and such that, on the set $\mathcal{D}_N\coloneqq \{ \sfH_N <\tilde{\sfd}_{\f{C}}^2 \circ \pi^{0,N} \}\subset \f{C}[X]^{N+1}$, it holds

\begin{equation}\label{eq:radialboundG_N}
\mathsf{r} \circ \pi^i \circ G_N \in \left [ \Big(1-\frac{\pi}{4}\Big)\Big[(\mathsf{r} \circ \pi^0) \wedge (\mathsf{r} \circ \pi^N)\Big],  (\mathsf{r} \circ \pi^0) \vee (\mathsf{r} \circ \pi^N) \right ]
\end{equation}
for every $i=0, \dots, N,$ and

\begin{equation}\label{eq:spatialboundG_N}
\mathsf{d} ( \mathsf{x} ( \pi^{i-1} \circ G_N), \mathsf{x} ( \pi^{i} \circ G_N)) \in \left [ 0, \frac{1}{\sqrt{N}}\frac{\pi}{2-\frac{\pi}{2}}\frac{ \sqrt{(\mathsf{r} \circ \pi^0) (\mathsf{r} \circ \pi^N)}}{(\mathsf{r} \circ \pi^0) \wedge (\mathsf{r} \circ \pi^N) } \right ]
\end{equation}
for every $i=1, \dots, N$, where $\sfH_N$ is as in \eqref{eq:NsumH}.
\end{theorem}

\begin{proof}
We start by noticing that for every $(\f{y}_0,\dots,\f{y}_N)\in \mathcal{D}_N$ we have $\f{y}_i=[x_i,r_i]\neq \mathfrak{o}$ for $i=0,N$. Indeed, by contradiction, if $(\f{y}_0,\dots,\f{y}_N)\in \mathcal{D}_N$ with $[x_N,r_N]=\f{o}$ (the case $[x_0,r_0]=\f{o}$ is similar), using \eqref{eq: Handtildedonvertex} and Jensen's inequality we have
\[
r_0^2 = \tilde{\sfd}^2_{\f{C}}(\f{y}_0,\f{o}) > N\sum_{i=1}^N\sfH_{\W\He}(\f{y}_{i-1},\f{y}_i) \ge N\sum_{i=1}^N|r_{i-1}-r_i|^2 \ge \left(\sum_{i=1}^N|r_{i-1}-r_i|\right)^2 \ge r^2_0\,,
\]
which yields the contradiction.

For $\f{C}[X]\ni \f{y}_0,\f{y}_N\neq \f{o}$ we define the function $\sfH_N(\f{y}_0,\f{y}_N;\cdot):\f{C}[X]^{N-1}\to [0,+\infty)$ as $\sfH_N(\f{y}_0,\f{y}_N;\f{y}_1,\dots,\f{y}_{N-1})=\sfH_N(\f{y}_0,\f{y}_1,\dots,\f{y}_{N})$, and the sets 
\begin{align*}
\f{C}[X]_{\f{o}}^{N-1} & \coloneqq \{(\f{y}_1,\dots,\f{y}_{N-1}) \in \f{C}[X]^{N-1} \ : \ \exists i=1,\dots,N-1 \text{ such that } \f{y}_i=\f{o}\}\,, \\
\f{C}[X]_{>0}^{N-1} & \coloneqq \{(\f{y}_1,\dots,\f{y}_{N-1}) \in \f{C}[X]^{N-1} \ : \ \f{y}_i\neq \f{o} \ \forall i=1,\dots,N-1\}\,.
\end{align*}
We claim that:
\begin{equation}\label{claim:positiveradii}
\begin{aligned}
\text{if $\f{C}[X]\ni \f{y}_0,\f{y}_N\neq \f{o}$, $\forall (\f{y}_1,\dots,\f{y}_{N-1})$} &\text{$\in \f{C}[X]_{\f{o}}^{N-1}$ $\exists(\tilde{\f{y}}_1,\dots,\tilde{\f{y}}_{N-1})\in \f{C}[X]_{>0}^{N-1}$ such that}\\
\sfH_N(\f{y}_0,\f{y}_N;\f{y}_1,\dots,\f{y}_{N-1})&>\sfH_N(\f{y}_0,\f{y}_N;\tilde{\f{y}}_1,\dots,\tilde{\f{y}}_{N-1}).
\end{aligned}
\end{equation}
To show this, we firstly notice that for every $[x, r], [y, s]\in \f{C}[X]$ with $r>0$, we can find a point $\f{o}\neq [z,t]\in \f{C}[X]$ such that 
\begin{equation}\label{eq:betterstaypositive}
r^2+s^2=\sfH_{\W\He}([x,r], \f{o}) + \sfH_{\W\He}(\f{o}, [y,s]) > \sfH_{\W\He}([x,r],[z,t])+ \sfH_{\W\He}([z,t], [y,s]).  
\end{equation}
Indeed, using the explicit expression of $\sfH_{\W\He}$ it is easy to see that if $s=0$ we can take $z=x$ and any $t$ satisfying $0<t<r$. 
If $s>0$ then we can take $z=y$ and any $t$ satisfying $0<t<2\frac{r+s}{2+\sfd^2(x,y)}$.

So let $\f{y}_0=[x_0, r_0]$ and $\f{y}_N=[x_N, r_N]$ with $r_0, r_N>0$, $(\f{y}_1,\dots,\f{y}_{N-1})\in \f{C}[X]_{\f{o}}^{N-1}$, and we now prove claim \eqref{claim:positiveradii} by induction on $N$. The case $N=2$ is a direct consequence of \eqref{eq:betterstaypositive}. Let us assume the result true for $N-1\in \mathbb{N}, N>2$, and prove the case $N$. 
If $r_{1}>0$ we can apply the induction assumption to the function $\sfH_{N-1}(\f{y}_1,\f{y}_{N}; \cdot)$ to infer the existence of $(\tilde{\f{y}}_2,\dots,\tilde{\f{y}}_{N-1})\in \f{C}[X]_{>0}^{N-1}$ such that
$$\sfH_{N-1}(\f{y}_1,\f{y}_N;\f{y}_2,\dots,\f{y}_{N-1})>\sfH_{N-1}(\f{y}_1,\f{y}_N;\tilde{\f{y}}_2,\dots,\tilde{\f{y}}_{N-1}).$$
It is then immediate to see that $(\tilde{\f{y}}_1,\dots,\tilde{\f{y}}_{N-1})=(\f{y}_1,\tilde{\f{y}}_2,\dots,\tilde{\f{y}}_{N-1})\in \f{C}[X]_{>0}^{N-1}$ satisfies the claim \eqref{claim:positiveradii}.
 
In the case $r_1=0$ we firstly apply \eqref{eq:betterstaypositive} to replace $\f{y}_1$ with a point $\tilde{\f{y}}_1=[z,t]\neq \f{o}$ such that 
\[
\sfH_{\W\He}(\f{y}_0, \f{o}) + \sfH_{\W\He}(\f{o}, \f{y}_2) > \sfH_{\W\He}(\f{y}_0,\tilde{\f{y}}_1) + \sfH_{\W\He}(\tilde{\f{y}}_1, \f{y}_2)\,,
\]
and then to conclude we reason as above using the induction assumption on $\sfH_{N-1}(\tilde{\f{y}}_1,\f{y}_{N}; \cdot)$.

Now, given $\f{y}_0=[x_0, r_0]$ and $\f{y}_N=[x_N, r_N]$ with $r_0, r_N>0$, we let $d\coloneqq \sfd(x_0,x_N)$ and we define a map $T_N(\f{y}_0, \f{y}_N; \cdot) : \Omega_d^N \to \f{C}[X]^{N-1}$ ($\Omega_d^N$ being defined in \eqref{def:Omega_dN}) as $T_N(\f{y}_0, \f{y}_N; r_1, \dots, r_{N-1}, d_1, \dots, d_N)\coloneqq (\f{y}_1,\dots,\f{y}_{N-1})$ where for $i=1,\dots,N-1$ we have set $\mathsf{r}(\f{y}_i)\coloneqq r_i$ and 
$$\mathsf{x}(\f{y}_i)\coloneqq \begin{cases} \mathsf{x}(\f{y}_0) \quad & \text{ if } d=0, \\ \Gamma^X(x_0, x_N) \left (\frac{\sum_{j=1}^i d_j}{d} \right ) \quad &\text{ if } d>0, \end{cases}$$
with $\Gamma^X$ being the $\mathcal{B}(X \times X)^*-\mathcal{B}(\rmC([0,1]; (X, \sfd)))$ measurable geodesic selection map constructed in Theorem \ref{th:measgeod}. Notice that $\sfd(\mathsf{x}(\f{y}_{i-1}),\mathsf{x}(\f{y}_i))=d_i$ for every $i=1,\dots,N.$ Let $U_N:Y\to \R^{2N-1}$ be the Borel measurable map defined in Lemma \ref{lem: properties-f}, with $Y=(0,+\infty)^2\times [0,+\infty)$, and we define $G_N: \f{C}[X]^{N+1}\to \f{C}[X]^{N+1}$ as
\begin{equation}\label{def:G_N}
 G_N(\f{y}_0,\dots,\f{y}_N)\coloneqq \begin{cases}
        (\f{y}_0,\dots,\f{y}_N) \quad &\textrm{on } \f{C}[X]^{N+1}\setminus\mathcal{D}_N,\\
        \left(\f{y}_0,T_N(\f{y}_0,\f{y}_N;U_N(\mathsf{r}(\f{y}_0),\mathsf{r}(\f{y}_N),\sfd(\mathsf{x}(\f{y}_0),\mathsf{x}(\f{y}_N)))),\f{y}_N\right) \ &\textrm{on } \mathcal{D}_N.
    \end{cases}
\end{equation}
Since $\f{y}_0,\f{y}_N\neq \f{o}$ in $\mathcal{D}_N$ and since $U_N(r_0,r_N,d)$ belongs to $\Omega^N_d$, we have that $G_N$ is well defined. By definition $\pi^0 \circ G_N = \pi^0$ and $\pi^N \circ G_N = \pi^N$. Let us prove that $\sfH_N \circ G_N \le \sfH_N$ on $\f{C}[X]^{N+1}$: this is clear on $\f{C}[X]^{N+1}\setminus\mathcal{D}_N$, so it remains to prove that for every $(\f{y}_0,\dots,\f{y}_N)\in \mathcal{D}_N$
\begin{equation}\label{eq:G_Ncontractive}
    \sfH_N\big(\f{y}_0,T_N(\f{y}_0,\f{y}_N;U_N(\mathsf{r}(\f{y}_0),\mathsf{r}(\f{y}_N),\sfd(\mathsf{x}(\f{y}_0),\mathsf{x}(\f{y}_N)))),\f{y}_N\big) \le \sfH_N(\f{y}_0,\dots,\f{y}_N).  
    \end{equation}
As a consequence of claim \eqref{claim:positiveradii} we already know that there exists $(\tilde{\f{y}}_1,\dots,\tilde{\f{y}}_{N-1})\in \f{C}[X]_{>0}^{N-1}$ such that 
\[
\sfH_N(\f{y}_0,\tilde{\f{y}}_1,\dots,\tilde{\f{y}}_{N-1},\f{y}_N) \le \sfH_N(\f{y}_0,\f{y}_1,\dots,\f{y}_{N-1},\f{y}_N)\,.
\]
Let $\tilde{r}_i\coloneqq \mathsf{r}(\tilde{\f{y}}_i)$ for $i=1,\dots,N-1$ and $\tilde{d}_i\coloneqq \sfd(\mathsf{x}(\tilde{\f{y}}_{i-1}),\mathsf{x}(\tilde{\f{y}}_{i}))$ for $i=1,\dots,N,$ where we have put $\tilde{\f{y}}_0\coloneqq \f{y}_0$ and $\tilde{\f{y}}_N\coloneqq \f{y}_N$. Notice that $(\tilde{r}_1,\dots,\tilde{r}_{N-1},\tilde{d}_1,\dots,\tilde{d}_N)\in \Omega^N_{\ge d}$ since $\sfd$ satisfies the triangle inequality. Using the notation $f_N(r_0, r_N; \cdot)$ for the function defined in \eqref{eq:deff_N}, by definition of $f_N(r_0, r_N; \cdot), \sfH_N$ and $U_N$, it holds
\[
\begin{split}
\sfH_N & \big(\f{y}_0,T_N(\f{y}_0,\f{y}_N ; U_N(\mathsf{r}(\f{y}_0),\mathsf{r}(\f{y}_N),\sfd(\mathsf{x}(\f{y}_0),\mathsf{x}(\f{y}_N)))),\f{y}_N\big) = f_N\big(r_0,r_N;U_N(r_0,r_N,d)\big) \\ 
& \le f_N(r_0,r_N;\tilde{r}_1,\dots,\tilde{r}_{N-1},\tilde{d}_1,\dots,\tilde{d}_N) \\
& = \sfH_N(\f{y}_0,\tilde{\f{y}}_1,\dots,\tilde{\f{y}}_{N-1},\f{y}_N) \le \sfH_N(\f{y}_0,\f{y}_1,\dots,\f{y}_{N-1},\f{y}_N) 
\end{split}
\]
which gives \eqref{eq:G_Ncontractive}.

The bounds \eqref{eq:radialboundG_N} and \eqref{eq:spatialboundG_N} are true by construction, as a consequence of the corresponding bounds contained in \eqref{eq:radialandspatialbounds}. 

It remains to show that $G_N$ is $\mathcal{B}(\f{C}[X]^{N+1})^*$-$\mathcal{B}(\f{C}[X]^{N+1})$ measurable. Notice that $\mathcal{D}_N$ is a Borel subset of $\f{C}[X]^{N+1}$, so that it is sufficient to show that the map
$$(\f{y}_0,\dots,\f{y}_N)\mapsto\left(\f{y}_0,T_N(\f{y}_0,\f{y}_N;U_N(\mathsf{r}(\f{y}_0),\mathsf{r}(\f{y}_N),\sfd(\mathsf{x}(\f{y}_0),\mathsf{x}(\f{y}_N)))),\f{y}_N\right)$$
is $\mathcal{B}(\f{C}[X]^{N+1})^*$-$\mathcal{B}(\f{C}[X]^{N-1})$ measurable, and this is a consequence of the $\mathcal{B}(\f{C}[X]^{2})^*$-$\mathcal{B}(\f{C}[X]^{N+1})$ measurability of 
$$(\f{y}_0,\f{y}_N)\mapsto T_N(\f{y}_0,\f{y}_N;U_N(\mathsf{r}(\f{y}_0),\mathsf{r}(\f{y}_N),\sfd(\mathsf{x}(\f{y}_0),\mathsf{x}(\f{y}_N)))).$$
The latter follows from the fact that $\mathcal{B}$-$\mathcal{B}$ measurable maps and $\mathcal{B}^*$-$\mathcal{B}$ measurable maps are all $\mathcal{B}^*$-$\mathcal{B}^*$ measurable (see \eqref{eq:universally}) and that all the involved maps posses such measurability, as a consequence of Theorem \ref{th:measgeod} and Lemma \ref{lem: properties-f}.
\end{proof}

\subsection{Auxiliary results}

This subsection contains some of the auxiliary results mentioned in the strategy outline at the beginning of the section. The first result allows us to reduce, without loss of generality, the proof of Theorem \ref{thm:main1} to the compact case. This additional assumption will contribute to ensure tightness of the projections $\ssigma^k \coloneqq \pi^{0,N_k}_\sharp \aalpha^k$ as well as of the liftings $\ppi^k$.

\begin{proposition}\label{prop:compred} Assume that, whenever $m \in \N$, $K_m \subset \R^m$ is a convex and compact subset and $\nu_0, \nu_1 \in \meas_+(K_m)$, we have
\[ (\He_{2,m} \nabla \W_{2, \sfd_m})(\nu_0, \nu_1) \ge \HK_{\sfd_m}(\nu_0, \nu_1),\]
being $\sfd_m$ the (restriction to $K_m$ of the) distance induced by the infinity norm in $\R^m$ and $\He_{2,m}$ the Hellinger distance in $\meas_+(K_m)$.
 Then, if $(X, \sfd)$ is a complete and separable metric space, it also holds
\[ (\He_2 \nabla \W_{2, \sfd})(\mu_0, \mu_1) \ge \HK_{\sfd}(\mu_0, \mu_1) \quad \text{ for every } \mu_0, \mu_1 \in \meas_+(X).\]
\end{proposition}
\begin{proof}
    Let $\iota: (X, \sfd) \to (\ell^\infty(\N), \sfd_\infty)$ be the Kuratowski embedding \cite[Proposition 1.2.12]{pasqualetto}, where $\sfd_\infty$ is the distance induced by the $\|\cdot\|_\infty$ norm in $\ell^\infty(\N)$. For every $m\in \N$, we define the sets
    \begin{align*}
    C_m &\coloneqq \{ (a_i)_{i} \in \ell^\infty(\N) : a_i =0 \text{ for every } i >m \text{ and } a_i \in [-m, m] \text{ for every } i \le m \} \subset \ell^\infty(\N), \\
    K_m &\coloneqq [-m,m]^m \subset \R^m,
    \end{align*}
and the maps $\iota_m: C_m \to K_m$ and $A_m: \ell^\infty(\N) \to \ell^\infty(\N)$ as
\begin{align*}
    \iota_m((a_i)_{i \in\N}) &\coloneqq (a_1, \dots, a_m) , \quad &&(a_i)_i \in \ell^\infty(\N),\\
    A_m((a_i)_{i \in \N}) &\coloneqq (-m \vee a_1 \wedge m, -m \vee a_2 \wedge m, \dots, -m \vee a_m \wedge m, 0, 0, \dots), \quad &&(a_i)_i \in \ell^\infty(\N).
\end{align*}
Note that $\iota$ is an isometry from $(X, \sfd)$ to $(\ell^\infty, \sfd_\infty)$, $\iota_m$ is an isometry from $(C_m, \sfd_\infty)$ to $(K_m, \sfd_m)$ and  $A_m$ is a $1$-Lipschitz map w.r.t.~$\sfd_\infty$. We finally define $\mathsf{T}_m : \meas_+(X) \to \meas_+(K_m)$ as 
\[ \mathsf{T}_m(\mu) \coloneqq (\iota_m \circ A_m \circ \iota)_\sharp \mu, \quad \mu \in \meas_+(X).\]
By Lemma \ref{le:contraction}, $\mathsf{T}_m$ is a contraction, both w.r.t.~the Wasserstein and the Hellinger distances, in the sense that, for every $\mu_0, \mu_1 \in \meas_+(X)$, it holds
\begin{align*}
\W_{2, \sfd_m}(\mathsf{T}_m(\mu_0), \mathsf{T}_m(\mu_1)) &\le \W_{2, \sfd}(\mu_0, \mu_1), \\
\He_{2,m}(\mathsf{T}_m(\mu_0), \mathsf{T}_m(\mu_1)) &\le \He_2(\mu_0, \mu_1). 
\end{align*}
Let now $\mu_0, \mu_1 \in \meas_+(X)$ and let $P_k \in \mathscr{P}(\mu_0, \mu_1; N_k)$ with $N_k \to + \infty$ be a sequence of paths such that
\[ \mathcal{E}_{N_k}(P_k) \to (\He_2 \nabla \W_2)^2(\mu_0, \mu_1),\]
where $\mathcal{E}_{N_k}$ is as in \eqref{eq:energy} with $\sfc_1=\He_2$ and $\sfc_2=\W_{2,\sfd}$. If $P_k= (\mu_0^k, \mu_1^k, \dots, \mu_{N_k}^k; \nu_1^k, \dots, \nu_{N_k}^k)$, we define, for every $m,k \in \N$, the paths
\[ P_k^m\coloneqq (\mathsf{T}_m(\mu_0^k), \mathsf{T}_m(\mu_1^k), \dots, \mathsf{T}_m(\mu_{N_k}^k); \mathsf{T}_m(\nu_1^k), \dots, \mathsf{T}_m(\nu_{N_k}^k)) \in \mathscr{P}(\mathsf{T}_m(\mu_0), \mathsf{T}_m(\mu_1); N_k).\]
By the contractivity of $\mathsf{T}_m$, we deduce that
\[ \mathcal{E}^m_{N_k}(P_k^m) \le \mathcal{E}_{N_k}(P_k) \quad k,m \in \N,\]
where $\mathcal{E}_{N_k}^m$ is as in \eqref{eq:energy} with $\sfc_1=\He_{2,m}$ and $\sfc_2=\W_{2,\sfd_m}$. We thus have, for every $m \in \N$, that
\begin{align*}
(\He_2 \nabla \W_{2,\sfd})^2(\mu_0, \mu_1)&=\lim_k \mathcal{E}_{N_k}(P_k) \ge \liminf_k \mathcal{E}^m_{N_k}(P_k^m) \\
&\ge (\He_{2,m} \nabla \W_{2,\sfd_m})^2(\mathsf{T}_m(\mu_0), \mathsf{T}_m(\mu_1)) \ge \HK^2_{\sfd_m}(\mathsf{T}_m(\mu_0), \mathsf{T}_m(\mu_1)),
\end{align*}
where the last inequality holds by assumption.
Applying Proposition \ref{prop:iso} to $X\coloneqq C_m$, $\sfd_X=\sfd_\infty$, $Y\coloneqq K_m$, $\sfd_Y=\sfd_m$ and $\iota\coloneqq \iota_m$, and setting
\[ \mu_i^m \coloneqq (A_m \circ \iota)_\sharp \mu_i, \quad i=0,1,\]
we see that 
\[ \HK_{\sfd_m}(\mathsf{T}_m(\mu_0), \mathsf{T}_m(\mu_1))=\HK_{\sfd_m}((\iota_m)_\sharp \mu_0^m, (\iota_m)_\sharp \mu_1^m) = \HK_{\sfd_\infty}(\mu_0^m, \mu_1^m) \quad m \in \N.\]
We deduce that 
\begin{equation}\label{eq:almostthere}
    (\He_2 \nabla \W_{2,\sfd})(\mu_0, \mu_1) \ge \HK_{\sfd_\infty}(\mu_0^m, \mu_1^m) \quad \text{ for every } m \in \N.
\end{equation}
Now we note that $\mu_i^m \weakto \iota_\sharp \mu_i$, $i=0,1$, as a consequence of the pointwise convergence (and continuity) of $A_m$ to the identity in $\ell^\infty(\N)$. Passing to the $\liminf$ as $m\to + \infty$ the inequality \eqref{eq:almostthere} using the lower semicontinuity of $\HK$ coming from Theorem \ref{thm:omnibus}(2), we deduce that
\[ (\He_2 \nabla \W_{2,\sfd})(\mu_0, \mu_1) \ge \HK_{\sfd_\infty}(\iota_\sharp \mu_0, \iota_\sharp \mu_1).\]
A further application of Proposition \ref{prop:iso} with $X$, $\sfd_X=\sfd$, $Y=\ell^\infty(\N)$, $\sfd_Y=\sfd_\infty$ and $\iota$, gives the desired conclusion.
\end{proof}

In the second lemma we show that if we are given a weakly convergent sequence of measures $(\mu_n)_{n \in \N}$ and a splitting of the limit measure $\mu$, then we can find a `consistent' splitting for each $\mu_n$. This fact will be pivotal in the first step of the proof of Theorem \ref{thm:main1}, see p.\pageref{text:splitting}.

\begin{lemma}\label{lem:decomposition}
Let $(Y,\sfd)$ be a metric space and $\mu,\mu_n \in \meas_+(Y)$ be such that $\mu_n \rightharpoonup \mu$ as $n \to +\infty$. Assume that $\mu = \nu^1 + \nu^2$ with $\nu^1,\nu^2 \in \meas_+(Y)$. Then there exist a subsequence $(n_k)_{k \in \N}$, $n_k \uparrow +\infty$ as $k \to +\infty$, and $\nu_k^1,\nu_k^2 \in \meas_+(Y)$, $k \in \N$, such that
\[
\mu_{n_k} = \nu_k^1 + \nu_k^2 \qquad \textrm{and} \qquad \nu_k^i \rightharpoonup \nu^i \quad \textrm{as } k \to +\infty,\, \quad i=1,2\,.
\]
\end{lemma}
\begin{proof}
Let us denote by $f \coloneqq \frac{\d\nu^1}{\d\mu}$ the Radon--Nikod\'ym derivative of $\nu^1$ w.r.t.\ $\mu$ and let $(f_k)_{k \in \N} \subset \rmC_b(Y)$ be such that $f_k \to f$ in $L^1(\mu)$.  
Up to replacing $f_k$ with $f_k \wedge 1 \vee 0$, we can assume without loss of generality that $0 \leq f_k \leq 1$ for every $k \in \N$. By continuity and boundedness of $f_k$ it holds $f_k\mu_n \rightharpoonup f_k\mu$ as $n \to +\infty$ and by construction $f_k\mu \rightharpoonup f\mu = \nu^1$ as $k \to +\infty$, so that the metrizability of weak convergence on $\meas_+(Y)$ (cf.~\cite[Theorem 5.1.3]{bogachevwcm}) and a diagonal argument provide us with a subsequence $(n_k)_{k \in \N}$ such that $n_k \uparrow +\infty$ and $\nu_k^1 \coloneqq f_k \mu_{n_k} \rightharpoonup \nu^1$ as $k \to +\infty$. As for $\nu^2$, it is sufficient to consider $\nu_k^2 \coloneqq (1-f_k)\mu_{n_k}$. The fact that $0 \leq f_k \leq 1$ ensures $\nu_k^1,\nu_k^2 \in \meas_+(Y)$. Moreover, $\mu_{n_k} = \nu_k^1 + \nu_k^2$ holds by construction, and $\nu_k^2 = \mu_{n_k} - \nu_k^1 \rightharpoonup \mu - \nu^1 = \nu^2$ as $k \to +\infty$.
\end{proof}

Finally, in the next result we show that if we have probability measures $\ppi, \ppi_k$ concentrated on curves with radii uniformly bounded away from zero and $\ppi_k \rightharpoonup \ppi$, then a sort of discretized action (which is upper bounded by the energy $\mathcal{E}_{N_k}(P_k)$, see \eqref{eq:pinguino} and \eqref{eq:caimano}) controls from above in the limit the action $\mathcal{A}_2$, thus bridging between inf-convolution and $\HK$.

\begin{proposition}\label{prop:conva2} 
Let $(X, \sfd)$ be a complete and separable metric space. For some $\eps,R>0$ let $\ppi, \ppi_k \in \prob(\AC^2([0,1]; (\f{C}_{R,\eps}[X],\sfd_{\f{C}})))$, where $\f{C}_{R,\eps}[X]$ is as in \eqref{eq:cepsr}. Assume that $\ppi_k \weakto \ppi$ and $\sup_k \int \mathcal{A}_2 \de \ppi_k =: C < +\infty$. Then
\[ 
\liminf_{k \to + \infty} \int \int_0^1 u_k(\f{y},t)|\f{y}'(t)|^2_{\sfd_{\f{C}}} \de t \de \ppi_k(\f{y}) \ge \int \mathcal{A}_2 \de \ppi\,,
\]
where $\mathcal{A}_2$ is as in \eqref{eq:a2def} and $u_k$ is defined as
\begin{equation}\label{eq:uk}
u_k(\f{y},t) \coloneqq 1 \wedge \frac{r_{\f{y}}(\ceiling{t N_k}/N_k)}{r_{\f{y}}(\floor{t N_k}/N_k)}, \qquad \f{y} \in \rmC([0,1];(\f{C}_{\Theta, \eps}[X],\sfd_{\f{C}})), \, t \in [0,1] 
\end{equation}
for a sequence $\N \ni N_k \uparrow +\infty$.
\end{proposition}

\begin{proof} 
We set
\[ 
\mathcal{A}_{2,k}(\f{y}) \coloneqq \int_0^1 u_k(\f{y},t)|\f{y}'(t)|^2_{\sfd_{\f{C}}} \de t, \qquad \f{y} \in \AC^2([0,1]; (\f{C}_{R, \eps}[X],\sfd_{\f{C}}))\,.
\]
We start by proving the following claim:
for every compact subset $B \subset \rmC([0,1]; (\f{C}[X],\sfd_{\f{C}}))$ and every $\gamma>0$ there exists $K = K(\eps, B, \gamma) \in \N$ such that
\[ 
\left |\mathcal{A}_{2, k}(\f{y}) - \mathcal{A}_2(\f{y})  \right | < \gamma \mathcal{A}_2(\f{y}) \,\text{ for every } \f{y} \in \AC^2([0,1]; (\f{C}_{R, \eps}[X],\sfd_{\f{C}})) \cap B, \, k \ge K. 
\]
To see this, we observe that, by compactness of $B$, for every $\eta>0$ there exists a $\delta=\delta(\eta, B)>0$ such that $|r_{\f{y}}(t)-r_{\f{y}}(t')| \le \eta$ for every $t,t' \in [0,1]$ with $|t-t'|<\delta$ and for every $\f{y} \in B$. Choosing $\eta\coloneqq \gamma/\eps$, we can find $K=K(\eps,B, \gamma) \in \N$ large enough such that $2/N_k \le \delta(\eta, B)$ whenever $k \ge K$. Hence, for $k \ge K$, $\f{y} \in \AC^2([0,1]; (\f{C}_{R, \eps}[X],\sfd_{\f{C}})) \cap B$ and $t \in [0,1]$, we have
\[ 
\left |1 - u_k(\f{y}, t) \right | \le \left | 1- \frac{r_{\f{y}}(\ceiling{t N_k} /N_k)}{r_{\f{y}}(\floor{t N_k}/N_k)}  \right | \le \frac{\left | r_{\f{y}}(\floor{t/N_k}N_k)-r_{\f{y}}(\ceiling{t/N_k}N_k)  \right |}{\eps} \le \gamma
\]
which also proves the claim.

Let us now conclude the proof of the proposition. Let $\gamma>0$ be fixed. By tightness of $(\ppi_k)_k$ we can find a compact set $B_\gamma \subset \rmC([0,1]; (\f{C}[X], \sfd_{\f{C}}))$ such that $\ppi_k(B^c_\gamma) < \gamma$ for every $k \in \N$. By the previous claim, we can find $K=K(\eps, B_\gamma, \gamma) \in N$ such that
\[ 
\left |\mathcal{A}_{2, k}(\f{y}) - \mathcal{A}_2(\f{y})  \right | < \gamma \mathcal{A}_2(\f{y}) \quad \text{for every } \f{y} \in \AC^2([0,1]; (\f{C}_{R, \eps}[X],\sfd_{\f{C}})) \cap B_\gamma, \, k \ge K \,. 
\]
We fix a threshold $m \in \N$ and we obtain for every $k \ge K$ that 
\begin{align*}
    \int \mathcal{A}_{2, k} \de \ppi_k &\ge \int \left (\mathcal{A}_{2, k} \wedge m \right ) \de \ppi_k = \int \left ( \mathcal{A}_{2, k} \wedge m - \mathcal{A}_2 \wedge m \right )\de \ppi_k + \int \left ( \mathcal{A}_2 \wedge m \right )\de \ppi_k \\
    & \ge - \int \left | \mathcal{A}_{2, k} \wedge m - \mathcal{A}_2 \wedge m  \right | \de \ppi_k + \int \left ( \mathcal{A}_2 \wedge m \right )\de \ppi_k \\
    & = - \int_{B_\gamma} \left | \mathcal{A}_{2, k} \wedge m - \mathcal{A}_2 \wedge m  \right | \de \ppi_k -\int_{B_\gamma^c} \left | \mathcal{A}_{2, k} \wedge m - \mathcal{A}_2 \wedge m  \right | \de \ppi_k + \int \left ( \mathcal{A}_2 \wedge m \right )\de \ppi_k \\
    & \ge - \int_{B_\gamma} \left | \mathcal{A}_{2, k} - \mathcal{A}_2 \right | \de \ppi_k - 2m\gamma  + \int \left ( \mathcal{A}_2 \wedge m \right )\de \ppi_k \\
    & \ge - \gamma \int_{B_\gamma} \mathcal{A}_2 \de \ppi_k - 2m\gamma  + \int \left ( \mathcal{A}_2 \wedge m \right )\de \ppi_k  \\
    & \ge - \gamma (C + 2m) + \int \left ( \mathcal{A}_2 \wedge m \right )\de \ppi_k \,.
\end{align*}
Passing to the limit inferior as $k \to +\infty$, using the weak convergence of $\ppi_k$ to $\ppi$ and the lower semicontinuity of $\mathcal{A}_2$, we get
\[ 
\liminf_{k \to + \infty} \int \mathcal{A}_{2, k} \de \ppi_k \ge - \gamma(C+2m) + \int \left ( \mathcal{A}_2 \wedge m \right )\de \ppi,
\]
so that passing first to the limit as $\gamma \downarrow 0$ and then to the limit as $m \uparrow + \infty$ conclude the proof.
\end{proof}

\subsection{Proof of Theorem \ref{thm:main1}}

Let us fix $\mu_0,\mu_1 \in \meas_+(X)$ and let us recall the definition of the measures $\aalpha^k$ provided in \eqref{eq:aalpha_k}. By Proposition \ref{prop:compred} it is not restrictive to assume $(X,\sfd)$ to be a \emph{compact} geodesic metric space. Moreover, from \eqref{eq:energy_to_HeW} and \eqref{eq:pinguino}, we only need to prove 
\[
\liminf_{k \to +\infty} \int \sfH_{N_k} \de\aalpha^k \geq \HK^2(\mu_0,\mu_1) \,.
\]
As anticipated, this will be achieved in four steps.

\medskip

\noindent\textbf{1st step}. Let us define $\ssigma^k \coloneqq \pi^{0,N_k}_\sharp \aalpha^k$ and assume, up to extracting a subsequence, that $\ssigma^k \rightharpoonup \ssigma$ as $k \to +\infty$. This is possible because all measures $\ssigma^k$ are supported on the compact set $\f{C}_\Theta[X]^2$ (recall that $\aalpha^k$ is supported in $\f{C}_\Theta[X]^{N_k+1}$ with $\Theta$ as in \eqref{eq:theta} and $X$ is compact) and are probabilities, hence are tight. Then, let us decompose $\ssigma$ into $\ssigma_>$ and $\ssigma_{\f{o}}$, i.e.\ $\ssigma = \ssigma_> + \ssigma_{\f{o}}$ where
\[
\ssigma_> \coloneqq \ssigma \mres (\f{C}[X] \setminus \{\f{o}\} \times \f{C}[X] \setminus \{\f{o}\}) \qquad \textrm{and} \qquad \ssigma_{\f{o}} \coloneqq (\f{C}[X] \times \{\f{o}\} \cup \{\f{o}\} \times \f{C}[X]) \,.
\]
Up to extracting a further subsequence, by Lemma \ref{lem:decomposition}\label{text:splitting} there exist $\ssigma_>^k,\ssigma_{\f{o}}^k \in \meas_+(\f{C}_\Theta[X]^2)$ such that $\ssigma^k = \ssigma_>^k + \ssigma_{\f{o}}^k$ for every $k \in \N$, $\ssigma_>^k \rightharpoonup \ssigma_>$ and $\ssigma_{\f{o}}^k \rightharpoonup \ssigma_{\f{o}}$ as $k \to +\infty$. In order to decompose each $\aalpha^k$ accordingly, if we disintegrate $\aalpha^k$ w.r.t.\ $\pi^{0,N_k}$ we obtain a family of probability measures $(\eeta^k_{\f{y}_0,\f{y}_{N_k}})_{\f{y}_0,\f{y}_{N_k} \in \f{C}[X]} \subset \prob(\f{C}[X]^{N_k+1})$ such that $\eeta^k_{\f{y}_0,\f{y}_{N_k}}$ is concentrated on $\{\f{y}_0\} \times \f{C}[X]^{N_k-1} \times \{\f{y}_{N_k}\}$ for $\ssigma^k$-a.e.\ $(\f{y}_0,\f{y}_{N_k}) \in \pc$ and
\[
\aalpha^k = \int_{\pc} \eeta^k_{\f{y}_0,\f{y}_{N_k}}\,\d\ssigma^k(\f{y}_0,\f{y}_{N_k}) \,.
\]
Hence, the natural way to decompose $\aalpha^k$ is given by $\aalpha^k = \aalpha_>^k + \aalpha_{\f{o}}^k$ with
\[
\aalpha_>^k \coloneqq \int_{\pc} \eeta^k_{\f{y}_0,\f{y}_{N_k}}\,\d\ssigma_>^k(\f{y}_0,\f{y}_{N_k}) \qquad \textrm{and} \qquad \aalpha_{\f{o}}^k \coloneqq \int_{\pc} \eeta^k_{\f{y}_0,\f{y}_{N_k}} \,\d\ssigma_{\f{o}}^k(\f{y}_0,\f{y}_{N_k}) \,,
\]
as in such a way $\pi^{0,N_k}_\sharp \aalpha_>^k = \ssigma_>^k$ and $\pi^{0,N_k}_\sharp \aalpha_{\f{o}}^k = \ssigma_{\f{o}}^k$. Let us further split the measure $\aalpha_>^k$ into
\[
\aalpha^k_g \coloneqq \aalpha_>^k \mres \mathcal{D}_{N_k}^c \qquad \textrm{and} \qquad \aalpha^k_b \coloneqq \aalpha_>^k \mres \mathcal{D}_{N_k}\,,
\]
where we recall
\[
\mathcal{D}_{N_k}= \left\{\sfH_{N_k} < \tilde{\sfd}_{\f{C}}^2 \circ \pi^{0,N_k} \right\} \subset \f{C}[X]^{N_k+1}\,,
\]
see also \eqref{eq:NsumH} and \eqref{eq: fakeconedistance}.
Let us observe, for later use, that if $(\f{y}_0, \dots, \f{y}_{N_k}) \in \mathcal{D}_{N_k}$, then $\sfr(\f{y}_0) > 0$ and $\sfr(\f{y}_{N_k}) > 0$, as shown at the beginning of the proof of Theorem \ref{thm:thmdinicolo}. This implies that $\aalpha_b^k$ is actually concentrated on
\begin{equation}\label{eq:supp_positive_radius}
\left \{ (\f{y}_0, \dots, \f{y}_{N_k}) \,:\, 
\begin{array}{c}
(\sfr(\f{y}_0), \dots, \sfr(\f{y}_{N_k})) \in (0,\Theta] \times [0,\Theta]^{N_k-1} \times (0,\Theta] \\[10pt]
\sfH_{N_k}(\f{y}_0, \dots, \f{y}_{N_k}) < \tilde\sfd_{\f{C}}^2 (\f{y}_0,\f{y}_{N_k})
\end{array}
\right\} \subset \f{C}[X]^{N_k+1}\,.
\end{equation}
As already remarked for $(\ssigma^k)_{k \in \N}$, also $\ssigma^k_g \coloneqq \pi^{0,N_k}_\sharp \aalpha^k_g$ and $\ssigma^k_b \coloneqq \pi^{0,N_k}_\sharp \aalpha^k_b$ are tight families of measures. We can thus assume that, up to extracting subsequences, $\ssigma^k_x$ and their 2-homogeneous marginals $\mu_{x,0}^k \coloneqq \f{h}_0^2(\ssigma^k_x)$, $\mu_{x,1}^k \coloneqq \f{h}_1^2(\ssigma^k_x)$ converge weakly to some $\ssigma_x \in \f{H}^2(\mu_{x,0}, \mu_{x,1})$, $\mu_{x,0}$, and $\mu_{x,1}$ respectively as $k \to + \infty$, where $x \in \{g,b,\f{o}\}$. As a consequence, 
\begin{equation}\label{eq:good_marginals}
\mu_{g,0}+\mu_{b,0}+\mu_{\f{o},0} = \mu_0 \qquad \textrm{and} \qquad \mu_{g,1}+\mu_{b,1}+\mu_{\f{o},1} = \mu_1\,.
\end{equation}
Indeed, by construction for all $k \in \N$ it holds
\[
\mu_{g,i}^k + \mu_{b,i}^k + \mu_{\f{o},i}^k = \f{h}^2_i(\ssigma^k_g + \ssigma^k_b + \ssigma_{\f{o}}^k) = \f{h}_i^2(\pi^{0,N_k}_\sharp \aalpha^k) = \mu_i, \qquad i=0,1\,.
\]
After this premise, note that by Jensen's inequality\label{text:caimano}
\[
\begin{split}
\int \sfH_{N_k} \de \aalpha^k_{\f{o}} & = N_k \int \sum_{i=1}^{N_k}\sfH_{\W\He}(\f{y}_{i-1},\f{y}_i) \de \aalpha^k_{\f{o}} \ge N_k\int \sum_{i=1}^{N_k}|r_{i-1}-r_i|^2 \de \aalpha^k_{\f{o}} \\
& \ge \int \left(\sum_{i=1}^{N_k}|r_{i-1}-r_i|\right)^2 \de \aalpha^k_{\f{o}} \ge \int |r_0 - r_{N_k}|^2 \de \aalpha^k_{\f{o}} \\
& = \int_{\pc} |r_0 - r_1|^2 \de \ssigma^k_{\f{o}} \,,
\end{split}
\]
so that
\[
\begin{split}
\liminf_{k \to +\infty} \int \sfH_{N_k} \de \aalpha^k_{\f{o}} & \geq \liminf_{k \to +\infty} \int_{\pc} |r_0 - r_1|^2 \de \ssigma^k_{\f{o}} = \int_{\pc} |r_0 - r_1|^2 \de \ssigma_{\f{o}} \\
& = \int_{\pc} \sfd_{ \pi/2,\f{C}}^2 \,\de \ssigma_{\f{o}} \ge \HK^2(\mu_{\f{o},0},\mu_{\f{o},1}) \,,
\end{split}
\]
where the last identity comes from the definition of the cone distance \eqref{ss22:eq:distcone} in the particular case when one of the two points is the vertex: note indeed that by construction $\ssigma_{\f{o}}$ is concentrated on points like $(\f{o},\f{y})$ or $(\f{y},\f{o})$, so that either $r_0$ or $r_1$ is equal to zero.

Moreover, on the `good' set $\mathcal{D}_{N_k}^c$ we easily get the following lower bound
\[
\int \sfH_{N_k} \de \aalpha^k_g \geq \int_{\f{C}[X,X]} \tilde{\sfd}_{\f{C}}^2 \de \ssigma^k_g \geq \int_{\f{C}[X,X]} \sfd_{\pi/2,\f{C}}^2 \de \ssigma^k_g \geq \HK^2(\mu_{g,0}^k,\mu_{g,1}^k)\,,
\]
where $\mathsf{d}_{\pi/2, \f{C}}$ was defined in \eqref{ss22:eq:distcone} and, as pointed out right after \eqref{eq: Handtildedonvertex}, it is dominated by $\tilde{\sfd}_{\f{C}}$. Therefore, if we are able to show that
\begin{equation}\label{eq:final_goal}
\liminf_{k \to +\infty} \int \sfH_{N_k} \de \aalpha^k_b \geq \HK^2(\mu_{b,0},\mu_{b,1})\,,
\end{equation}
then the conclusion immediately follows observing that
\begin{align*}
\liminf_{k \to +\infty} \int \sfH_{N_k} \de \aalpha^k & \geq \liminf_{k \to +\infty} \int \sfH_{N_k} \de \aalpha^k_g + \liminf_{k \to +\infty} \int \sfH_{N_k} \de \aalpha^k_b + \liminf_{k \to +\infty} \int \sfH_{N_k} \de \aalpha^k_{\f{o}} \\
& \geq \HK^2(\mu_{g,0},\mu_{g,1}) + \HK^2(\mu_{b,0},\mu_{b,1}) + \HK^2(\mu_{\f{o},0},\mu_{\f{o},1}) \geq \HK^2(\mu_{0},\mu_{1})\,,
\end{align*}
where the second inequality relies on the weak lower semicontinuity of $\HK$ together with $\mu_{g,i}^k \rightharpoonup \mu_{g,i}$ as $k \to +\infty$ for $i=0,1$; the last inequality exploits \eqref{eq:good_marginals} and the subadditivity of $\HK^2$ (see Theorem \ref{thm:omnibus}(2)). For this reason, let us now focus on \eqref{eq:final_goal}. 

\medskip

\noindent\textbf{2nd step.} For every $m \in \N_{\ge 1}$, let us consider an increasing function $f_m \in \rmC^\infty([0,+\infty))$ such that 
\[ 
f_m \equiv 0 \text{ in } [0, 1/m]\,, \qquad f_m \equiv 1 \text{ in } [2/m, +\infty)\,, 
\]
and let us set
\[ 
\aalpha^{k,m}_b \coloneqq (f_m \circ \sfr \circ \pi^0) (f_m \circ \sfr \circ \pi^{N_k})\aalpha_b^k \,. 
\]
As preliminary considerations, note that $\aalpha^k_b \geq \aalpha^{k,m}_b$, whence
\begin{equation}\label{eq:ineq_m}
\int \sfH_{N_k} \de \aalpha^k_b \ge \int \sfH_{N_k} \de \aalpha_b^{k,m}\,, \qquad \forall k,m \in \N_{\geq 1}\,.
\end{equation}
It is also easy to check that $\aalpha_b^{k,m}$ is concentrated on 
\begin{equation}\label{eq:bad_support}
\left \{ (\f{y}_0, \dots, \f{y}_{N_k}) \,:\, 
\begin{array}{c}
(\sfr(\f{y}_0), \dots, \sfr(\f{y}_{N_k})) \in [m^{-1}, \Theta] \times [0,\Theta]^{N_k-1} \times [m^{-1}, \Theta] \\[10pt]
\sfH_{N_k}(\f{y}_0, \dots, \f{y}_{N_k}) < \tilde\sfd_{\f{C}}^2 (\f{y}_0,\f{y}_{N_k})
\end{array}
\right\} \subset \f{C}[X]^{N_k+1}\,.
\end{equation}
Moreover, if we introduce $\ssigma_b^{k,m} \coloneqq \pi_\sharp^{0,N_k} \aalpha_b^{k,m}$ and its 2-homogneous marginals 
\[ 
\nu_0^{k,m} \coloneqq \f{h}^2(\pi^0_\sharp \aalpha_b^{k,m}), \qquad \nu_1^{k,m} \coloneqq \f{h}^2(\pi^{N_k}_\sharp \aalpha_b^{k,m})\,,
\]
it is not difficult to see that, as $k \to +\infty$, they converge weakly to measures
\[
\bar{\ssigma}_b^m \coloneqq (f_m \circ \sfr \circ \pi^0) (f_m \circ \sfr \circ \pi^{1})\ssigma_b
\]
and
\[ 
\nu_0^m \coloneqq \f{h}^2 ( \pi^0_\sharp \bar{\ssigma}_b^m)\,, \qquad \nu_1^m \coloneqq \f{h}^2 ( \pi^1_\sharp \bar{\ssigma}_b^m)  
\]
respectively, satisfying $\bar{\ssigma}_b^m \weakto \ssigma_b$ and $\nu_i^m \weakto \mu_{b,i}$, $i=0,1$, as $m \to + \infty$. Indeed, by construction $\ssigma_b$ does not charge the set $\f{C}[X] \times \{\f{o}\} \cup \{\f{o}\} \times \f{C}[X]$, so that $(f_m \circ \sfr \circ \pi^0) (f_m \circ \sfr \circ \pi^{1}) \to 1$ $\ssigma_b$-a.e.\ as $m \to +\infty$. Since $|f_m| \leq 1$ for all $m \in \N_{\geq 1}$, by Lebesgue's dominated convergence theorem we deduce that, for every $\varphi \in \rmC_b(\pc)$,
\[
\lim_{m \to +\infty}\int_{\pc} \varphi \,\de\bar{\ssigma}_b^m = \lim_{m \to +\infty}\int_{\pc} \varphi \,(f_m \circ \sfr \circ \pi^0) (f_m \circ \sfr \circ \pi^{1}) \,\de\ssigma_b = \int_{\pc} \varphi \,\de\ssigma_b \,,
\]
as desired.

We now claim that it is sufficient to prove that
\begin{equation}\label{eq:final_goal_2}
\liminf_{k \to +\infty} \int \sfH_{N_k} \de \aalpha_b^{k,m} \ge \HK^2(\nu_0^m, \nu_1^m)\,, \qquad \forall m \in \N_{\geq 1}\,.
\end{equation}
Indeed, by \eqref{eq:ineq_m} and \eqref{eq:final_goal_2} we have
\[ 
\liminf_{k \to +\infty} \int \sfH_{N_k} \de \aalpha^k_b \ge \liminf_{k \to +\infty} \int \sfH_{N_k} \de \aalpha_b^{k,m} \ge \HK^2(\nu_0^m, \nu_1^m)\,, \qquad \forall m \in \N_{\geq 1}\,,
\]
so that passing to the limit as $m \to +\infty$ and using the convergence of $\nu_i^m$ to $\mu_{b,i}$, for $i=0,1$, together with the joint weak lower semicontinuity of $\HK^2$ lead to \eqref{eq:final_goal}.

\medskip

\noindent\textbf{3rd step.} Thus, we fix $m \in \N$ and we devote the remaining part of the proof to show \eqref{eq:final_goal_2}, passing first of all to a non-relabeled subsequence such that the limit inferior in \eqref{eq:final_goal_2} is achieved as a limit. Second, we observe that the weak convergence of $\ssigma_b^{k,m}$ towards $\bar{\ssigma}_b^m$ implies the convergence of the masses:
\[
\ssigma_b^{k,m}(\f{C}[X,X]) \to \bar{\ssigma}_b^m(\f{C}[X,X]) \qquad \textrm{as } k \to +\infty\,.
\]
If $\bar{\ssigma}_b^m(\f{C}[X,X])=0$, then $\nu_0^m$ and $\nu_1^m$ are both the null measure and in this case \eqref{eq:final_goal_2} trivially holds. We can thus assume $\bar{\ssigma}_b^m(\f{C}[X,X])>0$. This implies that (up to passing to a subsequence) the mass of all $\ssigma_b^{k,m}$ is bounded from below by a strictly positive constant, uniformly in $k$. As on the other hand $\ssigma_b^{k,m}(\f{C}[X,X]) \leq 1$ (because $\ssigma_b^{k,m}$ is obtained via restriction/cut-off of the probability measure $\aalpha^k$), we conclude that there exists a constant $c_m \in (1,+\infty)$ such that
\begin{equation}\label{eq:uniform_mass}
1 \leq c_{k,m} \coloneqq \left(\ssigma_b^{k,m}(\f{C}[X,X])\right)^{-1/2} \leq c_m\,, \qquad \forall k \in \N\,.
\end{equation}
We now make use of Theorem \ref{thm:thmdinicolo} to modify $\aalpha_b^{k,m}$: leveraging the map $G_{N_k}$ introduced therein we define a new plan
\[
\tilde\aalpha_b^{k,m} \coloneqq (G_{N_k})_\sharp \aalpha_b^{k,m}
\]
and note that $\pi^{0,N_k}_\sharp \tilde\aalpha_b^{k,m} \in \f{H}^2(\nu_0^{k,m}, \nu_1^{k,m})$, because of $\pi^0 \circ G_{N_k} = \pi^0$ and $\pi^{N_k} \circ G_{N_k} = \pi^{N_k}$, see \eqref{eq:fermo_ai_bordi}. Using the fact that $H_{N_k} \circ G_{N_k} \leq H_{N_k}$, again by \eqref{eq:fermo_ai_bordi}, we also remark that
\begin{equation}\label{eq:intermediate}
\int \sfH_{N_k} \de \aalpha_b^{k,m} \geq \int \sfH_{N_k} \de \tilde\aalpha_b^{k,m}\,.
\end{equation} 
Finally, observe that $\tilde\aalpha_b^{k,m}$ is concentrated on
\[ 
\left \{ (\f{y}_0, \dots, \f{y}_{N_k}) \,:\, 
\begin{array}{c}
(\sfr(\f{y}_0), \dots, \sfr(\f{y}_{N_k})) \in \left[(5m)^{-1}, \Theta\right]^{N_k+1} \\[10pt]
\sfd(\sfx(\f{y}_{i-1}),\sfx(\f{y}_i)) \in \left[0,\frac{8\Theta m}{\sqrt{N_k}}\right]\,, \quad i=1,\dots,N_k
\end{array}
\right\}\,.
\]
Indeed, the fact that $\aalpha_b^{k,m}$ is concentrated on \eqref{eq:bad_support} allows to apply the radial and spatial quantitative bounds \eqref{eq:radialboundG_N} and \eqref{eq:spatialboundG_N}. For ease of computations, we slightly weaken them, observing that $\frac15 < 1-\frac{\pi}{4}$ and $\frac{\pi}{2-\pi/2} < 8$. Then, since $\Theta$ and $m$ are fixed, up to choosing $k$ sufficiently large we can assume that $\sfd(\sfx(\f{y}_{i-1}),\sfx(\f{y}_i)) \in [0,1]$ for every $i=1,\dots,N_k$. 

As next step, we apply to $\tilde\aalpha_b^{k,m}$ a dilation of parameter $c_{k,m}$ (with $c_{k,m}$ defined in \eqref{eq:uniform_mass}, see also \eqref{eq:dilation} for the definition of $\text{dil}_{c_{k,m}, 2}$):
\[
\hat\aalpha_b^{k,m} \coloneqq \text{dil}_{c_{k,m}, 2}(\tilde\aalpha_b^{k,m}).
\]
This further measure is concentrated on 
\[ 
\left \{ (\f{y}_0, \dots, \f{y}_{N_k}) \,:\, 
\begin{array}{c}
(\sfr(\f{y}_0), \dots, \sfr(\f{y}_{N_k})) \in \left[(5mc_m)^{-1}, \Theta\right]^{N_k+1} \\[10pt]
\sfd(\sfx(\f{y}_{i-1}),\sfx(\f{y}_i)) \in \left[0,1\right]\,, \quad i=1,\dots,N_k
\end{array}
\right\}\,,
\]
since $1 \leq c_{k,m} \leq c_m$ by \eqref{eq:uniform_mass} and dilations only act radially: this implies that the conditions $\sfd(\sfx(\f{y}_{i-1}),\sfx(\f{y}_i)) \in [0,1]$, $i=1,\dots,N_k$, valid in the support of $\tilde\aalpha_b^{k,m}$ still hold in the support of $\hat\aalpha_b^{k,m}$. Moreover, let us observe that $\hat\aalpha_b^{k,m} \in \prob(\f{C}[X]^{N_k+1})$, because
\[
\hat{\aalpha}_b^{k,m}(\f{C}[X]^{N_k+1}) = \int c_{k,m}^2 \de\tilde{\aalpha}_b^{k,m} = \int_{\pc} c_{k,m}^2 \de\ssigma_b^{k,m} = 1\,,
\]
and $\pi^{0,N_k}_\sharp \hat\aalpha_b^{k,m} \in \f{H}^2(\nu_0^{k,m}, \nu_1^{k,m})$, since homogeneous marginals are invariant w.r.t.\ dilations, see \eqref{eq:dil_invariance}. By \eqref{eq:hqhom} we also have
\[
\int \sfH_{N_k} \de \tilde\aalpha_b^{k,m} = \int \sfH_{N_k} \de \hat\aalpha_b^{k,m}
\]
and combining this identity with \eqref{eq:intermediate} yields
\begin{equation}\label{eq:chain_HNk}
\int \sfH_{N_k} \de \aalpha_b^{k,m} \geq \int \sfH_{N_k} \de \tilde\aalpha_b^{k,m} = \int \sfH_{N_k} \de \hat\aalpha_b^{k,m}\,.
\end{equation}
In view of \eqref{eq:final_goal_2}, it is thus sufficient to prove
\begin{equation}\label{eq:final_goal_3}
\liminf_{k \to +\infty} \int \sfH_{N_k} \de \hat\aalpha_b^{k,m} \geq \HK^2(\nu_0^m,\nu_1^m), \qquad \forall m \in \N_{\geq 1},
\end{equation}
and this is the purpose of the fourth and final step of the proof.

\medskip

\noindent\textbf{4th step.} Let us now lift $\hat\aalpha_b^{k,m}$ to the path space. To this end, let us define the geodesic interpolation map $\Gamma^{N_k} : \f{C}[X]^{N_k+1} \to \AC^2([0,1]; (\f{C}[X], \sfd_{\f{C}}))$ as
\begin{equation}\label{eq:geo_interpolation}
\begin{split}
\Gamma^{N_k}(\f{y}_0,\dots,\f{y}_{N_k})(t) \coloneqq \Gamma^X(\f{y}_{i-1},\f{y}_i)(N_k t - (i-1)), \qquad \textrm{if} \quad & \frac{i-1}{N_k} \leq t \leq \frac{i}{N_k} \\[5pt]
& i=1,\dots,N_k,
\end{split}
\end{equation}
where $\Gamma^X$ is the $\mathcal{B}(X \times X)^*$-$\mathcal{B}(\rmC([0,1]; (\f{C}[X], \sfd_{\f{C}})))$ measurable geodesic selection map constructed in Theorem \ref{th:measgeod}. Then
\[
\ppi^{k,m} \coloneqq \Gamma^{N_k}_\sharp \hat\aalpha_b^{k,m} \in \prob(\AC^2([0,1]; (\f{C}[X], \sfd_{\f{C}})))\,,
\]
and note that by construction and by Remark \ref{rem:geocone} the plan $\ppi^{k,m}$ is concentrated on piecewise geodesic curves taking values in $\f{C}_{\Theta,\eps_m}[X]$ (cf.~\eqref{eq:cepsr}), i.e.\ $\ppi^{k,m} \in \prob(\AC^2([0,1]; (\f{C}_{\Theta,\eps_m}[X],\sfd_{\f{C}})))$, where 
\[ 
\eps_m \coloneqq \frac{1}{5\sqrt{2}m c_m}\,.
\]
Indeed, the fact that $\hat\aalpha_b^{k,m}$ is concentrated on points $(\f{y}_0, \dots, \f{y}_{N_k}) \in \f{C}[X]^{N_k+1}$ with $\sfr(\f{y}_i) \geq (5mc_m)^{-1} > 0$ for all $i=0,\dots,N_k$ and $\sfd(\sfx(\f{y}_{i-1}),\sfx(\f{y}_i)) \leq 1 \leq \pi/2$ for all $i=1,\dots,N_k$ ensures, by \eqref{eq:rmin_nicer}, that the minimal radius of any constant speed geodesic connecting $\f{y}_{i-1}$ and $\f{y}_i$ is larger than $\frac{1}{\sqrt{2}} \sfr(\f{y}_{i-1}) \wedge \sfr(\f{y}_i) \geq \frac{1}{\sqrt{2}}(5mc_m)^{-1} = \eps_m$.

Recalling the definition \eqref{eq:a2def} of $\mathcal{A}_2$, we also have
\begin{equation}\label{eq:A2_bound}
\begin{split}
    \int \mathcal{A}_2 \de \ppi^{k,m} & = \int N_k \sum_{i=1}^{N_k} \sfd_{\f{C}}^2(\sfe_{(i-1)/N_k}, \sfe_{i/N_k}) \de \ppi^{k,m} \\ 
    & = \int N_k \sum_{i=1}^{N_k} \sfd_{\f{C}}^2(\pi^{i-1},\pi^i) \de \hat\aalpha_b^{k,m} \le 2 \int \sfH_{N_k} \de \hat\aalpha_b^{k,m} \le 2 \int \sfH_{N_k} \de \aalpha_b^{k,m} \\
    & \le 2 \int \sfH_{N_k} \de \aalpha_b^k \le 2\int \sfH_{N_k} \de \aalpha^k = 2\mathcal{E}_{N_k}(P_k) \le 2C < +\infty\,,
\end{split}
\end{equation}
where we have used, in order, the definition of $\ppi^{k,m}$, Lemma \ref{le:double}, \eqref{eq:chain_HNk}, \eqref{eq:ineq_m}, \eqref{eq:pinguino}, and \eqref{eq:cbound}. This fact together with the tightness of $\{(\sfe_t)_\sharp \ppi^{k,m} \,:\, t \in [0,1],\,k \in \N\} \subset \prob(\f{C}_\Theta[X])$ (immediate consequence of the compactness of $\f{C}_\Theta[X]$) entail, by \cite[Theorem 10.4]{ABS21}, that $(\ppi^{k,m})_k$ is tight. Hence, up to passing to a subsequence, $(\ppi^{k,m})_k$ weakly converges as $k \to +\infty$ to some $\ppi^m \in \prob(\rmC([0,1]; (\f{C}[X], \sfd_{\f{C}})))$, but a more accurate statement actually holds. 

Indeed, the lower semicontinuity of $\mathcal{A}_2$ w.r.t.\ weak convergence together with \eqref{eq:A2_bound} implies that $\int \mathcal{A}_2 \de\ppi^m < +\infty$, namely $\ppi^m$ is concentrated on $\AC^2([0,1]; (\f{C}[X], \sfd_{\f{C}}))$. Moreover, recalling that $\pi^{0,N_k}_\sharp \hat\aalpha_b^{k,m} \in \f{H}^2(\nu_0^{k,m}, \nu_1^{k,m})$, we see that $(\sfe_0, \sfe_1)_\sharp \ppi^{k,m} \in \f{H}^2(\nu_0^{k,m}, \nu_1^{k,m})$ as well, so that in conclusion
\[
\ppi^{k,m} \rightharpoonup \ppi^m \quad \textrm{with} \quad \ppi^m \in \prob(\AC^2([0,1]; (\f{C}[X], \sfd_{\f{C}}))) \quad \textrm{and} \quad (\sfe_0, \sfe_1)_\sharp \ppi^m \in \f{H}^2(\nu_0^m,\nu_1^m)\,. 
\]
We are now ready to prove \eqref{eq:final_goal_3}. Let us start observing that for all $\f{y}_i = [x_i,r_i] \in \f{C}_{\Theta,\eps_m}[X]$, $i=0,1$, it holds

\begin{equation}\label{eq:caimano}
\begin{split}
\sfH_{\W\He}(\f{y}_0,\f{y}_1) & = |r_0-r_1|^2 + \frac{r_1}{r_0} r_0 r_1 \sfd^2(x_0,x_1) \\
& \geq \left(1 \wedge \frac{r_1}{r_0}\right)\left(|r_0-r_1|^2 + r_0r_1 \sfd^2(x_0,x_1)\right) \\
& \geq \left(1 \wedge \frac{r_1}{r_0}\right) \left(|r_0-r_1|^2 + 4r_0r_1 \left(\frac{\sfd(x_0,x_1) \wedge \pi}{2}\right)^2 \right) \\
& \geq \left(1 \wedge \frac{r_1}{r_0}\right) \left(|r_0-r_1|^2 + 4r_0r_1 \sin^2\left(\frac{\sfd(x_0,x_1) \wedge \pi}{2}\right) \right) \\
& = \left(1 \wedge \frac{r_1}{r_0}\right) \sfd_{\f{C}}^2(\f{y}_0,\f{y}_1) \,,
\end{split}
\end{equation}

so that
\begin{align*}
\int \sfH_{N_k} \de \hat\aalpha_b^{k,m} & = \int N_k \sum_{i=1}^{N_k} \sfH_{\W\He} \circ \pi^{i-1,i} \de  \hat\aalpha_b^{k,m} \ge \int \left [ N_k \sum_{i=1}^{N_k}\left ( 1 \wedge \frac{\sfr \circ \pi^i}{\sfr \circ \pi^{i-1}}\right ) \sfd_{\f{C}}^2 \circ \pi^{i-1,i} \right ]\de \hat\aalpha_b^{k,m} \\
& = \int \sum_{i=1}^{N_k}\int_{(i-1)/N_k}^{i/N_k} u_k(\f{y},t)|\f{y}'(t)|^2_{\sfd_{\f{C}}} \de t \de \ppi^{k,m}(\f{y})  \\
& = \int \int_0^1 u_k(\f{y},t)|\f{y}'(t)|^2_{\sfd_{\f{C}}} \de t \de \ppi^{k,m}(\f{y}) \, ,
\end{align*}
where we have used the definition of $\sfH_{N_k}$, the fact that $\ppi^{k,m}$ is concentrated on piecewise geodesic curves, and we have introduced, for ease of notation, the function $u_k$ defined as in \eqref{eq:uk} with $\eps=\eps_m$. An application of Proposition \ref{prop:conva2} then gives
\[ 
\begin{split}
\liminf_{k \to + \infty} \int \sfH_{N_k} \de \hat\aalpha_b^{k,m} & \ge \liminf_{k \to +\infty} \int \int_0^1 u_k(\f{y},t)|\f{y}'(t)|^2_{\sfd_{\f{C}}} \de t \de \ppi^{k,m}(\f{y}) \\ 
& \ge \int \mathcal{A}_2 \de \ppi^m \ge \HK^2(\nu_0^m,\nu_1^m)\,, 
\end{split}
\]
where the last inequality comes from \eqref{eq:dynamic_HK}. This is precisely \eqref{eq:final_goal_3}, so that the proof is complete.

\appendix

\section{Measurability of the geodesic selection map}\label{sec:measappendix}

In some of the constructions involved in the proof of Theorem \ref{thm:thmdinicolo}, we need to consider a map associating to a pair of points a geodesic connecting them.
If for every pair of points there exists a unique geodesic connecting them, such map is Borel measurable (see \cite[Theorem 6.9.3]{Bogachev07}). However, in general, this requires to provide a selection of the possibly multi-valued map sending a pair of points to the set of geodesics connecting them. This cannot in general be required to be Borel measurable, so that we have to introduce the notion of universally measurable set. This is only a technical complication, since the push-forward of a Borel measure via a map which is only universally measurable is still a Borel measure, as discussed below.

Let $(X, \sfd)$ be a complete and separable metric space. Recall that $\mathcal{B}(X)$ denotes the Borel sigma-algebra in $(X, \sfd)$. If $\mu \in \meas_+(X)$, we say that a set $A \subset X$ is $\mu$-measurable if there exist $E,F \in \mathcal{B}(X)$ such that $E \subset A$, $A \setminus E \subset F$ and $\mu(F)=0$. The sigma-algebra of $\mu$-measurable subsets of $X$ is denoted by $\mathcal{B}(X)_\mu$.

\begin{definition} Let $(X, \sfd)$ be a complete and separable metric space. The sigma-algebra of \emph{universally measurable} subsets of $X$ is denoted by $\mathcal{B}(X)^*$ and it is defined as
\[ \mathcal{B}(X)^* \coloneqq \bigcap \left \{ \mathcal{B}(X)_\mu \mid \mu \in \prob(X) \right \}.\]  
\end{definition}

Given a measure $\mu \in \meas_+(X)$, it clearly admits a (non-relabeled) extension to $\mathcal{B}(X)^*$: if $A \in \mathcal{B}(X)^*$, it belongs in particular to $\mathcal{B}(X)_\mu$ so that it is enough to set $\mu(A)\coloneqq \mu(E)$, where $E,F \in \mathcal{B}(X)$ are such that $E \subset A$, $A \setminus E \subset F$ and $\mu(F)=0$. Notice that if $A \subset B \in \mathcal{B}(X)^*$ and $\mu(B)=0$, then $A \in \mathcal{B}(X)^*$ and $\mu(A)=0$.

Therefore, if $(X, \sfd_X)$ and $(Y, \sfd_Y)$ are complete and separable metric spaces and $f: X \to Y$ is $\mathcal{B}(X)^*-\mathcal{B}(Y)$ measurable, the push-forward $f_\sharp \mu$ defines an element of $\meas_+(Y)$ for any measure $\mu \in \meas_+(X)$.

This fact can be used to prove that 
\begin{equation}\label{eq:universally}
    \text{any $\mathcal{B}(X)^*-\mathcal{B}(Y)$ measurable map $f:X \to Y$ is also $\mathcal{B}(X)^*-\mathcal{B}(Y)^*$ measurable. }
\end{equation}
Indeed, if $B \subset \mathcal{B}(Y)^*$ and $\mu \in \meas_+(X)$, then $B$ belongs in particular to $\mathcal{B}(Y)_\nu$, where $\nu\coloneqq f_\sharp \mu \in \meas_+(Y)$, so that we can find $E,F \in \mathcal{B}(Y)$ such that $E \subset B$, $B \setminus E \subset F$ and $\mu(f^{-1}(F))=\nu(F)=0$. Thus $f^{-1}(B)= f^{-1}(E) \cup f^{-1}(B \setminus E)$ with $f^{-1}(E) \in \mathcal{B}(X)^*$. Since $f^{-1}(B \setminus E) \subset f^{-1}(F) \in \mathcal{B}(X)^*$ and $\mu(f^{-1}(F))=0$, we deduce that  $f^{-1}(B \setminus E) \in \mathcal{B}(X)^*$, hence $f^{-1}(B) \in \mathcal{B}(X)^*$.

Before stating the following measurable selection theorem, let us recall that we denote by $\geo((X, \sfd)) \subset \rmC([0,1];(X, \sfd))$ the space of (constant speed, length-minimizing) geodesics in $(X, \sfd)$, whenever $(X, \sfd)$ is a complete, separable and geodesic metric space.

\begin{theorem}[Measurable selection of geodesics] \label{th:measgeod}
Let $(X, \sfd)$ be a complete, separable and geodesic metric space. Then there exists a $\mathcal{B}(X \times X)^*$-$\mathcal{B}(\rmC([0,1]; (X, \sfd)))$ measurable map $\Gamma^X: X^2 \to \geo((X, \sfd))$ such that, for every $(x_0, x_1) \in X^2$, $\Gamma(x_0, x_1)$ is a geodesic connecting $x_0$ to $x_1$.
\end{theorem}
\begin{proof} This result is an application of \cite[Theorem 6.9.12]{Bogachev07}. In the notation of such theorem, we are taking $(\Omega, \mathcal{B})\coloneqq (X^2, \mathcal{B}(X \times X))$ which is clearly a measurable space, $X\coloneqq \rmC([0,1]; (X, \sfd))$ which is a Polish space (hence Souslin, see \cite[Definition 6.6.1]{Bogachev07}) and as $A$ the set
\[ 
A\coloneqq \left \{ (x_0, x_1, \gamma) : \gamma \text{ is a geodesic connecting $x_0$ to $x_1$}\right \} \subset X^2 \times \rmC([0,1]; (X, \sfd)). 
\]
It is not difficult to see that $A$ is a closed, hence Borel, subset of $X^2 \times \rmC([0,1]; (X, \sfd))$ and, in particular, it belongs to the class $S\left ( \mathcal{B}(X \times X) \otimes \mathcal{B}(\rmC([0,1]; (X, \sfd))) \right )$ (see \cite[Definition 1.10.1]{Bogachev07}). Thus \cite[Theorem 6.9.12]{Bogachev07} gives the existence of a map $\Gamma^X: \pi_{\Omega}(A)=X^2 \to \rmC([0,1]; (X, \sfd))$ whose graph is contained in $A$ and which is $\sigma(S(\mathcal{B}(X\times X)))-\mathcal{B}(\rmC([0,1]; (X, \sfd)))$ measurable. To conclude, it is enough to observe that $\sigma(S(\mathcal{B}(X\times X))) \subset \mathcal{B}(X \times X)^*$ which is equivalent to say that $\sigma(S(\mathcal{B}(X\times X))) \subset \mathcal{B}(X \times X)_\mu$ for every $\mu \in \meas_+(X\times X)$. This is in turn a consequence of e.g.~\cite[Theorem 1.10.5]{Bogachev07}.
\end{proof}

\section{Minimal radius of cone geodesics}\label{app:radius}

Aim of this appendix is to prove the statements on the minimal radius of geodesics on the metric cone $(\f{C}[X],\sfd_\f{C})$ contained in Remark \ref{rem:geocone}. Thus, given two points $\f{y}_0=[x_0,r_0], \f{y}_1=[x_1,r_1]$, $\f{y}_0 \neq \f{y}_1$, both different from $\f{o}$ with $0 \leq d \coloneqq \sfd(\sfx(\f{y}_0), \sfx(\f{y}_1)) < \pi$, recall that
\[
r^2(t) = (1-t)^2 r_0^2 + t^2 r_1^2 + 2t(1-t) r_0 r_1\cos(d)
\]
and start observing that
\[
\begin{split}
(r^2(t))' & = -2(1-t) r_0^2 + 2t r_1^2 + 2(1-2t) r_0 r_1\cos(d) \\
& = 2t(r_0^2 + r_1^2 - 2r_0 r_1\cos(d)) + 2r_0 r_1\cos(d) - 2r_0^2.
\end{split}
\]
We deduce that
\[
(r^2(t))' \geq 0 \qquad \Longleftrightarrow \qquad t \geq \frac{r_0^2 - r_0 r_1\cos(d)}{r_0^2 + r_1^2 - 2r_0 r_1\cos(d)}
\]
and the only critical point for $r^2(t)$ is thus given by
\[
t^* = \frac{r_0^2 - r_0 r_1\cos(d)}{r_0^2 + r_1^2 - 2r_0 r_1\cos(d)} = 1 - \frac{r_1^2 - r_0 r_1\cos(d)}{r_0^2 + r_1^2 - 2r_0 r_1\cos(d)}.
\]
If $t^* \in (0,1)$, this implies that $r^2(t)$ attains its minimum on $[0,1]$ at $t=t^*$; if instead $t^* \notin (0,1)$, then $r^2(t)$ is monotone on $(0,1)$ and the minimum is therefore attained in one of the interval endpoints. To determine whether $t^* \in (0,1)$ or not, it is sufficient to observe that 
\[
t^* > 0 \quad \Longleftrightarrow \quad \cos(d) < \frac{r_0}{r_1} \qquad \textrm{and} \qquad
t^* < 1 \quad \Longleftrightarrow \quad \cos(d) < \frac{r_1}{r_0}
\]
since $r_0^2 + r_1^2 - 2r_0 r_1\cos(d) = \sfd^2_\f{C}(\f{y}_0,\f{y}_1) > 0$ and $r_0,r_1 > 0$ too. As a consequence,
\begin{itemize}
    \item $t^*<0$ if and only if $\cos(d) \geq \frac{r_0}{r_1}$: in this case, $r^2(t)$ is monotone increasing on $(0,1)$ and the minimum is attained at $t=0$;
    \item $t^*>1$ if and only if $\cos(d) \geq \frac{r_1}{r_0}$: in this case, $r^2(t)$ is monotone decreasing on $(0,1)$ and the minimum is attained at $t=1$.
\end{itemize}
This proves \eqref{eq:tmin} as well as the fact that $r_{\min}(\f{y}_0, \f{y}_1) \coloneqq r_{\f{y}}(t_{\min}(\f{y}_0, \f{y}_1)) = r_0\wedge r_1$ if $\cos(d) \ge \frac{r_0}{r_1}\wedge\frac{r_1}{r_0}$. To conclude the proof of \eqref{eq:rmin}, we are only left to evaluate $r^2(t^*)$. Thus observe that
\[
\begin{split}
r^2(t^*) & = \left(\frac{r_1^2 - r_0 r_1\cos(d)}{r_0^2 + r_1^2 - 2r_0 r_1\cos(d)}\right)^2 r_0^2 + \left(\frac{r_0^2 - r_0 r_1\cos(d)}{r_0^2 + r_1^2 - 2r_0 r_1\cos(d)}\right)^2 r_1^2 \\
& \quad + 2\frac{r_0^2 - r_0 r_1\cos(d)}{r_0^2 + r_1^2 - 2r_0 r_1\cos(d)} \cdot \frac{r_1^2 - r_0 r_1\cos(d)}{r_0^2 + r_1^2 - 2r_0 r_1\cos(d)} r_0 r_1 \cos(d) \\
& = \frac{r_0^2 r_1^2}{(r_0^2 + r_1^2 - 2r_0 r_1\cos(d))^2}R,
\end{split}
\]
where 
\[
\begin{split}
R & = (r_1 - r_0\cos(d))^2 + (r_0 - r_1\cos(d))^2 + 2(r_0 - r_1\cos(d))(r_1 - r_0\cos(d))\cos(d) \\
& = r_0^2 - r_0^2\cos^2(d) + r_1^2 - r_1^2\cos^2(d) - 2r_0 r_1\cos(d) + 2r_0 r_1\cos^3(d) \\
& = \sin^2(d)(r_0^2 + r_1^2 - 2r_0 r_1 \cos(d)) 
\end{split}
\]
Combining this expression with the previous one finally yields
\[
r^2(t^*) = \frac{r_0^2 r_1^2 \sin^2(d)}{r_0^2 + r_1^2 - 2r_0 r_1 \cos(d)} = \frac{r_0^2 r_1^2 \sin^2(d)}{\sfd_\f{C}^2(\f{y}_0,\f{y}_1)}.
\]
As for \eqref{eq:rmin_nicer}, this inequality is trivially true when $\cos(d)\ge \frac{r_0}{r_1} \wedge \frac{r_1}{r_0}$. When $0\le \cos(d)< \frac{r_0}{r_1} \wedge \frac{r_1}{r_0}$, the bound follows since
\[
\frac{2r^2_0r^2_1\sin^2(d)}{r^2_0+r^2_1-2r_0r_1\cos(d)}\ge \frac{2r^2_0r^2_1}{r^2_0+r^2_1}\ge (r_0\wedge r_1)^2.
\]

\bibliographystyle{abbrv}
\bibliography{biblio}
\end{document}